\documentclass{article}
\usepackage{nameref}
\usepackage{pdfpages}
\usepackage{hyperref}
\usepackage{xcolor}
\usepackage[utf8]{inputenc}
\usepackage{amsmath}
\usepackage{amsfonts}
\usepackage{amssymb}
\usepackage{amsthm}
\usepackage{marginnote}
\usepackage{enumerate}
\usepackage{graphicx}
\usepackage[normalem]{ulem}
\usepackage[title, toc, page]{appendix}
\usepackage[]{algorithm2e}
\usepackage{epstopdf}

\newcommand{\defstyle}[1]{\textbf{#1}}
\newcommand{\myprob}[1]{\mathbb P \left[ #1 \right]}

\newcommand{\probCond}[2]{\mathbb P \left[ #1 \left| #2 \right. \right]}

\newcommand{\omid}[1]{\mathbb E \left[ #1 \right]}

\newcommand{\omidCond}[2]{\mathbb E \left[ #1 \left| #2 \right. \right]}

\newcommand{\norm}[1]{\left| #1 \right|}

\newcommand{\nei}[2]{N_{#1} \left(#2\right)}

\newcommand{\identity}[1]{1_{#1}}
\newcommand{\bs}[1]{\boldsymbol{#1}}
\newcommand{\card}[1]{\# #1}

\newcommand{\formir}[1]{#1}
\newcommand{\del}[1]{}
\newcommand{\dels}[1]{}
\newcommand{\rewritten}[1]{}

\newcommand{\gw}{\texttt{GW}}
\newcommand{\egw}{\texttt{EGW}}

\newcommand{\dimMu}[1]{\overline{\mathrm{{udim}}}_M(#1)}
\newcommand{\dimMl}[1]{\underline{\mathrm{{udim}}}_M(#1)}
\newcommand{\dimM}[1]{{\mathrm{{udim}}}_M(#1)}

\newcommand{\dimH}[1]{\mathrm{{udim}}_H(#1)}

\newcommand{\contentH}[2]{\mathcal H^{#1}_{#2}}

\newcommand{\measH}[1]{\mathcal M^{#1}}
\newcommand{\floor}[1]{\left\lfloor #1 \right\rfloor}
\newcommand{\ceil}[1]{\left\lceil #1 \right\rceil}
\newcommand{\growthl}[1]{\underline{\mathrm{growth}}\left(#1 \right)}
\newcommand{\growthu}[1]{\overline{\mathrm{growth}}\left(#1\right)}
\newcommand{\growth}[1]{{\mathrm{growth}}\left(#1\right)}
\newcommand{\decayl}[1]{\underline{\mathrm{decay}}\left(#1\right)}
\newcommand{\decayu}[1]{\overline{\mathrm{decay}}\left(#1\right)}
\newcommand{\decay}[1]{{\mathrm{decay}}\left(#1\right)}

\newcommand{\essinf}{\mathrm{ess\: inf \: }}
\newcommand{\dstar}{{\mathcal D}_*}
\newcommand{\dstarr}{{\mathcal D}_{**}}
\newcommand{\dstarhat}{\widehat{\mathcal D}_{*}}

\newcommand{\oball}[2]{N_{#1}^{\circ}({#2})}
\newcommand{\intensity}[2]{\rho_{#1}(#2)}
\newcommand{\diam}{\mathrm{diam}}
\newcommand{\deltaone}{}
\newcommand{\rooot}{origin}
\newcommand{\rooted}{pointed}
\newcommand{\Rooted}{Pointed}

\newcommand{\densityU}{\overline{d}}
\newcommand{\densityL}{\underline{d}}
\newcommand{\xeq}[1]{{\normalfont (#1)}}

\numberwithin{equation}{section}

\theoremstyle{theorem}
\newtheorem{theorem}{Theorem}[section]
\newtheorem{lemma}[theorem]{Lemma}
\newtheorem{proposition}[theorem]{Proposition}
\newtheorem{corollary}[theorem]{Corollary}
\newtheorem{conjecture}[theorem]{Conjecture}
\newtheorem{problem}[theorem]{Problem}

\theoremstyle{definition}
\newtheorem{definition}[theorem]{Definition}

\newtheorem{example}[theorem]{Example}
\theoremstyle{definition}
\newtheorem{remark}[theorem]{Remark}
\newtheorem{convention}[theorem]{Convention}

\theoremstyle{theorem}

\numberwithin{equation}{section}

\makeatletter
\let\orgdescriptionlabel\descriptionlabel
\renewcommand*{\descriptionlabel}[1]{%
	\let\orglabel\label
	\let\label\@gobble
	\phantomsection
	\edef\@currentlabel{#1}%
	\let\label\orglabel
	\orgdescriptionlabel{#1}%
}

\makeatother

\begin{document}
\title{Unimodular Hausdorff and Minkowski Dimensions}
\author{Fran\c{c}ois Baccelli\footnote{The University of Texas at Austin \formir{and INRIA Paris},
baccelli@math.utexas.edu}, Mir-Omid Haji-Mirsadeghi\footnote{Sharif University of Technology, mirsadeghi@sharif.ir}, and Ali Khezeli\footnote{Tarbiat Modares University, khezeli@modares.ac.ir, 
current address: INRIA Paris, ali.khezeli@inria.fr}}

\maketitle

\begin{abstract}
This work introduces two new notions of dimension, namely the
\textit{unimodular Minkowski and Hausdorff dimensions},
which are inspired from the classical analogous notions.
These dimensions are defined for \textit{unimodular discrete spaces},
introduced in this work, which provide a common generalization to stationary
point processes under their Palm version and unimodular random rooted graphs. 
The use of unimodularity in the definitions of dimension is novel.
Also, a toolbox of results is presented for the analysis of these dimensions.
In particular, analogues of Billingsley's lemma and Frostman's lemma are presented.
These last lemmas are instrumental in deriving upper bounds on dimensions,
whereas lower bounds are obtained from specific coverings. 
The notions of unimodular Hausdorff {\formir{size, which is a discrete analogue
of the Hausdorff measure,}} and unimodular dimension
function are also introduced. This toolbox \formir{allows one}
to connect the unimodular dimensions to other notions such as \formir{volume} growth rate, 
\formir{discrete dimension and scaling limits}.
It is also used to analyze the dimensions of a set of examples pertaining to point processes,
branching processes, random graphs, random walks, and self-similar discrete random spaces. 
Further results of independent interest are also presented, like a version of the max-flow min-cut 
theorem for unimodular one-ended trees and a weak form of pointwise ergodic theorems
for all unimodular discrete spaces.
\end{abstract}



\section{Introduction}


Infinite discrete random structures are ubiquitous: random graphs, branching processes,
point processes, graphs or zeros of discrete random walks, {\formir{discrete or continuum}} percolation,
to name a few. 
\formir{The large scale {and macroscopic} properties of such spaces have been thoroughly discussed
in the literature. In particular, various notions of dimension
{have been} proposed; e.g., the \textit{mass dimension} and the \textit{discrete (Hausdorff) dimension}
defined by Barlow and Taylor~{\cite{BaTa89}} for subsets of $\mathbb Z^d$.}

The main novelty of the present paper is the definition of new notions of
dimension for a class of \formir{discrete} structures that, heuristically, enjoy a form of statistical homogeneity.
The mathematical framework proposed to handle such structures is that of 
\textit{unimodular (random) discrete spaces}, where unimodularity is defined here by
a version of the mass transport principle.
This framework unifies unimodular random graphs and networks, stationary point processes 
(under their Palm version) and point-stationary point processes.
It does not require more than a metric; for instance, no edges or no underlying
Euclidean spaces are needed. The statistical homogeneity of such spaces
has been used to define localized versions of global notions such as intensity.
The main novelty of the present paper is the use of this homogeneity to define
the notions of \textit{unimodular Minkowski and Hausdorff dimensions}, which are
inspired by the analogous classical notions.
\formir{The definitions are obtained naturally from the classical setting by
replacing the infinite sums pertaining to infinite coverings by the expectation
of certain random variables at the origin (which is a distinguished point),
and also by considering large balls instead of small balls.} These definitions are local but
capture \formir{macroscopic (large scale)} properties of the space. 

The definitions are complemented by a toolbox for the analysis of unimodular dimensions.
Several analogues of the important results known about
the classical Hausdorff and Minkowski dimensions are established, like for instance
the comparison of the unimodular Minkowski and Hausdorff dimensions as well as
unimodular versions of Billingsley's lemma and Frostman's lemma. These lemmas allow one to
connect the dimension to the (polynomial) \formir{volume} growth rate of the space,
\formir{which is also called \textit{mass dimension} or \textit{fractal dimension} in the literature}.
While many ideas in this toolbox are imported from the continuum setting, their adaptation is 
nontrivial and there is no automatic way to import results from the continuum to the discrete setting.
For some results, the statements fundamentally differ from their continuum analog; 
e.g., the statement of Billingsley's lemma.

\formir{These notions of dimension are complemented by further definitions which can be used for a finer study of dimension. An analogue of the Hausdorff measure is defined, which is called the unimodular Hausdorff size
	here.}
This can be used to compare sets with the same dimension.
The notion of unimodular dimension function is also defined for a finer quantification of the dimension.
\formir{Such notions are new for discrete spaces to the best of the authors' knowledge.
Another new notion introduced in the present paper is that of regularity for unimodular spaces,
which is the equality of the unimodular Minkowski and Hausdorff dimensions.
Similar notions of regularity exist in the continuum setting
(\formir{see e.g., the definition of fractals in~\cite{bookBiPe17}}) and for subsets
of $\mathbb Z^d$ \cite{BaTa92}.}


The paper also contains new mathematical results of independent interest. 
A weak version of Birkhoff's pointwise ergodic theorem is stated for all unimodular discrete spaces.
A unimodular version of the max-flow min-cut theorem is also proved
for unimodular one-ended trees, which is used in the proof of the
unimodular Frostman lemma. Also, for unimodular one-ended trees, a relation between
the \formir{volume} growth rate and the height of the root is established as explained below.

The framework is used to derive concrete results on the dimension of 
several instances of unimodular random discrete metric spaces.
This is done for the zeros and the graph of discrete random walks, sets defined by digit restriction,
trees obtained from branching processes and drainage network models,
etc. Some general results are obtained
for all unimodular trees. For instance, a general relation is established between the
unimodular dimensions of a unimodular one-ended tree and the tail of the distribution
of the height of the root. The dimensions of some unimodular discrete random \formir{self-similar} sets  
are also discussed. The latter are defined in this paper as
unimodular discrete analogues of self similar sets
such as the Koch snowflake, the Sierpinski triangle, etc.

\formir{This framework opens several further research directions. Firstly,
it might be useful for the study of some discrete examples which are of interest
in mathematical physics. Many examples in this domain enjoy some kind of
homogeneity and give rise to unimodular spaces, directly or indirectly; e.g., percolation
clusters and self-avoiding random walks.
A few such examples are studied in detail in this work and in the preprint~\cite{III}.
{One might expect that} in these examples, the values of unimodular dimensions match the
conjectures or results pertaining to other notions of dimension that are applicable.
Secondly, the definitions and many of the results are valid for exponential (or other) gauge functions.
The proposed framework might hence have applications in group theory (or other areas), where most interesting
examples have super-polynomial growth.
A third important line of thoughts is the connections of unimodular dimensions to other notions of dimension.
Some first connections are discussed in Subsection~\ref{subsec:connections}.
The preprint \cite{III} discusses ongoing research
on these questions as well as further developments of these notions of dimensions.}

\subsection{Summary of the Main Definitions and Results}
\label{subsec:intro-intro}

Recall that the ordinary Minkowski dimension of a compact metric space $X$
is defined using the minimum number of balls of radius $\epsilon$ needed to cover $X$.
Now, consider a (unimodular) discrete space $\bs D$ (it is useful to have in mind the example 
$\bs D=\mathbb Z^k$ to see how the definitions work).
It is convenient to consider coverings of $\bs D$ by balls of equal but large radius. 
Of course, if $\bs D$ is unbounded, then an infinite number of balls is needed to cover $\bs D$. 
So one needs another measure to assess how many balls are used in a covering. 
Let $S\subseteq \bs D$ be the set of centers of the balls in the covering.
The idea pursued in this paper is that if $\bs D$ is unimodular, then the \textit{intensity} of $S$
is a measure of \textit{the average number of points of $S$ per points of $\bs D$}
($S$ should be \textit{equivariant} for the intensity to be defined, as discussed later).
This gives rise to the definition of the unimodular Minkowski dimension naturally.

The idea behind the definition of the unimodular Hausdorff dimension is similar.
Recall that the $\alpha$-dimensional \textit{Hausdorff content} of a compact metric space $X$ is
defined by considering the infimum of $\sum_i R_i^{\alpha}$, where the $R_i$'s are the radii of a
sequence of balls that cover $X$. Also, it is convenient to enforce an upper bound on the radii.
Now, consider a unimodular discrete space $\bs D$ and a covering of $\bs D$ by balls which
may have different radii. Let $R(v)$ be the radius of the ball centered at $v$.
It is convenient to consider a lower bound on the radii, say $R(\cdot)\geq 1$.
Again, if $\bs D$ is unbounded, then $\sum_v R(v)^{\alpha}$ is always infinite.
The idea is to leverage the unimodularity of $\bs D$ and to consider
\textit{the average of the values $R(\cdot)^{\alpha}$ per point}
as a replacement of the sum. Under the unimodularity assumption, this can be defined
by $\omid{R(\bs o)^{\alpha}}$, where $\bs o$ stands for the distinguished point of $\bs D$
(called the origin) and where, by convention, $R(\bs o)$ is zero if there is no ball centered at $\bs o$.
This is used to define the \emph{unimodular Hausdorff dimension} of $\bs D$ in a natural way.

The \formir{volume} growth rate of the space is the polynomial growth rate of $\card{N_r(\bs o)}$,
where $N_r(\bs o)$ represents the closed ball of radius $r$ centered at the \rooot{}
and $\card{N_r(\bs o)}$ is the number of points in this ball.
It is shown that the upper and lower \formir{volume} growth rates of $\card{N_r(\bs o)}$ (i.e., limsup and liminf
of $ {\log(\card{N_r(\bs o)})}/{\log r}$ as $r\rightarrow \infty$) provide upper and lower bound
for the unimodular Hausdorff dimension, respectively. This is a discrete analogue of Billingsley's
lemma (see e.g., \cite{bookBiPe17}). 
A discrete analogue of the \textit{mass distribution principle} is also provided, which is
useful to derive upper bounds on the unimodular Hausdorff dimension. In the Euclidean case 
(i.e., for point-stationary point processes equipped with the Euclidean metric), it is shown
that the unimodular Minkowski dimension is bounded from above by the polynomial decay rate
of $\omid{1/ \card{N_n(\bs o)}}$. Weighted versions of these inequalities, where a weight
is assigned to each point, are also presented. As a corollary, a weak form of Birkhoff's
pointwise ergodic theorem is established for all unimodular discrete spaces.
These results are very useful for calculating the
unimodular dimensions in many examples. {\formir{An important result is an
analogue of Frostman's lemma}}. Roughly speaking, this lemma states that the mass distribution
principle is sharp if the weights are chosen appropriately.
This lemma is a powerful tool to study the unimodular Hausdorff dimension.
In the Euclidean case, another proof of Frostman's lemma is provided using a version
of the max-flow min-cut theorem for unimodular one-ended trees, which is of independent interest.

Depending on whether one defines the unimodular Minkowski dimension as the
decay rate or the growth rate of the optimal intensity of the coverings by balls of
radius $r$, one gets positive or negative dimensions. The present paper adopts the
convention of positive dimensions for the definitions of both the unimodular Minkowski
and Hausdorff dimensions, despite some mathematical arguments in favor of negative
dimensions. Further discussion on the matter is provided in Subsection~\ref{subsec:negative}.

\subsection{Organization of the Material}
\label{subsec:summary}

Section~\ref{sec:unimodular} defines unimodular discrete spaces and \textit{equivariant processes},
which are needed throughout. Section~\ref{sec:dimension} presents the definitions of the
unimodular Minkowski and Hausdorff dimensions and the unimodular Hausdorff
{\formir{size}}. It also provides some basic properties of these unimodular dimensions as part of the toolbox for
the analysis of unimodular dimensions. Various examples are discussed in Section~\ref{sec:examples}.
These examples are used throughout the paper. 
Section \ref{sec:volumeGrowth} is focused on the connections with \formir{volume} growth rates and contains
the statements and proofs of the unimodular Billingsey lemma and of the mass distribution principle.
The unimodulat Frostman lemma is discussed in Section \ref{sec:frostman}.
Section \ref{sec:examples2} completes the analysis of the examples discussed in Section \ref{sec:examples}
and also discusses new examples for further illustration of the results.
\formir{Section~\ref{s:MiTo} discusses further topics on the matter. 
This includes a discussion of the connections to earlier notions of dimensions for discrete sets,
in particular those proposed by Barlow and Taylor in \cite{BaTa89,BaTa92}, as well as a discussion
on negative dimensions. A collection of conjectures and open problems are also listed in this section.}


\section{Unimodular Discrete Spaces}
\label{sec:unimodular}

The main objective of this section is the definition of unimodular discrete spaces
as a common generalization of unimodular graphs, {Palm probabilities} and point-stationary point processes. 
{If the reader is familiar with unimodular random graphs, he or she can restrict attention to the case of unimodular graphs and jump to Subsection~\ref{subsec:process} at first reading.}

\subsection{Notation \formir{and Definitions}}
\label{subsec:notations}
The following notation will be used throughout. The set of nonnegative real (resp. integer) numbers 
is denoted by $\mathbb R^{\geq 0}$ (resp. $\mathbb Z^{\geq 0}$). The minimum and maximum binary operators 
are denoted by $\wedge$ and $\vee$ respectively. The number of elements in a set $A$ is denoted by $\card{A}$,
which is a number in $[0,\infty]$. 
If $P(x)$ is a property about $x$, the indicator $\identity{\{P(x)\}}$ is equal to 1 if $P(x)$ is true and 0 otherwise.

Discrete metric spaces (discussed in details in Subsection~\ref{subsec:D_*}) are denoted by $D$, $D'$, etc. 
Graphs are an important class of discrete metric spaces. So the symbols and notations are mostly borrowed
from graph theory.

For $r>0$, ${N_r(v):=}N_r(D,v)$ denotes the closed $r$-neighborhood of $v\in D$; i.e.,
the set of points of $D$ with distance less than or equal to $r$ from $v$.
An exception is made for $r=0$ (Subsection~\ref{subsec:HausdorffDim}),
where $N_0(v):=\emptyset$.  The diameter of a subset $A$ is denoted by $\diam(A)$.
For a function $f:[1,\infty)\rightarrow \mathbb R^{\geq 0}$, 
the \emph{polynomial growth rates} and \emph{polynomial decay rates} are defined by the following formulas: 
\begin{eqnarray*}
	\growthl{f} &:=& -\decayu{f}:= \liminf_{r\rightarrow\infty} {\log f(r)}/{\log r},\\
	\growthu{f} &:=& -\decayl{f}:= \limsup_{r\rightarrow\infty} {\log f(r)}/{\log r},\\
	\growth{f} &:=& -\decay{f}:= \lim_{r\rightarrow\infty} {\log f(r)}/{\log r}.\\
\end{eqnarray*}
  
\formir{
\begin{definition}
        \label{def:bias}
        Let $\mu$ be a probability measure on a measurable space $X$ and $w:X\rightarrow\mathbb R^{\geq 0}$
        be a measurable function. Assume $0<c:=\int_{X}w(x) d\mu(x)<\infty$. By \defstyle{biasing $\mu$ by $w$}
        we mean the probability measure $\nu$ on $X$ defined by
        \[
        \nu(A):= \frac 1 c \int_A w(x)d\mu(x).
        \]
\end{definition}
}

\subsection{The Space of Pointed Discrete Spaces}
\label{subsec:D_*}

Throughout the  paper, the metric on any metric space is  denoted by $d$, 
except when explicitly mentioned. In this paper, 
it is always assumed that the discrete metric space\formir{s under study are}
\defstyle{boundedly finite}; i.e., every  set included in a ball of finite radius in $D$
is finite \formir{(note that this is stronger than being locally-finite)}.
\formir{This implies that the metric space is indeed discrete; i.e., every point is isolated.}
The term \defstyle{discrete space} will always refer to boundedly finite discrete metric space.
A \defstyle{\rooted{} set} (or a \textit{rooted set}) is a pair $(D,o)$, where $D$ is a set and $o$ 
a distinguished point of $D$ called the \defstyle{origin} (or the \textit{root}) of $D$. 
Similarly, a \defstyle{doubly-\rooted{}} set is a triple $(D,o_1,o_2)$,
where $o_1$ and $o_2$ are two distinguished points of $D$.

Let $\Xi$ be a complete separable metric space called the \defstyle{mark space}.
A \defstyle{marked pointed discrete space} is a tuple $(D,o;m)$,
where $(D,o)$ is a pointed discrete space and $m$ is a function $m :D\times D\rightarrow\Xi$.
The mark of a single point $x$ may also be defined by $m(x):=m(x,x)$,
where the same symbol $m$ is used for simplicity.
An \defstyle{isomorphism} {(or \textit{rooted isomorphism})} between two such spaces
$(D,o;m)$ and $(D',o';m')$ is an isometry $\rho:D\rightarrow D'$
such that $\rho(o)=o'$ and $m'(\rho(u),\rho(v))=m(u,v)$ for all $u,v\in D$.
An isomorphism between doubly-\rooted{} marked discrete spaces is defined similarly.
An isomorphism from a space to itself is called an \defstyle{automorphism}.

Most of the examples of discrete spaces in this work are graphs or discrete subsets
of the Euclidean space. More precisely, connected and locally-finite simple graphs
equipped with the graph-distance metric \cite{processes} are instances of discrete spaces.
Similarly, \textit{networks}; i.e., graphs equipped with marks on the edges \cite{processes},
are instances of marked discrete spaces.

Let $\mathcal D_*$ (resp. $\mathcal D_{**}$) be the set of equivalence classes of 
pointed (resp. doubly-pointed) discrete spaces under isomorphism.
Let $\mathcal D'_*$ and $\mathcal D'_{**}$ be
defined similarly for marked discrete spaces with mark space $\Xi$ (which is usually {given}).
The equivalence class containing $(D,o)$, $(D,o;m)$ etc., is denoted by brackets $[D,o]$, $[D,o;m]$, etc. 

Every element of $\mathcal D_*$ can be regarded as a \textit{boundedly-compact}
measured metric space (where the measure is the counting measure). Therefore, the generalization
of the \textit{Gromov-Hausdorff-Prokhorov} metric in~\cite{Kh19ghp} defines a metric on $\dstar$.
By using the results of~\cite{Kh19ghp}, one can show that $\dstar$ is a Borel subset of some
complete separable metric space. The proof of this result is {given in Appendix~\ref{ap:metric}}.
Similarly, one can show that $\dstarr, \dstar'$ and $\dstarr'$ are Borel subsets of
some complete separable metric spaces.
This enables one to define \textit{random pointed discrete spaces}, etc.,
which are discussed in the next subsection.

\subsection{Random \Rooted{} Discrete Spaces}
\label{subsec:random}

\begin{definition}
\label{def:randomRooted}
A \defstyle{random \rooted{} discrete space} is a random element in $\dstar$ and
is denoted by bold symbols $[\bs D,\bs o]$. Here, $\bs D$ and $\bs o$ represent
the discrete space and the \rooot{} respectively.
\end{definition}

\formir{In this paper, the probability space is not referred to
explicitly\footnote{\formir{Indeed, one may regard $\dstar$, equipped with a probability measure, as the canonical probability space. The last paragraph of Subsection~\ref{subsec:D_*} ensures that this is a standard probability space, and hence, the classical tools of probability theory are available. One may also define a random pointed discrete space as a measurable function from some standard probability space to $\dstar$.}}. The main reason is that the notions of dimension, to be defined, depend only on the distribution of the random object. Also, extra randomness will be considered frequently and it is easier to forget about the probability space. By an abuse of notation,
the symbols $\mathbb P$ and $\mathbb E$ are used for all random objects, possibly
living in different spaces. They refer to probability and expectation with respect
to the random object under consideration.}


Note that the whole symbol $[\bs D, \bs o]$ represents one random object, which is a random equivalence
class of \rooted{} discrete spaces. Therefore, any formula using $\bs D$ and $\bs o$ should be well
defined for equivalence classes; i.e., it should be invariant under \rooted{} isomorphisms. 

The following convention is helpful throughout.
\begin{convention}
In this paper, bold symbols are usually used in the random case or when extra randomness is used.
For example, $[D,o]$ refers to a deterministic element of $\dstar$ and $[\bs D, \bs o]$ refers to a random \rooted{} discrete space. 
\end{convention}

Note 
that the distribution of a random \rooted{} network $[\bs D, \bs o]$ is a probability
measure on $\dstar$ defined by $\mu(A):=\myprob{[\bs D, \bs o]\in A}$ for events $A\subseteq \dstar$.

\begin{definition}
	\label{def:randomSpaceMarked}
	A \defstyle{random \rooted{} marked discrete space} is a random element in $\mathcal D'_*$ and 
	is denoted by bold symbols $[\bs D,\bs o; \bs m]$. Here, $\bs D$, $\bs o$ and $\bs m$ represent the
        discrete space, the \rooot{} and the mark function respectively.
\end{definition}

{Most examples in this work are {either} random rooted graphs {(or networks)} \cite{processes} {or} point processes (i.e., random discrete subset
of $\mathbb R^k$) and marked point processes that contain 0, where 0 is considered as the \rooot{}. Other examples are also studied by considering different metrics on such objects. 
}

\subsection{Unimodular Discrete Spaces}
\label{subsec:unimodular}

{Once the notion of random pointed discrete space is defined, the definition of unimodularity is a straight variant of~\cite{processes}. In what follows, the notation is similarly to~\cite{eft}.}
Here, the symbol $g[D,o,v]$ is used as a short form of $g([D,o,v])$.
Similarly, brackets $[\cdot]$ are used as a short form of $([\cdot])$.
\begin{definition} 
	\label{def:unimodular}
	A \defstyle{unimodular discrete space} is a random \rooted{} discrete space, namely $[\bs D, \bs o]$, such that for 
	all measurable functions $g:\dstarr\rightarrow \mathbb R^{\geq 0}$,
	\begin{equation} 
		\label{eq:unimodular}
		\omid{\sum_{v\in \bs D}g[\bs D, \bs o, v]}=\omid{\sum_{v\in \bs D}g[\bs D, v, \bs o]}.
	\end{equation}
	Similarly, a \defstyle{unimodular marked discrete space} is a random \rooted{} marked discrete space
        $[\bs D, \bs o; \bs m]$ such that for all measurable functions $g:\mathcal D'_{**}\to\mathbb R^{\geq 0}$,
	\begin{equation} 
		\label{eq:unimodularMarked}
		\omid{\sum_{v\in \bs D}g[\bs D, \bs o, v;\bs m]}=\omid{\sum_{v\in \bs D}g[\bs D, v, \bs o; \bs m]}.
	\end{equation}
	Note that the expectations may be finite or infinite. 
	
\end{definition}

When there is no ambiguity, the term $g[D,o,v]$ is also denoted by $g_D(o,v)$ {or simply $g(o,v)$}.
The sum in the left (respectively right) side of~\eqref{eq:unimodular} is called the 
\defstyle{outgoing mass from $\bs o$} (respectively 
\defstyle{incoming mass into $\bs o$}) and is denoted by 
$g^+(\bs o)$ (respectively $g^-(\bs o)$). The same notation can be used for the
terms in~\eqref{eq:unimodularMarked}.
So \eqref{eq:unimodular} and~\eqref{eq:unimodularMarked} can be summarized by
\[
	\omid{g^+(\bs o)}=\omid{g^-(\bs o)}.
\]
These equations are called the \defstyle{mass transport principle} in the literature.
The reader will find further discussion on the mass transport principle and unimodularity
in~\cite{processes} and the examples therein.

As a basic example, every finite metric space $D$, equipped with a random root
$\bs o\in D$ chosen uniformly, is unimodular. Also, the lattices of the Euclidean space
rooted at 0; e.g., $[\mathbb Z^k,0]$ and $[\delta\mathbb Z^k,0]$, are unimodular.
In addition, unimodularity is preserved under weak convergence, as observed in~\cite{BeSc01}
for unimodular graphs.

The following two examples show that unimodular discrete spaces unify unimodular graphs and point-stationary point processes. Most of the examples in this work are of these types.

\begin{example}[Unimodular Random Graphs]
In the case of random rooted graphs and networks, the concept of unimodularity
in Definition~\ref{def:unimodular} coincides with that of~\cite{processes} 
(see also {Remark~\ref{rem:topology}} regarding the topologies). Therefore, unimodular
random graphs and networks are special cases of unimodular (marked) discrete spaces.
\end{example}

\begin{example}[Point-Stationary Point Processes]
\label{ex:point-stationary}
\defstyle{Point-stationarity} is defined for point processes $\Phi$ in $\mathbb R^k$ such
that $0\in\Phi$ a.s. (see e.g., \cite{ThLa09}). This definition is equivalent to~\eqref{eq:unimodular},
except that $g$ is required to be invariant under translations only (and not under all isometries).
This implies that $[\Phi,0]$ is unimodular. In addition, by considering the mark $\bs m(x,y):=y-x$
on pairs of points of $\Phi$, point-stationarity of $\Phi$ will be equivalent to the unimodularity
of $[\Phi,0;\bs m]$ (see also {Remark~\ref{rem:topology}} regarding the topologies).
Note also that $\Phi$ can be recovered from $[\Phi,0;\bs m]$.
\\
For example, if $\Phi$ is a stationary point process in $\mathbb R^k$
(i.e., its distribution is invariant under all translations), with finite intensity
(i.e., a finite expected number of points in the unit cube),
then the \textit{Palm version} of $\Phi$ is a point-stationary point process,
where the latter is heuristically obtained by conditioning $\Phi$
to contain the origin (see {e.g., Section~13 of~\cite{bookDaVe03II}} for the precise definition).
Also, if $(X_n)_{n\in\mathbb Z}$ is a stochastic process in $\mathbb R^k$ with stationary
increments such that $X_0=0$ and $X_i\neq X_j$ a.s. for every $i\neq j$,
then the image of this random walk is a point-stationary point process.
\end{example}

\subsection{Equivariant Process on a Unimodular Discrete Space}
\label{subsec:process}

In many cases in this paper, an unmarked unimodular discrete space $[\bs D, \bs o]$
is given and various ways of assigning marks to $\bs D$ are considered.
Intuitively, an \textit{equivariant process} on {$\bs D$} is an assignment
of (random) marks to $\bs D$ such that the new marked space is unimodular. 
Formally, it is

\begin{center}
	\begin{minipage}{.85\textwidth}
\emph{a unimodular marked discrete space $[\bs D', \bs o'; \bs m]$ such that the space
$[\bs D', \bs o']$, obtained by forgetting the marks, has the same distribution as $[\bs D, \bs o]$.}
	\end{minipage}
\end{center}
In this paper, it is more convenient to work with a disintegrated form of this heuristic, defined below.
It can be proved that the two notions are equivalent, but the proof is skipped for brevity
(this claim is similar to \textit{invariant disintegration} for group actions). \formir{The easy part of the claim is Lemma~\ref{lem:equivProcess} below.
	For the other direction, see} {Proposition~\ref{prop:equivProcessConverse}}.

In the following, the mark space $\Xi$ is fixed as in Subsection~\ref{subsec:D_*}.
	
\begin{definition}
\label{def:marking}
Let $D$ be a deterministic discrete space which is boundedly-finite. A \defstyle{marking} of $D$
is a function from $D\times D$ to $\Xi$; i.e., an element of $\Xi^{D\times D}$.
A \defstyle{random marking} of $D$ is a random element of $\Xi^{D\times D}$.
\end{definition}

\begin{definition} 
\label{def:equivProcess}
An \defstyle{equivariant process} $\bs Z$ with values in $\Xi$ is 
a map that assigns to every deterministic discrete space $D$ a random marking $\bs Z_D$
of $D$ satisfying the following properties:
\begin{enumerate}[(i)]
\item \label{def:equivProcess:1} $\bs Z$ is compatible with {isometries} in the sense that
for every {isometry} $\rho:D_1\rightarrow D_2$, the random marking
$\bs Z_{D_1} \circ \rho^{-1}$ of ${D_2}$ has the same distribution {as} $\bs Z_{D_2}$.
\item \label{def:equivProcess:2}
For every measurable subset $A\subseteq D'_{*}$, the following function on $\mathcal D_*$ is measurable:
\[ [D,o]\mapsto \myprob{[D,o;\bs Z_D] \in A}. \]
\end{enumerate}
In addition, given a unimodular discrete space  $[\bs D, \bs o]$, such a map is also
called an \defstyle{equivariant process {on $\bs D$}}. In this case, one can also let
$\bs Z_{(\cdot)}$ be undefined for a class of discrete spaces, as long as it is defined
for almost all realizations of $\bs D$. It is important that extra randomness be allowed here.
\end{definition}

\begin{convention}
If $D$ is clear from the context, $\bs Z_{D}(\cdot)$ is also denoted by $\bs Z(\cdot)$ for simplicity.
\end{convention}

Note that in the above definition, $D$ is deterministic and is not an equivalence 
class of discrete spaces. However, for an equivariant process on $[\bs D, \bs o]$, 
one can define $[\bs D, \bs o; \bs Z_{\bs D}]$ as a random \rooted{} marked discrete
space with distribution $\mathcal Q$ \formir{(on $\mathcal D'_*$)} defined by
\begin{equation}
\label{eq:equivProcessDistribution}
\mathcal Q(A):= \int \int \identity{A}[D,o;m]{\mathrm d}\mathcal P_D(m) d\mu([D,o]),\quad \formir{\forall A\subseteq \mathcal D'_*,}
\end{equation}
where $\mathcal P_D$ is the distribution of $\bs Z_D$ (for every $D$) and $\mu$ is the
distribution of $[\bs D, \bs o]$ (note that only the distribution of $\bs Z_D$ is important
\formir{here and it doesn't matter which probability space is used for $\bs Z_D$}). {It can be seen that $\mathcal Q(A)$ is indeed well defined and is a probability measure on $\mathcal D'_*$.} \formir{As mentioned before, the probabilities and expectations to be used for $\bs Z_D$ and $[\bs D, \bs o; \bs Z_{\bs D}]$ will be denoted by the same symbols $\mathbb P$ and $\mathbb E$; e.g., $\myprob{\bs Z_D\in B}$, $\myprob{[D,o;\bs Z_D]\in A}$ and $\omid{f[\bs D, \bs o; \bs Z_{\bs D}]}$. The symbols $\mathcal P_D$ and $\mathcal Q$ will not be used in what follows.}

The following basic examples help to illustrate the definition.
\begin{example}
\label{ex:equiv-basic}
By choosing the marks of points (or pairs of points) in an i.i.d. manner, one obtains an equivariant process.
Also, the following periodic marking of $\mathbb Z$ is an equivariant process on $\mathbb Z$: 
Choose $\bs U\in\{0,1,\ldots,n-1\}$ uniformly at random and let $\bs Z_{\mathbb Z}(x):=1$
if $x\in n\mathbb Z+\bs U$ and $\bs Z_{\mathbb Z}(x):=0$ otherwise. Moreover, given a measurable
function $z:\dstarr\to \Xi$, one can define $\bs Z_D(u,v):=z[D,u,v]$, which will be called a
\textit{deterministic process} here.
\end{example}

\begin{lemma}
\label{lem:equivProcess}
Let $[\bs D, \bs o]$ be a unimodular discrete space. If $\bs Z$ is an
equivariant process on $\bs D$, then $[\bs D, \bs o; \bs Z_{\bs D}]$ is also unimodular.
\end{lemma}
The proof is straightforward and {is given in Appendix~\ref{ap:lemmas}}. The converse of this claim also holds
({Proposition~\ref{prop:equivProcessConverse}}). It is important here to assume that the distribution of $\bs Z_D$ does not depend on the \rooot{}
(as in Definition~\ref{def:equivProcess}).

\begin{remark}
\label{rem:equivProcessExtraMarks}
One can easily extend the definition of equivariant processes to allow the base space to be marked. 
Therefore, for point-stationary point processes, one can replace condition~\eqref{def:equivProcess:1}
by invariance under translations only (see Example~\ref{ex:point-stationary}).
In particular, every stationary stochastic process on $\mathbb Z^k$ defines an
equivariant process on $\mathbb Z^k$.  
\end{remark}

\begin{definition}
\label{def:subset}
An \defstyle{equivariant subset} $\bs S$ is the set of points with mark 1 in some
$\{0,1\}$-valued equivariant process. In addition, if $[\bs D, \bs o]$ is a unimodular discrete space,
then the \defstyle{intensity} of $\bs S$ in
$\bs D$ is defined by $\intensity{\bs D}{\bs S}:=\myprob{\bs o\in \bs S_{\bs D}}$. 
\end{definition} 
For example, $\bs S_D:=\{v\in D: \card{N_1(v)}=4\}$ defines an equivariant subset. 
Also, let $D=\mathbb Z$ and $\bs S_D$ be the set of even numbers with probability $p$ and 
the set of odd numbers with probability $1-p$. Then, $\bs S$ is an equivariant subset
of $\mathbb Z$ if and only if $p=\frac 1 2$ 
(notice Condition~\eqref{def:equivProcess:1} of Definition~\ref{def:equivProcess}).

\begin{lemma}
\label{lem:happensAtRoot}
Let $[\bs D, \bs o]$ be a unimodular discrete space and $\bs S$ an equivariant subset. Then
$\bs S_{\bs D}\neq \emptyset$ with positive probability if and only if it has positive intensity.
Equivalently, $\bs S_{\bs D}=\bs D$ a.s. if and only if $\intensity{\bs D}{\bs S}=1$.
\end{lemma}

\begin{proof}
	The claim is implied by the mass transport principle~\eqref{eq:unimodularMarked}
	for the function  $g[D,u,v;S]:=\identity{\{v\in S\}}$. The details are left to the reader.
\end{proof}

{The above lemma is a generalization of similar results in~\cite{eft} and~\cite{processes}.}

\subsection{Notes and Bibliographical Comments}


The mass transport principle was introduced in~\cite{Ha97}. The concept of unimodular graphs was
first defined for deterministic transitive graphs in~\cite{BLPS} and generalized to random rooted
graphs and networks in~\cite{processes}. 

Unimodular graphs have many analogies and connections to (Palm versions of) stationary
point processes and point-stationary point processes, as discussed in
Example~9.5 of~\cite{processes} and also in~\cite{eft} and~\cite{shift-coupling}. 
As already explained, the framework of unimodular discrete spaces introduced in this section
can be regarded as a common generalization of these concepts.

Special cases of the notion of equivariant processes have been considered in the literature.
The first formulation in Subsection~\ref{subsec:process} is considered in~\cite{processes}
for unimodular graphs. \textit{Factors of IID}~\cite{Ly17} are special cases of equivariant 
processes where the marks of the points are obtained from i.i.d. marks
(Example~\ref{ex:equiv-basic}) in an equivariant way. \textit{Covariant subsets}
and \textit{covariant partitions} of unimodular graphs are defined similarly in~\cite{eft}, 
but no extra randomness is allowed therein. In the case of stationary (marked) point processes,
the first formulation of Subsection~\ref{subsec:process} is used in the literature.
However, the authors believe that the general formulation of Definition~\ref{def:equivProcess}
is new even in those special cases.

\section{The Unimodular Minkowski and Hausdorff Dimensions}
\label{sec:dimension} 
This section presents {the} new notions of dimension for unimodular discrete spaces.
As mentioned in the introduction, the statistical homogeneity of unimodular discrete spaces
is used to define discrete \formir{analogues} of the Minkowski and Hausdorff dimensions.
Also, basic properties of these definitions are discussed.

\subsection{The Unimodular Minkowski Dimension}
\label{subsec:Miknowski}

\begin{definition}
\label{def:r-covering}
Let $[\bs D, \bs o]$ be a unimodular discrete space and {$r\geq 0$}. An \defstyle{equivariant $r$-covering}
$\bs R$ of $\bs D$ is an equivariant subset of $\bs D$ 
such that the set of balls $\{N_r(v):v\in \bs R_{\bs D}\}$ cover $\bs D$ almost surely.
Here, the same symbol $\bs R$ is used for the following equivariant process (Definition~\ref{def:equivProcess}): 
\[
	\bs R(v):= {\bs R_{\bs D}(v):=} \left\{
		\begin{array}{ll}
			r, & \text{there is a ball centered at $v$ in the covering},\\
			0, & \text{otherwise},
		\end{array}
	\right.
\]
for $v\in \bs D$. \formir{Note that an equivariant covering may use extra
randomness and is not necessarily a function of $\bs D$. This is essential in the following definition.}

Let $\mathcal C_r$ be the set of all equivariant $r$-coverings. Define 
\begin{equation}
\label{eq:I_r}
\lambda_r:=\lambda_r(\bs D):= \inf\{\text{intensity of } \bs R \text{ in } \bs D: {\bs R\in \mathcal C_r} \},
\end{equation}
where 
the intensity is defined in Definition~\ref{def:subset}.
\end{definition}
Note that $\lambda_r$ is non-increasing in terms of $r$. A smaller $\lambda_r$ heuristically
means that \textit{a smaller number of balls per point} is needed to cover $\bs D$. So define

\begin{definition}
\label{def:minkowski}
The \defstyle{upper and lower unimodular Minkowski dimensions} of $\bs D$ are defined by
\begin{eqnarray*}
	\dimMu{\bs D}&:=& {\decayu{\lambda_r}}, \\ 
	\dimMl{\bs D}&:=& {\decayl{\lambda_r}},  
\end{eqnarray*}
as $r\rightarrow \infty$.
If the decay rate of $\lambda_r$ exists, define the \defstyle{unimodular Minkowski dimension} of $\bs D$ by
\[ \dimM{\bs D}:=\decay{\lambda_r}.  \]
\end{definition}

Here are some first illustrations of the definition.

\begin{example}
\label{ex:lattice-lowerbound}
The randomly shifted lattice $\bs S_n:=(2n+1)\deltaone\mathbb Z^k-\deltaone \bs U_n$, 
where $\bs U_n\in\{-n,\ldots,n\}^k$ is chosen uniformly, is an equivariant $n$-covering of
$\mathbb Z^k$ equipped with the $l_{\infty}$ metric (other metrics can be treated similarly).
This implies that $\lambda_{n\deltaone}\leq \myprob{0\in \bs S_n} = (2n+1)^{-k}$, and hence,
$\dimMl{\deltaone\mathbb Z^k}\geq k$.
\end{example}

\begin{example}
\label{ex:finite-Minkowski}
If $\bs D$ is finite with positive probability, then it can be seen that any non-empty
equivariant subset has intensity at least $\omid{{1}/{\card{\bs D}}}$ (use the mass transport
principle when sending mass $1/{\card{\bs D}}$ from every point of the subset to every point of $\bs D$).
This implies that $\dimM{\bs D}= 0$.
\end{example}

\begin{remark}[Bounding the Minkowski Dimension]
\label{rem:boundingMinkowski}
In all examples in this work, lower bounds on the unimodular Minkowski dimension are obtained by constructing 
explicit examples of $r$-coverings. 
Upper bounds can be obtained by constructing {\textit{disjoint} or \textit{bounded}} coverings, as discussed in
Subsection~\ref{subsec:optimal} below, or by comparison with the unimodular Hausdorff dimension
defined in Subsection~\ref{subsec:HausdorffDim} below (see Theorem~\ref{thm:comparison}).
\end{remark}

\subsection{{Optimal Coverings for the Minkowski Dimension}}
\label{subsec:optimal}

\begin{definition}
\label{def:optimalCovering}
Let $[\bs D, \bs o]$ be a unimodular discrete space and {$r\geq  0$}. If the infimum in the 
definition of $\lambda_r$ \eqref{eq:I_r} is attained by an equivariant $r$-covering $\bs S$; i.e.,
$\myprob{\bs o \in \bs S_{\bs D}}=\lambda_r$, then $\bs S$ is called an \defstyle{optimal $r$-covering} for $\bs D$.
\end{definition}

\begin{theorem}
\label{thm:optimalCovering}
Every unimodular discrete space has an optimal $r$-covering for every {$r\geq 0$}.
\end{theorem}


{This theorem is proved in Appendix~\ref{ap:tightness} by a tightness argument.}
In general, finding an optimal covering is difficult. In some examples, the following is easier to study.

\begin{definition}
\label{def:nearlyDisjoint}
Let $K<\infty$ {and $r\geq 0$}. An $r$-covering of $\bs D$ is {\defstyle{$K$-bounded}}
if each point of $\bs D$ is covered at most $K$ times a.s.
A sequence $(\bs R_n)_n$ of equivariant coverings of $\bs D$ is called {\defstyle{uniformly bounded}} 
if there is $K<\infty$ such that each $\bs R_n$ is $K$-bounded.
\end{definition}

\begin{lemma}
\label{lem:nearlyDisjoint}
If $\bs R$ is a $K$-bounded equivariant $r$-covering of $\bs D$, then
\begin{equation}
\label{eq:Kcover}
\frac 1K \myprob{\bs R(\bs o)\neq 0} \leq \lambda_r \leq \myprob{\bs R(\bs o)\neq 0}.
\end{equation}
So if $(\bs R_n)_n$ is a sequence of equivariant coverings which is uniformly bounded,
with $\bs R_n$ an $n$-covering for each $n\geq 1$, then
	\begin{eqnarray*}
		\dimMu{\bs D} &=& {\decayu{\myprob{\bs R_n(\bs o)\neq 0}}}, \\
		\dimMl{\bs D} &=& {\decayl{\myprob{\bs R_n(\bs o)\neq 0}}}.
	\end{eqnarray*}
\end{lemma}

\begin{proof}
The rightmost inequality in (\ref{eq:Kcover}) is immediate from the definition of $\lambda_r$.
Let $\bs R'$ be another equivariant $r$-covering. Let $g(u,v)=1$ if 
$\bs R'(u)=\bs R(v)=r$ and $d(u,v)\leq r$.
Then $g^+(\bs o)\leq K \identity{\{\bs R'(\bs o)\neq 0 \}}$ and $g^-(\bs o)\geq \identity{\{\bs R(\bs o)\neq 0 \}}$.
Hence by the mass transport principle~\eqref{eq:unimodularMarked},
$\frac 1K \myprob{\bs R(\bs o)\neq 0} \leq \myprob{\bs R'(\bs o)\neq 0}$ and the leftmost
inequality in (\ref{eq:Kcover}) then follows from the definition of $\lambda_r$.
The last two equalities follow immediately from~\eqref{eq:Kcover}. 
\end{proof}

\begin{corollary}
\label{cor:disjoint}
If $\bs R$ is an equivariant \defstyle{disjoint $r$-covering} of $\bs D$ 
(i.e., the balls used in the covering are pairwise disjoint a.s.), then it is an optimal $r$-covering for $\bs D$.
\end{corollary}

\begin{example}
\label{ex:lattice-Minkowski}
{The covering of $\mathbb Z^k$ (equipped with the $l_{\infty}$ metric) constructed in Example~\ref{ex:lattice-lowerbound} is a disjoint covering. So it is optimal and hence $\dimM{\mathbb Z^k}=k$. For $\mathbb Z^k$ equipped with the Euclidean metric, one can construct a $3^k$-bounded covering similarly and deduce the same result.}
\end{example}

\begin{example}
\label{ex:regtree-Minkowski}
Let $T_k$ be the $k$-regular tree. \formir{For $r\geq 1$,} consider a deterministic covering of $T_k$
by disjoint balls of radius $r$. By choosing $\bs o$ in one of these balls uniformly at random, 
it can be seen that an equivariant disjoint $r$-covering of $[T_k,\bs o]$ is obtained
(the proof is left to the reader). So Corollary~\ref{cor:disjoint} implies that
$\lambda_r=1/\card{\nei{r}{\bs o}}$ which has exponential decay when $k\geq 3$.
Hence, $\dimM{T_k}=\infty$ for $k\geq 3$.
\end{example}

\begin{proposition}
\label{prop:lowerBoundR}
For any point-stationary point process $\Phi$ on $\mathbb R$ endowed with the Euclidean metric, 
by letting $p(r):=\myprob{\Phi \cap (0,r)=\emptyset}$, one has
\begin{eqnarray*}
\dimMu{\Phi} & = \decayu{\frac 1 r \int_0^r p(s){\mathrm d}s} \leq & 1\wedge \decayu{p(r)},\\
\dimMl{\Phi}
& =  \decayl{\frac 1 r \int_0^r p(s){\mathrm d}s}
= & 1\wedge \decayl{ p(r)}.
\end{eqnarray*}
\end{proposition}

\begin{proof}
Let $r>0$ and $\varphi$ be a discrete subset of $\mathbb R$.
Let $\bs U_r$ be a random number in $[0,r)$ chosen uniformly.
For each $n\in \mathbb Z$, put a ball of radius $r$ centered at the largest element of
$\varphi \cap [nr+\bs U_r, (n+1)r+\bs U_r)$.
Denote this random $r$-covering of $\varphi$ by $\bs R_{\varphi}$.
One can see that $\bs R$
is equivariant under translations (see Remark~\ref{rem:equivProcessExtraMarks}).
This implies that $\bs R$ is an equivariant covering (verifying Condition~\eqref{def:equivProcess:2}
of Definition~I.\ref{def:equivProcess} is skipped here).  One has
\[
\myprob{0\in \bs R_{\Phi}} = \myprob{\Phi\cap (0,\bs U_r) = \emptyset} 
= \frac 1 r \int_0^r \myprob{\Phi\cap (0,s) = \emptyset}{\mathrm d}s \formir{=: q(r)}.
\]
Now, since $\bs R$ is a 3-bounded covering, Lemma~\ref{lem:nearlyDisjoint} implies 
the two left-hand-side equalities.
For all $\beta< \decayl{p(r)}$, 
one has $p(r)< r^{-\beta}$ for large enough $r$. So, if in addition, $\beta < 1$,
then $q(r) < c r^{-\beta}$ for some constant $c$, so that $\decayl{q(r)} \geq  \beta$.
\formir{Therefore $\decayl{q(r)}\geq 1\wedge\decayl{p(r)}$.
Now, the final equality in the claim is deduced from $q(r)\geq p(r)$.
Similarly, if $\decayu{p(r)}<1$, one can deduce $\decayu{q(r)}\leq\decayu{p(r)}$.
Also, $q(r)\geq \frac 1 r \int_0^1 p(s)ds$, and hence $\decayu{q(r)}\leq 1$.
This implies the first inequality and completes the proof.}
%
\end{proof}


\subsection{The Unimodular Hausdorff Dimension}
\label{subsec:HausdorffDim}

The definition of the unimodular Hausdorff dimension is based on coverings of the discrete space
by balls of possibly different radii. Such a covering can be represented by an assignment of marks to the points, 
where the mark of a point $v$ represents the radius of the ball centered at $v$.
\formir{As mentioned earlier, it is convenient to assume that the radii are at least 1
(in fact, this condition is technically necessary in what follows).}
Also, by convention, if there
is no ball centered at $v$, the mark of $v$ is defined to be 0. In relation with this convention,
the following notation is used for all discrete spaces $D$ and points $v\in D$:
\[
	N_r(v):=
	\left\{
		\begin{array}{ll}
			\{u\in D: d(v,u)\leq r \}, & r\geq 1,\\
			\emptyset, & r=0.
		\end{array}
	\right.
\]
In words, $N_r(v)$ is the \textit{closed ball} of radius $r$ centered at $v$, except when $r=0$.

\begin{definition} 
\label{def:convering}
Let $[\bs D, \bs o]$ be a unimodular discrete space. An \defstyle{equivariant (ball-) covering} 
$\bs R$ of $\bs D$ is an equivariant process on $\bs D$ (Definition~\ref{def:equivProcess})
with values in $\Xi:=\{0\}\cup [1,\infty)$, 
such that the family of balls $\{N_{{\bs R}(v)}(v): v\in \bs D \}$ covers the points of $\bs D$ almost surely.
For simplicity, $N_{{\bs R}(v)}(v)$ will also be denoted by $N_{\bs R}(v)$.
Also, for $0\leq \alpha<\infty$ and $1\leq M<\infty$, let
\begin{equation}
\label{eq:hcontent}
\contentH{\alpha}{M}(\bs D):=\inf \left\{\omid{\bs R(\bs o)^{\alpha}}:  \bs R(v) \in \{0\}\cup[M,\infty),\; \forall v,\; \text{a.s.}\right\},
\end{equation}
where the infimum is over all equivariant coverings $\bs R$ such that almost surely, 
$\forall v\in \bs D: \bs R(v)\in \{0\}\cup[M,\infty)$, 
and, by convention, $0^0:=0$.
Note that $\contentH{\alpha}{M}(\bs D)$ is a non-decreasing function of both $\alpha$ and $M$.
\end{definition}

\formir{In the \textit{ergodic} case, $\omid{\bs R(\bs o)^{\alpha}}$ can be interpreted
as the average of $\bs R(\cdot)^{\alpha}$ over the vertices.} 
\formir{Also, $\myprob{\bs R(\bs o)>0}$
(which is used for defining the unimodular Minkowski dimension) can be interpreted as 
\textit{the number of balls per point}. Ergodicity is however a special case, and there
is no need to assume it in what follows; for more on the matter, 
see Example~\ref{ex:nonergodic} and the discussion after it.}

\begin{definition}	
\label{def:HausdorffDim}
Let $[\bs D, \bs o]$ be a unimodular discrete space. The number
$\contentH{\alpha}{1}(\bs D)$, defined in~\eqref{eq:hcontent}, is called the
\defstyle{$\alpha$-dimensional Hausdorff content} of $\bs D$.
The \defstyle{unimodular Hausdorff dimension} of $\bs D$ is defined by 
	\begin{equation}
		\dimH{\bs D}:= \sup\{\alpha\geq 0: \contentH{\alpha}{1}(\bs D)=0 \},
	\end{equation}
with the convention that $\sup \emptyset = 0$.
\end{definition}

The key point of assuming equivariance in the above definition is that by Lemma~\ref{lem:equivProcess},
$[\bs D, \bs o; \bs R]$ is a unimodular marked discrete space. Note also that extra randomness
is allowed in the definition of equivariant coverings.
Note also that
\[
	0\leq \contentH{\alpha}{1}(\bs D)\leq 1,
\]
since for the covering by balls of radius 1, one has $\omid{\bs R(\bs o)^{\alpha}}=1$.

Examples~\ref{ex:hausdorff-basic}
and~\ref{ex:nonergodic} below provide basic
illustrations of the unimodular Hausdorff dimension.

\begin{example}
	\label{ex:hausdorff-basic}
	If $\bs D$ is finite with positive probability, then one can show similarly to Example~\ref{ex:finite-Minkowski} that $\omid{\bs R(\bs o)^\alpha}\geq \omid{1/\card{\bs D}}$ for every $\bs R$, and hence, $\dimH{\bs D}= 0$. Also, for the covering $\bs S_n$ of $\mathbb Z^k$ constructed in Example~\ref{ex:lattice-lowerbound}, one has $\omid{\bs S_n(\bs o)^{\alpha}} = (2n+1)^{\alpha-k}$. If $\alpha<k$, this implies that $\contentH{\alpha}{1}(\mathbb Z^k)=0$, and hence, $\dimH{\mathbb Z^k}\geq k$.
	The upper bound $\dimH{\mathbb Z^k}\leq k$ is implied by Lemma~\ref{lem:mdp-simple} below. So $\dimH{\mathbb Z^k}=k$. 
\end{example}

%
%
%
%

\begin{lemma}
\label{lem:mdp-simple}
Let $[\bs D, \bs o]$ be a unimodular discrete space and $\alpha\geq 0$. If there exists $c\geq 0$ such that
$\forall r\geq 1:\card{N_r(\bs o)\leq cr^{\alpha}}$ a.s., then $\dimH{\bs D}\leq \alpha$.
\end{lemma}
\begin{proof}
Let $\bs R$ be an arbitrary equivariant covering. For all discrete spaces $D$ and $u,v\in D$,
let $g_{D}(u,v)$ be 1 if  $d(u,v)\leq \bs R_D(u)$ and 0 otherwise. One has $g^+(u)=\card{\nei{\bs R}{u}}$ 
and $g^-(u)\geq 1$ a.s. (since $\bs R$ is a covering). By the assumption and the mass transport 
principle~\eqref{eq:unimodularMarked}, one gets
	\begin{eqnarray*}
		\omid{\bs R(\bs o)^{\alpha}} \geq \frac 1 c \omid{\card{\nei{\bs R}{\bs o}}} = \frac 1 c \omid{g^+(\bs o)} = \frac 1 c \omid{g^-(\bs o)} \geq \frac 1 c.
	\end{eqnarray*}
Since $\bs R$ is arbitrary, one gets $\contentH{\alpha}{1}(\bs D)\geq \frac 1 c>0$, and hence, $\dimH{\bs D}\leq \alpha$.
\end{proof}

\begin{remark}[Bounding the Hausdorff Dimension]
	\label{rem:boundingHausdorff}
In most examples in this work, a lower bound on the unimodular Hausdorff dimension is provided, 
either by comparison with the Minkowski dimension (see Subsection~\ref{subsec:comparison} below),
or by explicit construction of a sequence of equivariant coverings $\bs R_1,\bs R_2,\ldots$ such
that $\omid{\bs R_n(\bs o)^{\alpha}}\rightarrow 0$ as $n\rightarrow\infty$. Note that this gives
$\contentH{\alpha}{1}(\bs D)=0$, which implies that $\dimH{\bs D}\geq \alpha$. 
Constructing coverings does not help to find upper bounds for the Hausdorff dimension.
The derivation of upper bounds is mainly discussed in {Section~\ref{sec:volumeGrowth}}. 
The main tools are \textit{the mass distribution principle} (Theorem~\ref{thm:mdp-simple}),
which is a stronger form of Lemma~\ref{lem:mdp-simple} above, and the \textit{unimodular Billingsley's lemma}
(Theorem~\ref{thm:billingsley}).
\end{remark}

\begin{example}
\label{ex:nonergodic}
Let $[\bs D, \bs o]$ be $[\mathbb Z,0]$ with probability $\frac 12$ and $[\mathbb Z^2,0]$
with probability $\frac 12$. It is shown below that $\dimM{\bs D} = \dimH{\bs D} = 1$.
\\	
For $n\in\mathbb N$, the equivariant $n$-covering of Example~\ref{ex:lattice-lowerbound} makes sense
for $\bs D$ and is uniformly bounded. One has $\myprob{\bs R(0)>0} = \frac 12 (n^{-1}+n^{-2})$.
This implies that $\dimM{\bs D} = \decay{\frac 12 (n^{-1}+n^{-2})}= 1$. \formir{Also, for $\alpha<1$, one has $\omid{\bs R(\bs o)^{\alpha}}=\frac 12(n^{\alpha-1}+n^{\alpha-2})\to 0$ as $n\to\infty$. This implies that $\contentH{\alpha}{1}(\bs D)=0$ for all $\alpha<1$ and hence} $\dimH{\bs D}\geq 1$.
On the other hand, for any equivariant covering $\bs S$, one has
	\[
		\omid{\bs S(\bs o)}\geq \omidCond{\bs S(\bs o)}{\bs D=\mathbb Z}\myprob{\bs D = \mathbb Z} = \frac 12 \omidCond{\bs S(\bs o)}{\bs D=\mathbb Z}.
	\]
Let $c>2$. The proof of Lemma~\ref{lem:mdp-simple} for $[\mathbb Z,0]$ implies that
$\omidCond{\bs S(\bs o)}{\bs D=\mathbb Z}\geq \frac 1 c$.
This implies that $\contentH{1}{1}(\bs D)\geq \frac 1{2c}>0$. So $\dimH{\bs D}\leq 1$.
\end{example}

\begin{remark}
\label{rem:nonergodic-justification}
\formir{The result of this example might seem counterintuitive
at first glance as the union of a filled square and a segment is two dimensional.}
The number of balls of radius $\epsilon$ required to cover the square dominates
the number of balls required to cover the segment, 
but in Example~\ref{ex:nonergodic}, the situation is reversed: 
a larger \textit{fraction} of points is needed to cover $\mathbb Z$ than $\mathbb Z^2$. 
This is a consequence of considering large balls and also counting the number of balls \textit{per point}. \formir{See also Subsection~\ref{subsec:negative}.} 
\\
In fact, the following example justifies more clearly why Example~\ref{ex:nonergodic} is one dimensional: 
Let $G_n$ be the union of a $n\times n$ square grid (regarded as a graph) and a path of length
$n^2$ sharing a vertex with the grid. To cover $G_n$ by balls of radius $r$, a fraction of order
$1/r$ of the vertices of $G_n$ are needed (as $r$ is fixed and $n\to \infty$).
So it is not counterintuitive to say that $G_n$ is one dimensional asymptotically.
Indeed, $G_n$ tends to the random graph of Example~\ref{ex:nonergodic} in the local
weak convergence \cite{processes} as $n\to\infty$ (if one chooses the root of $G_n$ randomly and uniformly).
\end{remark}

\begin{remark}
\label{rem:nonergodic}
In Example~\ref{ex:nonergodic} above, different samples of $\bs D$ have different natures heuristically.
This is formalized by saying that $[\bs D, \bs o]$ is \textit{non-ergodic}; i.e., there is an event 
$A\subseteq \dstar$ such that the proposition $[D,o]\in A$ does not depend on the \rooot{} of $D$ and
$0<\myprob{[\bs D, \bs o]\in A}<1$. 
\formir{In such cases, it is desirable to assign a dimension to every sample of $\bs D$. In easy examples like Example~\ref{ex:nonergodic}, this might be achieved by conditioning. For instance, in some examples, it is convenient to condition on having infinite cardinality (which is common, e.g., in branching processes). 
	However, in general, it doesn't seem easier to define the dimension of samples separately in a way that is compatible with the definitions of this paper.
	In the future work~\cite{III}, the notion of \textit{sample dimension} is defined by combining the definitions in this paper with either \textit{ergodic decomposition} or conditional expectation.}
In this work, \formir{the reader may focus mainly} on the ergodic case, \formir{but it should be noted that} {the} definitions and results do not require ergodicity.
%
\end{remark}

\subsection{Comparison of Hausdorff and Minkowski Dimensions}
\label{subsec:comparison}

\begin{theorem}[Minkowski vs. Hausdorff]
\label{thm:comparison}
One has
	\[
	\dimMl{\bs D}\leq \dimMu{\bs D}\leq \dimH{\bs D}.
	\]
\end{theorem}
\begin{proof}
The first inequality holds by the definition. For the second one, the definition of
$\lambda_r$ \eqref{eq:I_r} implies that for every $\alpha\geq 0$ and $r\geq 1$,
	\[
	\inf \{\omid{\bs R(\bs o)^{\alpha}}: \bs R \text{ is an equivariant } r\text{-covering}\} = r^\alpha \lambda_r.
	\]
This readily implies that 
	$\contentH{\alpha}{1}(\bs D) \leq r^\alpha \lambda_r$ for every $r\geq 1$. \formir{So, if $\alpha<\decayu{\lambda_r}$,}
one gets
 $\contentH{\alpha}{1}(\bs D)=0$, and hence, $\dimH{\bs D}\geq\alpha$. This implies the claim.
\end{proof}

\begin{remark}
\label{rem:regular}
There exist examples in which the inequalities in Theorem~\ref{thm:comparison}
are strict (see e.g., Subsections~\ref{subsec:canopyGeneralized}
and~\ref{subsec:subspace-minkowski}). \formir{However, equality holds in most examples. In what follows,
the equality $\dimM{\bs D}=\dimH{\bs D}$ will be referred to as \textit{regularity} 
for the unimodular discrete space $\bs D$, regarded as a fractal object.}
\end{remark}

\subsection{The Unimodular \formir{Hausdorff Size}}
\label{ss:hausmeas}
Consider the setting of Subsection~\ref{subsec:HausdorffDim}.
For $0\leq\alpha<\infty$, let
	\begin{equation}\label{eq:hcont-infinity}
		\contentH{\alpha}{\infty}(\bs D):=\lim_{M\rightarrow\infty} \contentH{\alpha}{M}(\bs D)\in [0,\infty],
	\end{equation} 
{where $\contentH{\alpha}{M}(\bs D)$ is defined in (\ref{eq:hcontent}).}
Note that the limit exists because of monotonicity.
\begin{definition}
\label{def:haus-meas}
The \formir{\defstyle{unimodular $\alpha$-dimensional Hausdorff size} of $\bs D$ (in short, unimodular $\alpha$-dim H-size of $\bs D$)} is 
	\begin{equation}\label{eq:hmeas}
	\measH{\alpha}(\bs D):= \left(\contentH{\alpha}{\infty}(\bs D)\right)^{-1}.
	\end{equation}
\end{definition}

\formir{This definition resembles the Hausdorff measure of compact sets. But since $\measH{\alpha}$ is
not a measure, the term \textit{size} is used instead.} It can be used to compare unimodular spaces 
with equal dimension.
The following results gather some elementary properties of the function
$\contentH{\alpha}{M}$ and the Hausdorff \formir{size}.

\begin{lemma}
	\label{lem:Hmeas-elementary}
One has
	\begin{enumerate}[(i)]
\item \label{lem:hausMeas:2} $\contentH{\alpha}{1}(\bs D)\leq \contentH{\alpha}{M}(\bs D)\leq M^{\alpha}\contentH{\alpha}{1}(\bs D)$.
\item \label{lem:hausMeas:1} \formir{$\contentH{\alpha}{1}(\bs D)=0 \Leftrightarrow \contentH{\alpha}{\infty}(\bs D)=0 \Leftrightarrow \measH{\alpha}(\bs D)=\infty$.}
\item \label{lem:hausMeas:3} If $\alpha\geq \beta$, then $\contentH{\alpha}{M}(\bs D)\geq M^{\alpha-\beta}\contentH{\beta}M{(\bs D)}$.
	\end{enumerate}
\end{lemma}

\begin{proof}
\eqref{lem:hausMeas:2}. If $\bs R$ is an equivariant covering, them $M\bs R$ is also an equivariant
covering and satisfies $\forall v\in \bs D: MR(v)\in \{0\}\cup[M,\infty)$ a.s.
	
\eqref{lem:hausMeas:1}. The claim is implied by part~\eqref{lem:hausMeas:2}.
	
\eqref{lem:hausMeas:3}. If $\bs R$ is an equivariant covering such that $\forall v\in \bs D: \bs R(v)\in \{0\}\cup[M,\infty)$ 
a.s., then $\bs R(\bs o)^{\alpha}\geq M^{\alpha-\beta} \bs R(\bs o)^{\beta}$ a.s.
\end{proof}

\begin{lemma}
	\label{lem:Hmeas}
		\formir{If $\alpha<\dimH{\bs D}$, then $\contentH{\alpha}{\infty}(\bs D)=0$ and $\measH{\alpha}(\bs D)=\infty$.
		Also, if $\alpha>\dimH{\bs D}$, then $\contentH{\alpha}{\infty}(\bs D)=\infty$ and $\measH{\alpha}(\bs D)=0$.}
\end{lemma}
\begin{proof}
For $\alpha<\dimH{\bs D}$, one has $\contentH{\alpha}{1}(\bs D)=0$. So part~\eqref{lem:hausMeas:1}
of Lemma~\ref{lem:Hmeas-elementary} implies that $\measH{\alpha}(\bs D)=\infty$. 
For $\alpha>\dimH{\bs D}$, there exists $\beta$ such that $\alpha>\beta>\dimH{\bs D}$.
For this $\beta$, one has $\contentH{\beta}{1}(\bs D)>0$ and part~\eqref{lem:hausMeas:3} of the same lemma implies that 
$\contentH{\alpha}{M}(\bs D)\geq M^{\alpha-\beta}\contentH{\beta}{M}(\bs D) \geq M^{\alpha-\beta}\contentH{\beta}{1}(\bs D)$.
This implies that $\contentH{\alpha}{\infty}(\bs D)=\infty$, which proves the claim.
\end{proof}

\begin{remark}
	For $\alpha:=\dimH{\bs D}$, the \formir{$\alpha$-dim H-size} of $\bs D$ can be zero, finite or infinite.
	The lattice $\mathbb Z^k$ provides a case where $\measH{\alpha}(\bs D)$ is positive and
        finite (Proposition~\ref{prop:lattice-Hmeasure} below).  
	Examples~\ref{ex:infiniteMeasure} and~\ref{ex:zeroMeasure} 
	provide examples of the infinite and zero cases respectively.
\end{remark}

The following propositions provide {basic examples} {of the computation of the \formir{Hausdorff size}}. 

\begin{proposition}[\formir{0-dim H-size}]
	\label{prop:finite-HausMeas}
	One has
	$
		\measH{0}(\bs D)= \left(\omid{ 1 /{\card{\bs D}}}\right)^{-1}.
	$
\end{proposition}
\begin{proof}
	As in {Example~\ref{ex:hausdorff-basic}}, one gets $\contentH{0}{M}(\bs D)\geq \omid{1/\card{\bs D}}$. It is enough to show that equality holds.
	{If $\bs D$ is finite a.s., this can be proved by putting a single ball of radius $M\vee\diam(\bs D)$ centered at a point of
		$\bs D$ chosen uniformly at random.}
%
Second, assume $\bs D$ is infinite a.s. It is enough to construct an equivariant covering $\bs R$	 
such that $\myprob{\bs R(\bs o)>0}$ is arbitrarily small. Let $p>0$ be arbitrary and $\bs S$ be the
\textit{Bernoulli equivariant subset} obtained by selecting each point with probability $p$ in an i.i.d. manner. 
For all infinite discrete spaces $D$ and $v\in D$, let $\bs \tau_D(v)$ be the closest point of $\bs S_D$ to $v$
(if there is a tie, choose one of them uniformly at random independently). It can be seen that
$\bs \tau_D^{-1}(u)$ is finite almost surely (use the mass transport principle for
$\formir{g(v,u)}:=\identity{\{u=\bs \tau_D(v)\}}$).
For $u\in \bs S_D$, let $\bs R(u):={1\vee}\diam (\bs \tau^{-1}(u))$ be the diameter of the \textit{Voronoi cell} of $u$. {For $u\in D\setminus\bs S_{D}$, let $\bs R(u):=0$.}
It is clear that $\bs R$ is a covering, and in fact, an equivariant covering.
One has $\myprob{\bs R(\bs o)>0} = \myprob{\bs o\in \bs S_{\bs D}} = p$, which is arbitrarily small. So the claim is proved in this case.
	
Finally, assume $\bs D$ is finite with probability $q$.
For all deterministic discrete spaces $D$, let $\bs R_D$ be one of the above
coverings depending on whether $D$ is finite or infinite.
It satisfies $\myprob{\bs R(\bs o)>0} = \omid{1/\card{\bs D}}+{p(1-q)}$.
Since $p$ is arbitrary, the claim is proved.
\end{proof}

\begin{proposition}
\label{prop:lattice-Hmeasure}
\formir{For all $\delta>0$,} the $k$-dim H-size
of the scaled lattice $[\delta\mathbb Z^k,0]$, equipped with the $l_{\infty}$ metric, is equal to
$ \left(2/{\delta}\right)^{k}.$
\end{proposition}
\begin{proof}
Let $\bs S_n$ be the covering in Example~\ref{ex:lattice-lowerbound} scaled by factor $\delta$. One has
$\omid{\bs S_n(\bs o)^k} = (n\delta)^k/(2n+1)^k$.
This easily implies that $\contentH{k}{\infty}(\delta\mathbb Z^k)\leq (\delta/2)^k$. 
On the other hand, the proof of Lemma~\ref{lem:mdp-simple} 
shows that $\contentH{k}{\infty}(\delta\mathbb Z^k)\geq c\delta^k$, where $c$ is any constant such that
$r^k\geq c\card{N_r(0)}$ for large enough $r$. It follows that
$\contentH{k}{\infty}(\delta\mathbb Z^k)\geq (\delta/2)^k$, and the claim is proved.
\end{proof}

\subsection{The Effect of a Change of Metric}
\label{subsec:metricChange}

To avoid confusion {when considering two} metrics, a \rooted{} discrete space is denoted by $((D,d),o)$ here,
where $d$ is the metric on $D$ and $o$ is the \rooot{}. Note that if $d'$ is another metric on $D$,
then $d'\in \mathbb R^{D\times D}$. So $d'$ can be considered as a marking of $D$ in the sense of
Definition~\ref{def:marking} and $((D,d),o;d')$ is a \rooted{} marked discrete space. 

\begin{definition}
\label{def:equivMetric}
An \defstyle{equivariant (boundedly finite) metric} is an $\mathbb R$-valued equivariant process
$\bs d'$ such that, for all discrete spaces $\formir{(D,d)}$, $\bs d'_{(D,d)}$ is almost surely
(w.r.t. the extra randomness) a metric on $D$ and $(D,\bs d'_{(D,d)})$ is a boundedly finite metric space.
\end{definition}

If in addition, $[(\bs D, \bs d), \bs o]$ is a unimodular discrete space,
then $[(\bs D, \bs d), \bs o; \bs d']$ is a unimodular marked discrete space by Lemma~\ref{lem:equivProcess}.
It can be seen that $[(\bs D, \bs d'), \bs o; \bs d]$,
obtained by swapping the metrics, makes sense as a random \rooted{} marked discrete space 
(see {Lemma~\ref{lem:metricChange-measurability}} for the measurability requirements). 
By verifying the mass transport principle~\eqref{eq:unimodularMarked} directly,
it is easy to show that $[(\bs D, \bs d'), \bs o; \bs d]$ is unimodular.

The following result is valid for both the Hausdorff and the (upper and lower) Minkowski dimensions.

\begin{theorem}[{Change of Metric}]
\label{thm:metricChange}
Let $[(\bs D, \bs d), \bs o]$ be a unimodular discrete space and  $\bs d'$ be an equivariant metric. 
If {$\bs d'\leq c \bs d+a$} a.s., {with $c$ and $a$ constants},
then the dimension of $(\bs D, \bs d')$ is larger 
than or equal to that of $(\bs D,\bs d)$. 
Moreover, for every $\alpha\geq 0$, 
$\measH{\alpha}(\bs D, \bs d')\geq c^{-\alpha} \measH{\alpha}(\bs D, \bs d).$
\end{theorem}

	\begin{proof}
The claim is implied by the fact that the ball $N_{cr+a}((\bs D, \bs d'), v)$ contains the ball $N_{r}((\bs D, \bs d), v)$ \formir{and is left to the reader.}
\end{proof}

{As a corollary, if $\frac 1 c\bs d -a\leq \bs d'\leq c \bs d+a$ {a.s.}, then $(\bs D, \bs d')$ has the same unimodular dimensions as $(\bs D,\bs d)$. Also, $c\bs D$ has the same dimension as $\bs D$ and
$
\measH{\alpha}(c\bs D) = c^{-\alpha} \measH{\alpha}(\bs D).
$}

For instance, this result can be applied to Cayley graphs, which are
an important class of unimodular graphs~\cite{processes}. It follows that the unimodular dimensions
of a Cayley graph do not depend on the generating set.
In fact, it will be proved in Subsection~\ref{subsec:cayley} that these dimensions
are equal to the \textit{polynomial growth {degree}} of $H$.

\begin{example}
Let $[\bs G, \bs o]$ be a unimodular graph. Examples of equivariant metrics on $\bs G$ are
the graph-distance metric corresponding to an equivariant spanning subgraph 
(e.g., the drainage network model of Subsection~\ref{subsec:drainage} below) and
metrics \textit{generated by equivariant edge lengths}. More precisely, if $\bs l$ is an 
equivariant process which assigns a positive \textit{weight} to the edges of every deterministic
graph, then one can let $\bs d'(u,v)$ be the minimum weight of the paths that connect $u$ to $v$.
If $\bs d'$ is a metric for almost every realization of $\bs G$ and is boundedly-finite a.s.,
then it is an equivariant metric. 
\end{example}

\subsection{Dimension of Subspaces}
\label{subsec:equiv-subspace}

Let $[\bs D, \bs o]$ be a unimodular discrete space and $\bs S$ be an equivariant subset which is almost surely nonempty.
Lemma~\ref{lem:happensAtRoot} implies that $\myprob{\bs o\in \bs S_{\bs D}}>0$. So one can consider
$[\bs S_{\bs D}, \bs o]$ conditioned on $\bs o\in \bs S_{\bs D}$. By directly verifying the mass transport
principle~\eqref{eq:unimodular}, it is easy to see that  $[\bs S_{\bs D}, \bs o]$ conditioned
on $\bs o\in \bs S_{\bs D}$ is unimodular (see the similar claim for unimodular graphs in~\cite{eft}). 

\begin{convention}
\label{conv:subsetDim}
For an equivariant subset $\bs S$ as above, the unimodular Hausdorff dimension of $[\bs S_{\bs D}, \bs o]$
(conditioned on $\bs o\in S_{\bs D}$) is denoted by $\dimH{\bs S_{\bs D}}$. The same convention is used for the Minkowski dimension, the \formir{Hausdorff size}, etc.
\end{convention}	
	
\begin{theorem}
\label{thm:subsetDimension}
Let $[\bs D, \bs o]$ be a unimodular discrete space and $\bs S$ an equivariant subset
such that $\bs S_{\bs D}$ is nonempty a.s. Then,
\begin{enumerate}[(i)]
\item \label{thm:subsetDimension:dim} {One has 
\begin{eqnarray*}
\dimH{\bs S_{\bs D}} &=& \dimH{\bs D},\\
\dimMu{\bs S_{\bs D}} &\geq& \dimMu{\bs D},\\
\dimMl{\bs S_{\bs D}} &\geq& \dimMl{\bs D}.
\end{eqnarray*}}
\item \label{thm:subsetDimension:measure} If $\rho$ is the intensity of $\bs S$
in $\bs D$, then for every $\alpha\geq 0$, the \formir{$\alpha$-dim H-size} of $\bs S_{\bs D}$ satisfies
\begin{equation*}
{2^{-\alpha}\rho\: \measH{\alpha}(\bs D)\leq \measH{\alpha}(\bs S_{\bs D}) \leq \rho\: \measH{\alpha}(\bs D).}
\end{equation*}
\end{enumerate}
\end{theorem}

\formir{
		Theorem~\ref{thm:subsetDimension} is proved below by using the fact that every covering of the
		larger set induces a covering of the subset by deleting some balls and then re-centering
		and enlarging the remaining balls. This matches the analogous idea in the continuum setting. 
		The apparently surprising direction of the inequalities is due to the definition of dimension
		which implies that having less balls means having larger or equal dimension.
		For more on the matter, see {the discussion on negative dimension in} Subsection~\ref{subsec:negative}.
}

\begin{remark}
\label{rem:subset}
Subsection~\ref{subsec:otherSets} below defines a modification $\mathcal M'_{\alpha}(\bs D)$
of the unimodular \formir{Hausdorff size} by considering coverings by arbitrary sets.
With this definition, {one has} $\mathcal M'_{\alpha}(\bs S_d)=\rho\:\mathcal M'_{\alpha}(\bs D)$.
This can be proved similarly to Theorem~\ref{thm:subsetDimension}. with the modification
that there is no need to double the radii. 
\end{remark}

\begin{remark}
In the setting of Theorem~\ref{thm:subsetDimension}, $\dimM{\bs S_{\bs D}}$ can
be strictly larger than $\dimM{\bs D}$ (see, e.g., Subsection~\ref{subsec:subspace-minkowski}).
\formir{However, equality holds when $\bf D$ is regular (see Remark~\ref{rem:regular}), which
immediately follows from Theorems \ref{thm:comparison} and \ref{thm:subsetDimension}.}
Also, equality is guaranteed if $\bs S_{\bs D}$ is a $r$-covering of $\bs D$ for some constant $r$.
\formir{In other words, roughly speaking, the unimodular dimensions are \textit{quasi-isometry invariant} (see e.g., \cite{bookGr91}) and do not depend on the fine details of the discrete space.}
\end{remark}

\begin{proof}[Proof of Theorem~\ref{thm:subsetDimension}]
The first claim of~\eqref{thm:subsetDimension:dim} is implied by
\eqref{thm:subsetDimension:measure} and Lemma~\ref{lem:Hmeas}, \formir{and hence, is skipped.}
Let $\bs R$ be an arbitrary equivariant $r$-covering of $\bs D$.
\formir{For every $v\in\bs R$}, let $\bs{\tau}(v)$ be an element picked uniformly at random in
$N_r(v)\cap {\bs S_{\bs D}}$, \formir{which is defined only when $N_r(v)\cap {\bs S_{\bs D}}\neq\emptyset$}.
Let $\bs R':=\{\bs\tau(v): \formir{v\in \bs R}, N_r(v)\cap \bs S_{\bs D}\neq\emptyset \}$
and note that $\bs R'$ is a $2r$-covering of $\bs S_{\bs D}$.
	%
\formir{One has
\begin{eqnarray*}
\myprob{\bs o\in\bs R'} &\leq& \omid{\sum_{v} \identity{\{v\in\bs R\}} \identity{\{\bs \tau(v)=\bs o\}}}\\
&=& \omid{\sum_{v} \identity{\{\bs o\in \bs R\}} \identity{\{\bs \tau(\bs o)=v\}}} \leq \myprob{\bs o\in \bs R},
\end{eqnarray*}
where the equality is by the mass transport principle. This gives}
\formir{$\rho\lambda_{2r}(\bs S_{\bs D})\leq \lambda_r(\bs D)$,
which implies the claims regarding the Minkowski dimension.}
	
\formir{Now, part~\eqref{thm:subsetDimension:measure} is proved.}
The definition of $\contentH{\alpha}{\infty}(\bs S_{\bs D})$ implies that there
exists a sequence $\bs R_n$ of equivariant coverings of $\bs S_{\bs D}$ such that
$\bs R_n(\cdot)\in\{0\}\cup[n,\infty)$ for all $n=1,2,\ldots$ and
$\omidCond{\bs R_n(\bs o)^{\alpha}}{\bs o\in \bs S_{\bs D}}\rightarrow \contentH{\alpha}{\infty}(\bs S_{\bs D})$.
One may extend $\bs R_n$ to be defined on $\bs D$ by letting $\bs R_n(v):=0$
for $v\in \bs D\setminus \bs S_{\bs D}$.
Let $\epsilon>0$ be arbitrary and $\bs B_n\subseteq \bs D$ be the union of 
$N_{(1+\epsilon)\bs R_n}(v)$ for all $v\in \bs D$. Define
$\bs R'_n(u):= (1+\epsilon)\bs R_n(u)$ for $u\in \bs B_n$
and $\bs R'_n(u):=1/\epsilon$ for $u\not\in \bs B_n$.
It is clear that $\bs R'_n$ is an equivariant covering of $\bs D$. Also, 
\begin{eqnarray}
\nonumber
\omid{\bs R'_n(\bs o)^{\alpha}} &=& 
(1+\epsilon)^{\alpha}\omid{\bs R_n(\bs o)^{\alpha}} +
{\frac{1}{{\epsilon^{\alpha}}}} \myprob{\bs o \not \in \bs B_n}\\
\label{eq:thm:subsetDimension:1}
&=& {\rho(1+\epsilon)^{\alpha}\omidCond{\bs R_n(\bs o)^{\alpha}}{\bs o\in \bs S_{\bs D}} +
{\frac{1}{{\epsilon^{\alpha}}}} \myprob{\bs o \not \in \bs B_n}.}
\end{eqnarray}
Since the radii of the balls in $\bs R_n$ are at least $n$,
one gets that $\bs B_n$ includes the $\epsilon n$-neighborhood of $\bs S_{\bs D}$. Therefore, 	
	$\myprob{\bs o \not \in \bs B_n} \leq 
        \myprob{N_{\epsilon n}(\bs o)\cap \bs S_{\bs D}=\emptyset}$.
	Since $\bs S_{\bs D}$ is nonempty a.s., this in turn implies that 
        $\myprob{\bs o \not \in \bs B_n}\rightarrow 0$ as $n\to\infty$
	(note that the events $N_{\epsilon n}(\bs o)\cap \bs S_{\bs D}=\emptyset$
        are nested and converge to the event $\bs S_{\bs D}=\emptyset$).
	So~\eqref{eq:thm:subsetDimension:1} implies that 
	\[
	\liminf_{n\rightarrow\infty} \omid{\bs R'_n(\bs o)^{\alpha}} =
        {\rho}(1+\epsilon)^{\alpha} \liminf_{n\rightarrow\infty}
        {\omidCond{\bs R_n(\bs o)^{\alpha}}{\bs o\in\bs S_{\bs D}}}
        =\rho (1+\epsilon)^{\alpha} \contentH{\alpha}{\infty}(\bs S_{\bs D}).
	\]
	Note that the radii of the balls in $\bs R'_n$ are at least $n\wedge (1/\epsilon)$.
	Therefore, one obtains
	$\contentH{\alpha}{{1/\epsilon}}(\bs D)\leq \rho(1+\epsilon)^{\alpha} 
        \contentH{\alpha}{\infty}(\bs S_{\bs D})$. By letting $\epsilon\to 0$,
        one gets $\contentH{\alpha}{\infty}(\bs D)\leq \rho \contentH{\alpha}{\infty}(\bs S_{\bs D})$;
	i.e., $\measH{\alpha}(\bs S_{\bs D})\leq \rho \measH{\alpha}(\bs D)$.
	
	Conversely, let $\bs R_n$ be a sequence of equivariant coverings of ${\bs D}$ for
	$n=1,2,\ldots$ such that $\bs R_n(\cdot)\in \{0\}\cup [n,\infty)$ a.s. 
	and $\omid{\bs R_n(\bs o)^{\alpha}} \rightarrow \contentH{\alpha}{\infty}(\bs D)$.
	Fix $n$ in the following. Let
	$\bs B:=\bs B_{\bs D}:= \{v: N_{\bs R_n}(v)\cap \bs S_{\bs D}\neq \emptyset\}$.
	For each $v\in \bs B$, let $\bs \tau_n(v)$ be {an element chosen uniformly at random in}
	$N_{\bs R_n}(v)\cap \bs S_{\bs D}$. For $v\not\in \bs B$, let $\bs \tau_n(v)$ be undefined.
	For $w\in \bs S_{\bs D}$, let
	$
		\bs R'_n(w):= 2 \max \{\bs R_n(v): v\in \bs \tau_n^{-1}(w) \}.
	$
	It can be seen that $\bs R'_n$ is an equivariant covering of $\bs S_{\bs D}$. One has
	\begin{eqnarray*}
		\nonumber
		\omid{\bs R'_n(\bs o)^{\alpha}}
		&\leq & 2^{\alpha} \omid{\sum_{v} \bs R_n(v)^{\alpha}\identity{\{v \in \bs \tau_n^{-1}(\bs o)\}} }\\
		\nonumber &=& 2^{\alpha}\omid{\sum_{v} \bs R_n(\bs o)^{\alpha}\identity{\{\bs o \in \bs \tau_n^{-1}(v)\}} }
		\label{eq:subset} \leq  2^{\alpha} \omid{\bs R_n(\bs o)^{\alpha} },
	\end{eqnarray*}
	where the equality is by the mass transport principle.
	It follows that 
	\[
	\rho \liminf_{n\rightarrow\infty} \omidCond{\bs R'_n(\bs o)^{\alpha}}{\bs o\in \bs S_{\bs D}}
	\leq 2^{\alpha} \contentH{\alpha}{\infty}(\bs D).
	\]
	So $\rho\contentH{\alpha}{\infty}(\bs S_{\bs D}) \leq 2^{\alpha}  \contentH{\alpha}{\infty}(\bs D)$. Hence,
	$\measH{\alpha}(\bs S_{\bs D}) \geq 2^{-\alpha}\rho \: \measH{\alpha}(\bs D)$ and the claim is proved.
\end{proof}


\subsection{Covering By Arbitrary Sets}
\label{subsec:otherSets}

According to Remark~\ref{rem:subset}, it is more natural to redefine the \formir{Hausdorff size}
by considering coverings by \formir{finite} subsets which are not necessarily balls 
(as in the continuum setting). A technical challenge is to define such coverings in
an equivariant way. This will be done at the end of this subsection using the notion of
equivariant processes of Subsection~\ref{subsec:process}. Once an equivariant covering 
$\bs C$ is defined \formir{(which is an equivariant collection of finite subsets)},
one can define the \textit{average diameter of sets $U\in \bs C$ per point}
by $$\omid{\sum_{U\in\bs C} \frac 1{\card{U}}\identity{\{\bs o\in U\}}\diam(U)}.$$
The same idea is used to redefine $\contentH{\alpha}{M}(\bs D)$ as follows:
\begin{equation*}
\mathcal H'_{\alpha,M}(\bs D):= {\inf_{\bs C}} \left\{\omid{
\sum_{U\in\bs C} \frac 1{\card{U}}\identity{\{\bs o\in U\}}\big(M\vee \frac 1 2 \diam(U)\big)^{\alpha} } \right\},
\end{equation*}
where the infimum is over all equivariant coverings $\bs C$. Here, taking the maximum with $M$
is \textit{similar} to the condition that the subsets have diameter at least $2M$ (note that
a ball of radius $M$ might have diameter strictly less than $2M$). 
Finally, define the \defstyle{modified unimodular \formir{Hausdorff size}} 
$\mathcal M'_{\alpha}(\bs D)$ similarly to~\eqref{eq:hmeas}.
Remark~\ref{rem:subset} shows an advantage of this definition. Also,the reader can verify that
${2^{-\alpha}}\contentH{\alpha}{2M}(\bs D) \leq \mathcal H'_{\alpha, M}(\bs D) \leq \contentH{\alpha}{M}(\bs D).$
Therefore, 
\[
\measH{\alpha}(\bs D) \leq \mathcal M'_{\alpha}(\bs D) \leq 2^{\alpha}\measH{\alpha}(\bs D).
\]
This implies that the notion of unimodular Hausdorff dimension is not changed by this modification.
One can also obtain a similar equivalent form of the unimodular Minkowski dimension.
This is done by redefining $\lambda_r$ by considering equivariant coverings by sets of
diameter at most $2r$. The details are left to the reader. 
A similar idea will be used in Subsection~\ref{subsec:one-ended} 
to calculate the Minkowski dimension of one-ended trees.

Finally, here is the promised representation of the above coverings as equivariant processes 
(it should be noted that 
\formir{it is not always possible to number the subsets in an equivariant way}
and the collection should be necessarily unordered). 
To show the idea, consider a covering  $C=\{U_1,U_2,\ldots\}$ of a deterministic discrete space $D$,
where each $U_i$ is bounded. For each $U_i$, assign the mark $(\bs X_i, \diam(\formir{U_i}))$
to every point of $U_i$, where $\bs X_i\in[0,1]$ is chosen i.i.d. and uniformly. Note that
multiple marks are assigned to every point and the covering can be reconstructed from the marks.
With this idea, let the mark space $\Xi$ be the set of discrete subsets of $\mathbb R^2$
(regard every discrete set as a counting measure and equip $\Xi$ with a metrization of the vague topology).
This mark space can be used to represent equivariant coverings by equivariant processes
(for having a complete mark space, one can extend $\Xi$ to the set of discrete multi-sets in $\mathbb R^2$).


\subsection{Notes and Bibliographical Comments}

Several definitions and basic results of this section have analogues in the continuum setting. 
A list of such analogies is given below. Note however that there is no systematic way of
translating the results in the continuum setting to that of unimodular discrete spaces.
In particular, inequalities are most often, but not always, in the other direction.
The comparison of the unimodular Minkowski and Hausdorff dimensions (Theorem~\ref{thm:comparison})
is analogous to the similar comparison in the continuum setting ({see e.g., (1.2.3) of~\cite{bookBiPe17}}),
but in the reverse direction. Theorem~\ref{thm:metricChange}, regarding changing the metric,
is analogous to the fact that the ordinary Minkowski and Hausdorff dimensions
are not increased by applying a Lipschitz function. Theorem~\ref{thm:subsetDimension}
regarding the dimension of subsets is analogous to the fact that the ordinary dimensions
do not increase by passing to subsets. Note however that equality holds
in Theorem~\ref{thm:subsetDimension} for the unimodular Hausdorff dimension 
(and also for the unimodular Minkowski dimension in most usual examples), in contrast to the continuum setting. 

\formir{For point processes (Example~\ref{ex:point-stationary}), one can redefine the
unimodular Hausdorff dimension by using \textit{dyadic cubes} instead of balls. This changes
the value of the Hausdorff size up to a constant factor, and hence, the value of Hausdorff
dimension is not changed. Since dyadic cubes are nested, this simplifies some of the arguments.
This approach will be used in Subsection~\ref{subsec:frostman-euclidean}.}

\section{{Examples}}
\label{sec:examples}

This section presents a set of examples of unimodular discrete spaces together with
discussions about their dimensions. 
Recall that the tools for bounding the dimensions are summarized
in Remarks~\ref{rem:boundingMinkowski} and~\ref{rem:boundingHausdorff}. 
As mentioned in Remark~\ref{rem:boundingHausdorff}, bounding the Hausdorff dimension
from above usually requires the unimodular mass transport principle or the unimodular 
Billingsley lemma, which will be stated in \formir{Section~\ref{sec:volumeGrowth}}.
So the upper bounds for \formir{some} of the following examples are completed
\formir{later in Subsection~\ref{subsec:remainingproofs}}.


\subsection{{General Unimodular Trees}}
\label{subsec:trees}
	
In this subsection, general results are presented regarding the dimension of unimodular 
trees with the graph-distance metric. Specific instances are presented later in the section.
It turns out that the number of \textit{ends} of the tree plays a key role 
(an \defstyle{end} in a tree is an equivalence class of simple paths in the tree, where two such paths
are equivalent if their symmetric difference is finite). 

It is well known that the number of ends in a unimodular 
tree belongs to $\{0,1,2,\infty\}$ \cite{processes}.  Unimodular trees without end are finite, 
and hence, are zero dimensional (Example~\ref{ex:hausdorff-basic}). The only thing to mention 
is that there exists an algorithm to construct an optimal $n$-covering for such trees.
This algorithm is similar to the algorithm for one-ended trees, discussed below,
and is skipped for brevity. In addition, It will be shown in
\formir{Subsection~\ref{subsec:trees2}} that unimodular trees with infinitely many
ends have exponential \formir{volume} growth, and hence, have infinite Hausdorff dimension. 
The remaining two cases are discussed below.

\subsubsection{Unimodular Two-Ended Trees}
\label{subsec:two-ended}
If $T$ is a tree with two ends, then there is a unique bi-infinite path in $T$ called its \defstyle{trunk}.
Moreover, each connected component of the complement of the trunk is finite.

\begin{theorem}
\label{thm:twoEnded}
For all unimodular two-ended trees $[\bs T, \bs o]$ endowed with the graph-distance metric, one has 
$\dimM{\bs T}=\dimH{\bs T}=1.$
{Moreover, if $\rho$ is the intensity of the trunk of $\bs T$, then the modified 
\formir{1-dim H-size} of $\bs T$ is $\mathcal M'_1(\bs T)= 2\rho^{-1}$.}
\end{theorem}
\begin{proof}
For all two-ended trees $T$, let $\bs S_T$ be the trunk of $T$. Then, $\bs S$ is an equivariant subset.
Therefore, Theorem~\ref{thm:subsetDimension}
implies that  $\dimH{\bs T} = \dimH{\bs S_{\bs T}}$. 
Since the trunk is isometric to $\mathbb Z$ as a metric space,
Example~\ref{ex:hausdorff-basic} implies that $\dimH{\bs T}=1$. In addition,
Remark~\ref{rem:subset} and Proposition~\ref{prop:lattice-Hmeasure}
imply that {$\mathcal M'_1(\bs T)=\rho^{-1}\mathcal M'_1(\mathbb Z)= 2\rho^{-1}$}.
			
The claim concerning the unimodular Minkowski dimension is implied by
Corollary~\ref{cor:graphs-lowerbound} \formir{of the next section}, 
which shows that any unimodular infinite graph satisfies $\dimMl{\bs G}\geq 1$
(this theorem will not be used throughout).
\end{proof}

\subsubsection{Unimodular One-Ended Trees}
\label{subsec:one-ended}

Unimodular one-ended trees 
arise naturally in many examples (see~\cite{processes}). In particular, the (local weak)
limit of many interesting sequences of finite trees/graphs are one-ended (\cite{objective, processes}).
In terms of unimodular dimensions, it will be shown that unimodular one-ended trees are the richest class of unimodular trees.
	
First, the following notation is borrowed from~\cite{eft}. Every one-ended tree $T$ can be regarded as a family tree as follows.
For every vertex $v\in T$, there is a unique infinite simple path starting from $v$.
Denote by $F(v)$ the next vertex in this path and call it the \defstyle{parent} of $v$.
By deleting $F(v)$, the connected component containing $v$ is finite.
This set is denoted by $D(v)$ and its elements are called the \defstyle{descendants} of $v$.
The maximum distance of $v$ to its descendants 
\formir{will be called the \defstyle{height} of $v$ and be denoted by $h(v)$}.

\begin{theorem}
\label{thm:one-ended}
If $[\bs T, \bs o]$ is a unimodular one-ended tree endowed with the graph-distance metric, then
		\begin{eqnarray}
		\label{eq:thm:eftMinkowski:mu}
		\dimMu{\bs T} &=& 1+ \decayu{\myprob{h(\bs o)\geq n}}, \\
		\label{eq:thm:eftMinkowski:ml}
		\dimMl{\bs T} &=& 1+ \decayl{\myprob{h(\bs o)\geq n}}. 
		\end{eqnarray}
In addition, 
			\begin{equation}
			\label{eq:thm:eftMinkowski:h}
			\dimH{\bs T} \geq\decayu{\myprob{h(\bs o)=n}}\geq \dimMu{\bs T}.
			\end{equation}
\end{theorem}
	
It should be noted that $\decayu{\myprob{h(\bs o)=n}}$ can be strictly
larger than $1+\decayu{\myprob{h(\bs o)\geq n}}$
(see e.g., Subsection~\ref{subsec:canopyGeneralized}), however,
they are equal in most \textit{usual} examples.
It is not known \formir{whether the first inequality in~\eqref{eq:thm:eftMinkowski:h} is always an equality.}

\formir{
The proof of Theorem~\ref{thm:one-ended} is based on a recursive construction of an optimal
covering by \textit{cones}, defined below, rather than balls. It is shown below that considering
cones instead of balls does not change the Minkowski dimension.
}

The \defstyle{cone} with height $n$ at $v\in \bs T$ is defined by $C_n(v):= N_n(v)\cap D(v)$; i.e.,
the first $n$ generations of the descendants of $v$, including $v$ itself. Let $\lambda'_n$ be the
infimum intensity of equivariant coverings by cones of height $n$. The claim is that 
	\begin{equation}
	\label{eq:eft:I vs I'}
	\lambda'_{2n}\leq \lambda_n\leq \lambda'_n.
	\end{equation}
This immediately implies that
	\begin{equation}
	\label{eq:eft:dim vs I'}
	\dimMl{\bs T} = \decayl{\lambda'_n}, \quad	\dimMu{\bs T} = \decayu{\lambda'_n}.
	\end{equation}
To prove~\eqref{eq:eft:I vs I'}, note that any covering by cones of height $n$ is also a
covering by balls of radius $n$. This implies that $\lambda_n\leq \lambda'_n$.
Also, if $\bs S$ is a covering by balls
of radius $n$, then $\{F^n(v): v\in \bs S \}$ is a covering by cones of height $2n$.
By the mass transport principle~\eqref{eq:unimodularMarked}, one can show that the intensity of the
latter is not greater than the intensity of $\bs S$. 
This implies that $\lambda'_{2n}\leq \lambda_n$. So \eqref{eq:eft:I vs I'} is proved.

\begin{lemma}
\label{lem:eft:optimalcone}
For every unimodular one-ended tree $[\bs T, \bs o]$,
the output $\bs S$ of 
the following greedy algorithm is an optimal equivariant covering of $\bs T$ by cones of height $n$. 

\begin{algorithm}[H]
	$\bs S:=\emptyset$\;
	\While{true}{
		Add {all vertices of height $n$ in $\bs T$ to $\bs S$\;}
		
		$\bs T:=\bs T\setminus \bigcup_{v\in \bs S} D(v)$\;
	}
\end{algorithm}
\end{lemma}
\formir{Note that the algorithm does not finish in finite time, but for each vertex $v$ of $\bs T$, it is determined in finite time whether a cone is put at $v$ or not. So the output of the algorithm is well defined.}
\begin{proof}
Let $\bs A$ be any equivariant covering of $\bs T$ by cones of height $n$. Consider a realization $(T;A)$ of $[\bs T; \bs A]$.
Let $v$ be a vertex such that $h(v)=n$. Since $A$ is a covering by cones of height $n$,
$A$ should have at least one vertex in $D(v)$ (to see this, consider the farthest leaf from $v$ in $D(v)$).
Now, for all such vertices $v$, delete the vertices in $A\cap D(v)$ from $A$ and then add $v$ to $A$. 
Let $A_1$ be the subset of $T$ obtained by doing this operation for all vertices $v$ of height $n$.
So $A_1$ is also a covering of $T$ by cones of height $n$.
Now, remove all vertices $\{v: h(v)=n\}$ and their descendants from $T$ to obtain a new one-ended tree.
Consider the same procedure for the remaining tree and its intersection with $A$. Inductively, one obtains a sequence of
subsets $A=A_0, A_1, \ldots$ of $T$ such that, for each $i$, 
$A_i$ is a covering of $T$ by cones of height $n$ which agrees with {$\bs S_T$} on
the set of vertices that are removed from the tree up to step $i$.
		
By letting $[\bs T; \bs A]$ be random, the above induction gives a sequence of equivariant subsets
$\bs A=\bs A_0, \bs A_1,\ldots$ on $\bs T$. It can be seen that the intensity of $\bs A_1$ is at most that
of $\bs A$ (this can be verified by the mass transport principle~\eqref{eq:unimodular})
\formir{and more generally, the intensity of $\bs A_{i+1}$ is at most that of $\bs A_i$; i.e.,}
$\myprob{\bs o \in \bs A_{i+1}}\leq \myprob{\bs o\in \bs A_i}$.
Also, $\lim_{i\rightarrow \infty} \bs A_i=\bs S$ as equivariant subsets of $\bs T$. 
This implies that $\myprob{\bs o\in \bs A}\geq \myprob{\bs o\in \bs S}$, hence, $\bs S$ is an optimal covering by cones of height $n$.
\end{proof}

The above algorithm can be modified to obtain an optimal ball-covering as well, which is mentioned in the following lemma. However, since it will not be used in this paper, the proof is skipped.

\begin{lemma}
	\label{lem:alg:one-ended}
	For every unimodular one-ended tree $[\bs T, \bs o]$,
	the output of {the following greedy algorithm} is an optimal equivariant $n$-covering  of $\bs T$.
	
	\begin{algorithm}[H]
		\del{\KwData{A unimodular one-ended tree $[\bs T, \bs o]$ and $n\in\mathbb N$\;}	
			\KwResult{An optimal $n$-covering of $\bs T$\;}}
		\del{$\bs S:=\emptyset$\;}
		\While{true}{
			{Let $\bs A$ be the set of vertices which are not yet covered (which might be a disconnected set)\;}
			{Let $\bs T'$ be the subtree spanned by $A$ and the shortest paths connecting the vertices in $A$\;}
			{Put balls of radii $n$ at all vertices of height $n$ in $\bs T'$\;}
			\del{{Let $\bs T'$ be the subgraph of $\bs T$ obtained by deleting the balls of radius $n$ centered at the points of $\bs S$\;}
				\For{each connected component $C$ of {$\bs T'$},}{
					\eIf{$C$ has some vertices of height $n$}{
						Add the vertices of height $n$ in $C$ to $\bs S$\;
					}{
						Add {the vertex of $C$ with the largest height} to $\bs S$\;
					}
			}}
		}
		\del{\caption{Greedy algorithm for optimal coverings of unimodular one-ended trees.}}
	\end{algorithm}
\end{lemma}

\begin{lemma}
\label{lem:eft:boundI'}
Under the above setting, one has 
		\begin{equation}
		\label{eq:lem:eft:boundI'}
		\myprob{h(\bs o) \bmod (n+1) =-1}\leq \lambda'_n \leq \myprob{h(\bs o) \bmod \floor{\frac n 2{+1}} =-1}. 
		\end{equation}
\end{lemma}
\begin{proof}
	\formir{The proof of the second inequality in~\eqref{eq:lem:eft:boundI'} is based on the construction of the following equivariant covering.}
Let $B_n:= \{v\in \bs T: h(v) \bmod n =-1 \}$ and $B'_n:=\{F^{n-1}(v): v\in B_{n}\}$. 
		The claim is that $B'_n$ is a covering of $\bs T$ by cones of height $2n-2$. 
		Let $v\in \bs T$ be an arbitrary vertex. Let $k$ be the unique integer such that $(k-1)n-1<h(v)\leq kn-1$. Let $j$ be the first
		nonnegative integer such that $h(F^j(v))\geq kn-1$ and let $w:=F^j(v)$.
		One has $0\leq j\leq n-1$. By considering the longest path in $D(w)$ from $w$ to the leaves,
		one finds $z\in D(w)$ such that {$h(z)\bmod n = -1$} and $0\leq d(w,z)\leq n-1$.
		Therefore $w$ (and hence $v$) is a descendant of $F^{n-1}(z)$. Also, $d(w, F^{n-1}(z))\leq n-1$.
		It follows that $d(v,F^{n-1}(z))\leq 2n-2$. So $v$ is covered by the cone of
		height $2n-2$ at $F^{n-1}(z)$. Since $F^{n-1}(z)\in B'_n$, it is proved that $B'_n$ gives a $(2n-2)$-covering by cones.
		It follows that $\lambda'_{2n-2}\leq \myprob{\bs o \in B'_n} \leq \myprob{\bs o\in B_n}$ 
                (where the last inequality can be verified by the mass transport principle~\eqref{eq:unimodular}).
		This implies the second inequality in~\eqref{eq:lem:eft:boundI'}.

		To prove the first inequality in~\eqref{eq:lem:eft:boundI'},
		let $\bs S$ be the optimal covering by cones of height $n$ given by {the algorithm of Lemma~\ref{lem:eft:optimalcone}}. 
		Send unit mass from each vertex $v\in \bs S$ to the first vertex in $v,F(v),\ldots, F^{n}(v)$
		which belongs to $B_{n+1}$ (if there is any). So the outgoing mass from $v$ is
		at most $\identity{\{v \in \bs S \}}$. In the next paragraph, it is proved that
		the incoming mass to each $w\in B_{n+1}$ is at least 1. {This in turn} (by the mass transport principle)
		implies that $\myprob{\bs o \in \bs S}\geq \myprob{\bs o \in B_{n+1}}$,
		which proves the first inequality in~\eqref{eq:lem:eft:boundI'}.
		
		The final step consists in proving that the incoming mass to each $w\in B_{n+1}$ is at least 1.
		If $h(w)=n$, then $w\in \bs S$ and the claim is proved. So assume $h(w)>n$. By considering the longest
		path from $w$ in $D(w)$, one can find a vertex $z$ such that $w=F^{n+1}(z)$
		and $h(z)=h(w)-(n+1)$. This implies that no vertex in $\{F(z), \ldots, F^n(z) \}$
		is in $B_{n+1}$. So to prove the claim, it suffices to show that at least one of 
		these vertices or $w$ itself lies in $\bs S$. Note that in the algorithm in
		Lemma~\ref{lem:eft:optimalcone}, at each step, the height of $w$ decreases
		by a value at least 1 and at most $n+1$ until $w$ is removed from the tree. So in
		the last step before $w$ is removed, the height of $w$ is in $\{0,1,\ldots,n\}$.
		This is possible only if in the same step of the algorithm, an element of
		$\{F(z), \ldots, F^n(z), w\}$ is added to $\bs S$. This implies the claim and the lemma is proved.
	\end{proof}
	
	Now, the tools needed to prove the main results are all available.
	
	\begin{proof}[Proof of Theorem~\ref{thm:one-ended}]
		Lemma~\ref{lem:eft:boundI'} and~\eqref{eq:eft:dim vs I'} imply that the upper and lower Minkowski
		dimensions of $\bs T$ are exactly the upper and lower decay rates of
		$\myprob{h(\bs o) \bmod n = -1}$ respectively. So one should prove that these rates
		are equal to the upper and lower decay rates of $\myprob{h(\bs o)\geq n}$ plus 1.
		
		The first step consists in showing that $\myprob{h(\bs o)=n}$ is non-increasing in $n$.
                To see this, send unit mass from each vertex $v$ to $F(v)$ if $h(v)=n$ and $h(F(v))=n+1$.
		Then the outgoing mass is at most $\identity{\{h(v)=n\}}$ and the
		incoming mass is at least $\identity{\{h(v)=n+1\}}$. The result then follows by the mass transport principle.
		This {monotonicity} implies that
		$
		n\cdot\myprob{h(\bs o) \bmod n = -1} \geq \myprob{h(\bs o)\geq n-1}.
		$
		Similarly, by monotonicity,
		\begin{eqnarray*}
			\frac n 2 \myprob{h(\bs o) \bmod n = -1} &\leq& \myprob{h(\bs o) \bmod n \in \{-1,-2,\ldots, -\ceil{\frac n 2}\}}\\
			&\leq&  \myprob{h(\bs o)\geq \floor{\frac n 2}}.
		\end{eqnarray*}
		These inequalities conclude the proof of {\eqref{eq:thm:eftMinkowski:mu} and~\eqref{eq:thm:eftMinkowski:ml}}.
		
        It remains to prove~\eqref{eq:thm:eftMinkowski:h}.
                The second inequality follows from~\eqref{eq:thm:eftMinkowski:mu} and the fact
		that $\decayu{\myprob{h(\bs o)=n}}\geq \decayu{\myprob{h(\bs o)\geq n}}+1$,
		which is not hard to see. We now prove the first inequality. 
		Fix $0<\epsilon<\alpha<\decayu{\myprob{h(\bs o)=n}}$. So
		there is a sequence $n_0<n_2<\cdots$ such that $\myprob{h(\bs o)=n_i}<n_i^{-\alpha}$ for each $i$.
		One may assume the sequence is such that $n_{i}\geq 2^i$  for each $i$.
		Now, for each $k\in \mathbb{N}$, consider the following covering of $\bs T$:
		\begin{equation*}
			\bs R_k(v):= \left\{ 
			\begin{array}{ll}
				2(n_i-n_{i-1}), & \text{if } h(v)=n_i \text{ and } i>k,\\
				2n_k, & \text{if } h(v)=n_k,\\
				0, & \text{otherwise.}
			\end{array}
			\right.
			.
		\end{equation*}
		By arguments similar to {Lemma~\ref{lem:eft:boundI'}}, it can
		be seen that $\bs R_k$ is indeed a covering.
		It is claimed that $\omid{\bs R_k(\bs o)^{\alpha-\epsilon}}\rightarrow 0$ as $k\rightarrow \infty$
		If the claim is proved, then $\dimH{\bs T}\geq \alpha-\epsilon$ and the proof
                of~\eqref{eq:thm:eftMinkowski:h} is concluded.
		Let $c:=2^{\alpha-\epsilon}$. One has
		\begin{eqnarray*}
			\omid{\bs R_k(\bs o)^{\alpha-\epsilon}} &=& cn_k^{\alpha-\epsilon}\myprob{h(\bs o)=n_k} + c\sum_{i=k+1}^{\infty} (n_i-n_{i-1})^{\alpha-\epsilon}\myprob{h(\bs o)=n_i}\\
			&\leq & cn_k^{-\epsilon} + c\sum_{i=k+1}^{\infty} (n_i-n_{i-1})^{\alpha-\epsilon} n_i^{-\alpha}.
		\end{eqnarray*}
		Therefore, it is enough to prove that 
		\begin{equation}
		\label{eq:lem:eftHausdorff}
		\sum_{i=1}^{\infty} (n_i-n_{i-1})^{\alpha-\epsilon} n_i^{-\alpha}<\infty.	
		\end{equation}
		
		It is easy to see that the maximum of the function 
		$(x-n_{i-1})^{\alpha-\epsilon}x^{-\alpha}$ over $x\geq n_{i-1}$ happens at
		$\frac{\alpha}{\epsilon}n_{i-1}$ and the maximum value is $c' n_{i-1}^{-\epsilon}$,
		where $c'=(\frac{\alpha}{\epsilon}-1)^{\alpha-\epsilon}$ is a constant. So
		the left hand side of~\eqref{eq:lem:eftHausdorff} is at most $c'\sum_{i=0}^{\infty} n_i^{-\epsilon}$,
		which is finite by the assumption $n_i\geq 2^i$. So \eqref{eq:lem:eftHausdorff} is proved and the proof is completed.
	\end{proof}

\subsection{Instances of Unimodular Trees}

This subsection discusses the dimension of some explicit unimodular trees.
More examples are given in Subsection~\ref{subsec:drainage}, in \formir{Section~\ref{sec:examples2}},
and also in \formir{the ongoing work} \cite{III} (e.g., uniform spanning forests).

\subsubsection{The Canopy Tree}
\label{subsec:canopy}

The canopy tree $C_k$ with offspring {cardinality} $k$~\cite{canopy} is constructed as follows.
Its vertex set is partitioned in levels $L_0,L_1,\ldots$. Each vertex in level $n$ is connected to $k$ vertices
in level $n-1$ (if $n\neq 0$) and one vertex (its parent) in level $n+1$.
Let $\bs o$ be a random vertex of $C_k$ such that $\myprob{\bs o\in L_n}$ is proportional to $k^{-n}$.
Then, $[C_k,\bs o]$ is a unimodular random tree.

Below, three types of metrics are considered on $C_k$.
First, consider the graph-distance metric. Given $n\in\mathbb N$, let $S:=\{v\in C_k: h(v)\geq n \}$,
where $h(v)$ is the height of $v$ defined in Subsection~\ref{subsec:one-ended}.
The set  $S$ gives an equivariant $n$-covering and $\myprob{\bs o\in S}$ is exponentially small
as $n\to \infty$. So $\dimM{C_k}=\dimH{C_k}=\infty$.

Second, for each $n$, let the length of each edge between $L_n$ and $L_{n+1}$ be $a^n$, where $a>1$ is constant. 
Let $d_1$ be the resulting metric on $C_k$. Given $r>0$, let $S_1$ be the set of vertices having distance at 
least $r/a$ to $L_0$ (under $d_1$). One can show that $S_1$ is an $r$-covering of $(C_k,d_1)$ and $\decay{\myprob{\bs o\in S_1}}=\log k/\log a$. Therefore, $\dimMl{C_k,d_1}\geq \log k/\log a$.
On the other hand, one can see that  the ball of radius $a^n$ centered at $\bs o$ (under $d_1$)
has cardinality of order $k^n$. One can then use Lemma~\ref{lem:mdp-simple} to show that
$\dimH{C_k,d_1}\leq \log k/\log a$. So $\dimM{C_k,d_1}=\dimH{C_k,d_1} = \log k/\log a$.

Third, replace $a^n$ by $n!$ in the second case and let $d_2$ be the resulting metric.
Then, the cardinality of the ball of radius $r$ centered at $\bs o$ has order less than
$r^{\alpha}$ for every $\alpha>0$. One can use Lemma~\ref{lem:mdp-simple} again to show that
$\dimH{C_k,d_2}\leq \alpha$. This implies that $\dimM{C_k,d_2} = \dimH{C_k,d_2}=0$.

\subsubsection{The Generalized Canopy Tree}
\label{subsec:canopyGeneralized}

This example generalizes the canopy tree of Subsection~\ref{subsec:canopy}.
The goal is to provide an example where the lower Minkowski dimension, the upper Minkowski dimension
and the Hausdorff dimension are all different when suitable parameters are chosen.

Fix  $p_0,p_1,\ldots> 0$ such that $\sum p_i=1$.
Let $\bs U_0,\bs U_1,\ldots$ be an i.i.d. sequence of random number in $[0,1]$ with the uniform distribution. For each $n\geq 0$, 
let $\Phi_n:=\big(\frac 1{p_n} (\mathbb Z+\bs U_n)\big)\times \{n\}$, which is a point process on the horizontal line $y=n$ in the plane.
Let $\bs o_n:=(\frac 1{p_n}\bs U_n, n)\in \Phi_n$ and $\Phi:=\cup_i \Phi_i$. 
Then, $\Phi$ is a point process in the plane which is stationary under horizontal translations.
Choose $\bs m$ independent of the sequence $(\bs U_i)_i$ such that $\myprob{\bs m=n}=p_n$ for each $n$. 
Then, let $\bs o:= \bs o_{\bs m}$.

Construct a graph $\bs T$ on $\Phi$ as follows: For each $n$, connect each $x\in \Phi_n$
to its closest point (or closest point on its right) in $\Phi_{n+1}$. Note that $\bs T$ is a forest by definition.
However, the next lemma shows that $[\bs T, \bs o]$ is a unimodular tree.

\begin{definition}
The \defstyle{generalized canopy tree} with parameters $p_0,p_1,\ldots$ is the unimodular tree $[\bs T, \bs o]$ constructed above.
\end{definition}

Note that in the case where $p_n$ is proportional to $k^{-n}$ for $k$ fixed, $[\bs T, \bs o]$
is just the ordinary canopy tree $C_k$ of Subsection~\ref{subsec:canopy}. Also, one can generalize
the above construction by letting $\Phi_n$ be a sequence of point processes
which are (jointly) stationary under horizontal translations.

\begin{lemma}
\label{lem:canopy}
One has
	\begin{enumerate}[(i)]
\item \label{lem:canopy:2} $[\Phi,\bs o]$, endowed with the Euclidean metric, is a unimodular discrete space.
\item \label{lem:canopy:1} $\bs T$ is a tree a.s. and $[\bs T, \bs o]$ is unimodular.
	\end{enumerate}
\end{lemma}

\begin{proof}
For part~\eqref{lem:canopy:2}, it is enough to show that $\Phi-\bs o$ is a point-stationary
point process in the plane {(see Example~\ref{ex:point-stationary})}. This is {proved in Lemma~\ref{lem:canopy-point-stationary}}. The main ingredients are using stationarity of $\Phi$ under horizontal
translations and the fact that $\Phi_n-\bs o_n$ is point-stationary (the proof is similar
to that of the formula for the Palm version of the superposition of stationary
point processes, e.g., in~\cite{bookScWe08}.)

To prove \eqref{lem:canopy:1}, note that $\bs T$ can be realized as an equivariant process on
$\Phi$ (see Definition~\ref{def:equivProcess} and Remark~\ref{rem:equivProcessExtraMarks}).
Therefore, by Lemma~\ref{lem:equivProcess} and Theorem~\ref{thm:metricChange},
it is enough to prove that $\bs T$ is connected a.s. Nevertheless, the same lemma 
implies that the connected component $\bs T'$ of $\bs T$ containing
$\bs o$ is a unimodular tree. Since it is one-ended, Theorem~3.9 of~\cite{eft} 
implies that the \textit{foils} $\bs T'\cap \Phi_i$ are infinite a.s. 
By noting that the edges do not cross (as segments in the plane),
one obtains that $\bs T'\cap \Phi_i$ should be the whole $\Phi_i$; hence, $\bs T'=\bs T$.
Therefore, $\bs T$ is connected a.s. and the claim is proved.
\end{proof}

\begin{proposition}
The sequence $(p_n)_n$ can be chosen such that
	\[
		\dimMl{\bs T}<\dimMu{\bs T}<\dimH{\bs T},
	\]
where $\bs T$ is endowed with the graph-distance metric.
Moreover, for any $0\leq \alpha\leq \beta \leq \gamma\leq \infty$, the sequence $(p_n)_n$ can be chosen such that 
	\[
		\dimMl{\bs T}\leq \alpha, \quad \dimMu{\bs T}=\beta, \quad \dimH{\bs T}\geq \gamma.
	\]
\end{proposition}

For example, it is possible to have $\dimM{\bs T}=0$ and $\dimH{\bs T}=\infty$ simultaneously.

\begin{proof}
$\bs T$ is a one-ended tree (see Subsection~\ref{subsec:one-ended}).
Assume the sequence $(p_n)_n$ is non-increasing. So the construction implies that there is no leaf of the tree
in $\Phi_n$ for all $n>0$. Therefore, for all $n\geq 0$, the height of every vertex in $\Phi_n$ is precisely $n$.
So by letting $q_n:=\sum_{i\geq n} p_i$, Theorem~\ref{thm:one-ended} implies that
\begin{eqnarray*}
		\dimH{\bs T} &\geq & \decayu{p_n},\\
		\dimMu{\bs T} &=& 1+\decayu{q_n},\\
		\dimMl{\bs T} &=& 1+\decayl{q_n}.
\end{eqnarray*}
For simplicity, assume $0<\alpha$ and $\gamma<\infty$ (the other cases can be treated similarly).
Define $n_0,n_1,\ldots$ recursively as follows. Let $n_0:=0$. Given that $n_i$ is defined,
let $n_{i+1}$ be large enough such that the line connecting points $(n_i,n_i^{-\beta})$ and
$(n_{i+1},n_{i+1}^{-\beta})$ intersects the graph of the function $x^{-\alpha}$ and has slope 
larger than $-n^{-\gamma}$. Now, let $q_{n_i}:=n_i^{-\beta}$ for each $i$ and define $q_n$
linearly in the interval $[n_i,n_{i+1}]$. Let $p_n:=q_{n}-q_{n+1}$. It can be seen that $p_n$
is non-increasing, $\decayl{q_n}\leq \alpha$, $\decayu{q_n}=\beta$ and $\decayu{p_n}\geq \gamma$.
\end{proof}

\subsubsection{Unimodular Eternal Galton-Watson Trees}
\label{subsec:egw}
Eternal Galton-Watson (\egw{}) trees  are defined in~\cite{eft}. Unimodular \egw{} trees
(in the nontrivial case) can be characterized as unimodular one-ended trees in which the descendants
of the root constitute a Galton-Watson tree. Also, \formir{unimodularity implies that} the latter Galton-Watson tree is necessarily critical \formir{(use the mass transport principle when sending a unit mass from each vertex to its \textit{parent})}.
Here, the trivial case that each vertex has exactly one offspring is excluded (where the corresponding \egw{} 
tree is a bi-infinite path). In particular, the \textit{Poisson skeleton tree}~\cite{objective} is an eternal Galton-Watson tree.

Recall that the offspring distribution of a Galton-Watson tree is the probability measure
$(p_0,p_1,\ldots)$ on $\mathbb Z^{\geq 0}$ where $p_n$ is the probability that the root has $n$ offsprings.

\begin{proposition}
\label{prop:EGWdimension}
Let $[\bs T, \bs o]$ be a unimodular eternal Galton-Watson tree.
If the offspring distribution has finite variance, then 
	$
	\dimM{\bs T} = \formir{\dimH{\bs T} =} 2.
	$
\end{proposition}

\begin{proof}[Proof (first part)]
	\formir{Here, it is only proved that $\dimM{\bs T}=2$. The other equality will be proved in Subsection~\ref{subsec:remainingproofs}.}
By Kesten's theorem~\cite{kesten} for the Galton-Watson tree formed by the descendants of the root,
$\lim_n n\myprob{h(\bs o)\geq n}$ exists and is positive. It follows that $\decay{\myprob{h(\bs o)\geq n}}=1$.
So the claim is implied by Theorem~\ref{thm:one-ended}.
\end{proof}

\subsection{Examples Associated with Random Walks}
\label{subsec:randomwalk}

Let $\mu$ be a probability measure on $\mathbb R^k$. Consider the \formir{(double-sided)}
simple random walk $(S_n)_{n\in \mathbb Z}$ \formir{in $\mathbb R^k$ starting from $S_0:=0$ such that}
and the jumps $S_n-S_{n-1}$ are i.i.d. with distribution $\mu$. 
In this subsection, unimodular discrete spaces are constructed based on the image and 
the zero set of this random walk and their dimensions are studied in some special cases.
The graph of the simple random walk will be studied in Subsection~\ref{subsec:srw-graph}.

\subsubsection{{The Image of the Simple Random Walk}}
\label{subsec:image}
Assume the random walk is transient; i.e., it visits every given ball only finitely many times.
It follows that the image $\Phi=\{S_n\}_{n\in\mathbb Z}$ is a random discrete subset of $\mathbb R^k$.
If no point of $\mathbb R^k$ is visited more than once \formir{(e.g., when $S_n$ is in the positive cone a.s.)},
then it can be seen that $\Phi$ is a point-stationary point process, hence, $[\Phi,0]$ is a
unimodular discrete space.  Hence, $[\Phi,0]$ is a unimodular discrete space.
In the general case, by similar arguments, one should bias the distribution of $[\Phi,0]$
by the inverse of the \textit{multiplicity} of the origin;
i.e., by $1/\card{\{n: S_n=0 \}}$, to obtain a unimodular discrete space. 
This claim can be proved by direct verification of the mass transport principle.

\formir{Below, the focus is on the case where the jumps are real-valued and strictly positive.
In this case, $\Phi$ is actually a point stationary {\em renewal process} \cite{bookFe66}.}

\begin{proposition}
\label{prop:image}
Let $\Phi:=\{S_n\}_{n\in\mathbb Z}$ be the image of a simple random walk $S$
in $\mathbb R$ \formir{starting from $S_0:=0$}. Assume the jumps $S_n-S_{n-1}$ are positive a.s. Then
\begin{enumerate}[(i)]
\item \label{part:thm:image:2} 
$\dimMl{\Phi} =\decayl{\formir{\frac 1 r \omid{S_1\wedge r}}} \formir{=} 1\wedge \decayl{\myprob{S_1>r}}$.
\item \label{part:thm:image:3} 
$\dimMu{\Phi} =\decayu{\formir{\frac 1 r \omid{S_1\wedge r}}} \leq 1\wedge \decayu{\myprob{S_1>r}}$.
\item \label{part:thm:image:h}
\formir{$\dimH{\Phi}\leq 1\wedge \decayu{\myprob{S_1>r}}$.}		
\item \label{part:thm:image:1} If $\beta:=\decay{\myprob{S_1>r}}$ exists, then
$\dimM{\Phi} = \formir{\dimH{\Phi}=} 1\wedge \beta.$
\end{enumerate}
\end{proposition}

\begin{proof}[Proof (first part)]
For every $r>0$, one has $\myprob{\Phi\cap (0,r)=\emptyset} = \myprob{S_1\geq r}$.
So the claims \formir{regarding the Minkowski dimension} are direct consequences
of Proposition~\ref{prop:lowerBoundR} \formir{and do not require the i.i.d. assumption. 
The proofs of the last two claims, 
will be given in Subsection~\ref{subsec:remainingproofs}.}
\end{proof}

The image of the nearest-neighbor simple random walk in $\mathbb Z^k$ will be studied in~\cite{III}.
It will be shown that it has dimension 2 when $k\geq 2$.
Furthermore, a \textit{doubling property} will be proved in this case. 
	
As another example, if $[\bs T, \bs o]$ is any unimodular tree such that the simple
random walk on $\bs T$ is transient a.s., then the image of the (two sided) 
simple random walk on $\bs T$ is another unimodular tree
(after biasing by the inverse of the multiplicity of the root).
The new tree is two-ended a.s., and hence, is 1-dimensional by Theorem~\ref{thm:twoEnded}.

\subsubsection{Zeros of the Simple Random Walk}
\label{subsec:zerosSRW}
\begin{proposition}
\label{prop:zeros}
Let $\Psi$ be the zero set of the symmetric simple random walk on
$\mathbb Z$ with uniform jumps in $\{\pm 1\}$. Then, $ \dimM{\Psi} = \formir{\dimH{\Psi}=} \frac 12.$
\end{proposition}

\begin{proof}
Represent $\Psi$ uniquely as $\Psi:=\{S_n: n\in \mathbb Z\}$ such that $S_0:=0$ and $S_n<S_{n+1}$ for each $n$.
Then, $(S_n)_n$ is another simple random walk and $\Psi$ is its image. The distribution of the jump $S_1$ is
explicitly computed in the classical literature on random walks (using the reflection principle).
In particular, there exist $c_1,c_2>0$ such that $c_1r^{-\frac 12}<\myprob{S_1>r}<c_2r^{-\frac 12}$
for every $r\geq 1$. Therefore, the claim is implied
\formir{by part~\eqref{part:thm:image:1} of} Proposition~\ref{prop:image}
\formir{(recall that this part of Proposition~\ref{prop:image} will be proved later)}.
\end{proof}

\subsection{A Subspace with Larger Minkowski Dimension}
\label{subsec:subspace-minkowski}

\formir{Let $\Phi\subseteq\mathbb R$ be an arbitrary point-stationary point process.
Let $S_1$ be the first point of $\Phi$ on the right of the origin.
Assume $\beta:=\decay{\myprob{S_1>r}}$ exists with $\beta<1$. 
Then Proposition~\ref{prop:image} gives that $\dimM{\Phi}=\beta$.

Let $\alpha<\beta<1$. Consider the intervals defined by consecutive points of $\Phi$.
In each such interval, say $(a,b)$, add $\ceil{(b-a)^{\alpha}}-1$ points so as to split
the interval into $\ceil{(b-a)^{\alpha}}$ equal parts.
Let $\Phi'$ denote the resulting point process (with the points of $\Phi$ and
the additional points). The assumption $\alpha<\beta$ implies
that $\omid{S_1^{\alpha}}<\infty$. Now, by biasing the distribution of $\Phi'$
by $\ceil{S_1^{\alpha}}$ and changing the origin to
a point of $\Phi'\cap [0,S_1)$ chosen uniformly at random, one obtains a point-stationary
point process $\Psi$ (see Theorem~5 in \cite{shift-coupling} and also the examples in~\cite{processes}),
(it is not a renewal process). The distribution of $\Psi$ is determined by the following equation,
where $h$ is any measurable nonnegative function:}
\begin{equation}
	\label{eq:regenerative:1}
	\omid{h(\Psi)} = \frac 1{\omid{\ceil{S_1^{\alpha}}}} \omid{\sum_{x\in \Phi'\cap [0,S_1)} h(\Phi'-x) }.
\end{equation}

\begin{proposition}
\label{prop:subset-larger}
Let $\Phi$ and $\Psi$ be as above. Then, $\Phi$ has the same distribution as
an equivariant subspace of $\Psi$ (conditioned on having the root) and 
	\[
	\dimM{\Phi}=\beta > \frac{\beta-\alpha}{1-\alpha} = \dimM{\Psi}.
	\]
\end{proposition}

Note that Theorem~\ref{thm:subsetDimension} implies that $\dimH{\Phi}=\dimH{\Psi}$.
Therefore, the proposition implies $\dimM{\Psi}<\dimH{\Psi}$.

\begin{proof}
Let $\bs A$ be the set of newly-added points in $\Psi$, which can be defined
by adding marks from the beginning and is an equivariant subset of $\Psi$.
By~\eqref{eq:regenerative:1}, one can verify that $\Psi\setminus \bs A$
conditioned on $0\not\in \bs A$ has the same distribution as $\Phi$
(see also Proposition~6 in~\cite{shift-coupling}). 
Also, by letting $c:={\omid{\ceil{S_1^{\alpha}}}}$, \eqref{eq:regenerative:1} gives
\begin{eqnarray*}
		\myprob{\Psi\cap (0,r)=0} &=& \frac 1 c \omid{\sum_{x\in \Phi'\cap [0,S_1)} 
                    \identity{\{(\Phi'-x)\cap (0,r)=\emptyset \}}}\\
		&=& \frac 1 c \omid{\ceil{S_1^{\alpha}}  \identity{\{\Phi'\cap (0,r)=\emptyset \}}}\\
		&=& \frac 1 c \omid{\ceil{S_1^{\alpha}} \identity{\{S_1/\ceil{S_1^{\alpha}}>r \}} }.
	\end{eqnarray*}
\formir{Now, by the assumption $\decay{\myprob{S_1>r}}=\beta$ and integration by parts,
it is straightforward to deduce} that
$
\decay{\myprob{\Psi\cap (0,r)=0}} = {(\beta-\alpha)}/{(1-\alpha)}.
$
Therefore, Proposition~\ref{prop:lowerBoundR} gives the claim.
\end{proof}

\formir{
\begin{remark}
The fact that $\Psi$ has a smaller Minkowski dimension than $\Phi$
means that the tail of the distribution of the jumps (or inter-arrivals) of $\Psi$
is heavier than that of the inter-arrivals of $\Phi$.
This may look surprising as the inter-arrival times of $\Psi$ are obtained
by subdividing those of $\Phi$ into smaller sub-intervals.
The explanation of this apparent contradiction is of the same nature as that of
Feller's paradox (Section I.4 of~\cite{bookFe66}). It comes from the renormalization
of {\textit{size-biased sampling}}: the typical inter-arrival of $\Psi$ has more
chance to be found in a larger inter-arrival of $\Phi$, and
this length-biasing dominates the effect of the subdivision.
\end{remark}
}

\subsection{A Drainage Network Model}
\label{subsec:drainage}

Practical observations show that large river basins have a fractal structure. For example, \cite{hack} discovered a power law relating the area and the height of river basins.
There are various ways to model river basins and their fractal properties in the literature. In particular, \cite{RoSaSa16} formalizes and proves a power law with exponent $3/2$ for a specific model called \textit{Howard's model}.
Below, the simpler model of~\cite{Ng90} is studied. {One can ask similar questions for Howard's model or other drainage network models.}

Connect each $(x,y)$ in the even lattice $\{(x,y)\in \mathbb Z^2: x+y \bmod 2 = 0\}$ to either $(x-1,y-1)$ or $(x+1,y-1)$ with equal probability in an i.i.d. manner to obtain a directed graph $\bs T$. {Note that the downward path starting at a given vertex is the rotated graph of a simple random walk. It is known that} $\bs T$ is connected and is a one-ended tree (see e.g., \cite{RoSaSa16}). Also, by Lemma~\ref{lem:equivProcess}, $[\bs T, 0]$ is unimodular. 

Note that by considering the Euclidean metric on $\bs T$, the Hausdorff dimension of $\bs T$ is 2. In the following, the graph-distance metric is considered on $\bs T$. 

\begin{proposition}
	\label{prop:drainage}
	One has $\dimM{\bs T} = \formir{\dimH{\bs T} =} \frac 3 2$ under the graph-distance metric.
\end{proposition}

\begin{proof}[Proof (first part)]
	\formir{Here, it will be proved that $\dimM{\bs T}=\frac 3 2$. The rest of the proof is postponed to Subsection~\ref{subsec:remainingproofs}.} 
	The idea is to use Theorem~\ref{thm:one-ended}. Following~\cite{RoSaSa16},
        there are two \textit{backward paths} (going upward) in the odd lattice 
        that surround the descendants $D(\bs o)$ of the origin. These two paths
        have exactly the same distribution as (rotated) graphs of independent simple
        random walks starting at $(-1,0)$ and $(1,0)$, respectively, until they hit for the first time.
        In this setting, $h(\bs o)$ is exactly the hitting time of these random walks.
        So classical results on random walks imply that $\myprob{h(\bs o)\geq n}$
        is bounded between two constant multiples of $n^{-\frac 1 2}$ for all $n$. 
	So Theorem~\ref{thm:one-ended} implies that $\dimM{\bs T} = \frac 3 2$.
\end{proof}

\subsection{Self Similar Unimodular Discrete Spaces}
\label{subsec:selfsimilar}
This section provides a class of examples {of unimodular discrete spaces} obtained by discretizing self-similar sets. 
Let $l\geq 1$ and $f_1,\ldots,f_l$ be similitudes of $\mathbb R^k$ with similarity
ratios $r_1,\ldots,r_l$ respectively (i.e., $\forall x,y\in\mathbb R^k: \norm{f_i(x)-f_i(y)}= r_i\norm{x-y}$). For every $n\geq 0$ and every string {$\sigma=(j_1,\ldots,j_n)\in\{1,\ldots,l\}^n$},
let $f_{\sigma}:=f_{j_1}\cdots f_{j_n}$. Also let $\norm{\sigma}:=n$.
Fix a point $o\in \mathbb R^k$ {(one can similarly start with a finite subset of $\mathbb R^k$ instead of a single point)}. Let $K_0:=\{o\}$ and $K_{n+1}:=\bigcup_{j} f_j(K_n)$ for each $n\geq 0$. Equivalently,
\begin{equation}
\label{eq:defKn}
K_n=\{f_{\sigma}(o): \norm{\sigma}=n\}.
\end{equation}
Recall that if $r_i<1$ for all $i$, then by contraction arguments, $K_n$ converges in the Hausdorff metric to \textit{the attractor} of $f_1,\ldots,f_l$ (see e.g., Section~2.1 of~\cite{bookBiPe17}). The attractor is the unique compact set $K\subseteq\mathbb R^k$ such that $K=\bigcup_i f_i(K)$. In addition, if the $f_i$'s satisfy the \textit{open set condition}; i.e.,
there is a bounded open set $V\subseteq \mathbb R^k$ such that
$f_i(V)\subseteq V$ and $f_i(V)\cap f_j(V)=\emptyset$ for each $i,j$,
then the Minkowski and Hausdorff dimensions of $K$ are equal to the \textit{similarity dimension},
which is the unique $\alpha\geq 0$ such that $\sum r_i^{\alpha}$=1.

{The following is the main result of this section. It introduces a discrete analogue of self-similar sets by scaling the sets $K_n$ and taking local weak limits.}

\begin{theorem}
	\label{thm:selfSimilar}
	Let $\bs o_n$ be a point of $K_n$ chosen uniformly at random, where $K_n$ is defined in (\ref{eq:defKn}). Assume that $r_i=r<1$ for all $i$ and the open set condition is satisfied. Then,
	\begin{enumerate}[(i)]
		\item 
			 $[r^{-n} K_n,\bs o_n]$ converges {weakly} to some unimodular discrete space.
		\item 
		
		The unimodular Minkowski and Hausdorff dimension of the limiting space are equal to $\alpha: = {\log l}/{\norm{\log r}}$. 
		Moreover, it has positive and finite \formir{$\alpha$-dim H-size}.
	\end{enumerate}
\end{theorem}

The proof is given at the end of this subsection.
In fact, a point process  $\Psi$ in $\mathbb R^k$ will be constructed such that $[r^{-n} K_n,\bs o_n]$ converges weakly to $[\Psi,o]$. In addition, $\Psi-o$ is point-stationary. It can also be constructed directly by the algorithm in Remark~\ref{rem:alg:selfSimilar} below. 

\begin{definition}
	{The unimodular discrete space in Theorem~\ref{thm:selfSimilar}} is called a {\textbf{self similar unimodular discrete space}}.
\end{definition}

{It should be noted that self similar unimodular discrete spaces depend on the choice of the initial point $o$ in general.}

The following are examples of unimodular self~similar discrete spaces. The reader is also invited to construct a unimodular discrete version of the Sierpinski carpet similarly.

\begin{example}
	{If $f_1(x):=x/2$ and $f_2(x):=(1+x)/2$, then the limiting space is just $\mathbb Z$. Similarly,} the lattice $\mathbb Z^k$ and the triangular lattice in the plane are self~similar unimodular discrete spaces.
\end{example}

\begin{example}[Unimodular Discrete Cantor Set]
	\label{ex:cantor}
	{{Start with two points $K_0:=\{0,1\}$}. Let $f_1(x):=x/3$ and $f_2(x):=(2+x)/3$. Then, $K_n$ is the set of {the interval ends}} in the $n$-th step of the definition of the Cantor set. Here, it is easy to see that the random set $\Psi_n:=3^n(K_n-\bs o_n)\subseteq\mathbb Z$ converges weakly to the random set $\Psi\subseteq\mathbb Z$ defined as follows: $\Psi:=\cup_n \bs T_n$,
	where $\bs T_n$ is defined by letting $\bs T_0:=\{0,{\pm 1}\}$ and $\bs T_{n+1}:=\bs T_n\cup (\bs T_n\pm 2\times 3^n)$, 
	where the sign is chosen i.i.d., each sign with probability $1/2$. 
	Note that each $\bs T_n$ has the same distribution as $\Psi_n$, but the sequence $\bs T_n$ is nested. In addition, since $\bs o_n$ is chosen uniformly, $\Psi_n$ and $\Psi$ are point-stationary point processes, and hence $[\Psi,0]$ is unimodular {(a deterministic discrete Cantor set exists in the literature which is not unimodular)}. Theorem~\ref{thm:selfSimilar} implies that $\dimM{\Psi}=\dimH{\Psi}={\log 2}/{\log 3}$.
\end{example}

\begin{example}[Unimodular Discrete Koch Snowflake]
	\label{ex:koch}
	Let $C_n$ be the set of points in the $n$-th step of the construction of the Koch snowflake. Let $\bs x_n$ be a random point
	of $C_n$ chosen uniformly and $\Phi_n:= 3^n(C_n-\bs x_n)$. It can be seen that $\Phi_n$ tends weakly to a random discrete subset
	$\Phi$ of the triangular lattice which is almost surely a bi-infinite path (note that the cycle disappears in the limit). It can be seen that $\Phi$ can be obtained by Theorem~\ref{thm:selfSimilar}. In this paper, $\Phi$ is called the \defstyle{unimodular discrete Koch snowflake}. Also, Theorem~\ref{thm:selfSimilar} implies that $\dimM{\Phi}=\dimH{\Phi}={\log 4}/{\log 3}$.
	\\
	In addition, $\Phi$ can be constructed explicitly as
	$\Phi:=\cup_n \bs T_n$, where $\bs T_n$ is a random finite path in
	the triangular lattice with distinguished end points $\bs A_n$ and $\bs B_n$ defined inductively as follows:
	Let $\bs T_1:=\{\bs A_1, \bs B_1\}$, where $\bs A_1$ is the origin and $\bs B_1$ is a neighbor of the origin
	in the triangular lattice chosen uniformly at random. For each $n\geq 1$, given $(\bs T_n, \bs A_n, \bs B_n)$,
	let $(\bs T_{n+1}, \bs A_{n+1}, \bs B_{n+1})$ be obtained by attaching to $\bs T_n$ three isometric copies of
	itself as shown in Figure~\ref{fig:koch}. There are 4 ways to attach the copies and one of them should be chosen
	at random  with equal probability (the copies should be attached to $\bs T_n$ relative to the position of
	$\bs A_n$ and $\bs B_n$). It can be seen that no points overlap.

	\begin{figure}
		\begin{center}
			\includegraphics[width=.45\textwidth]{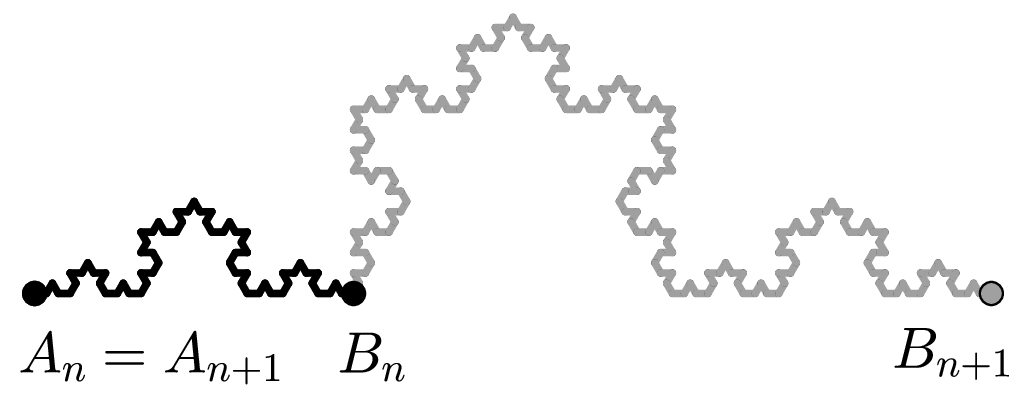}
			\includegraphics[width=.45\textwidth]{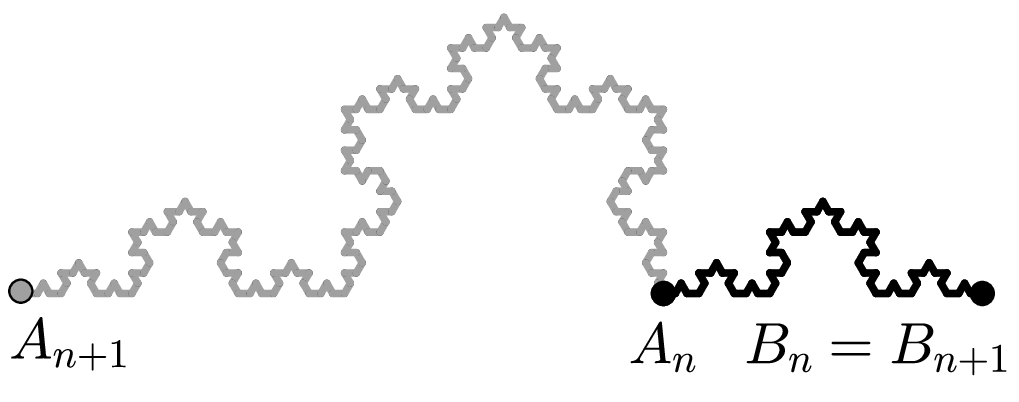}
			\includegraphics[width=.45\textwidth]{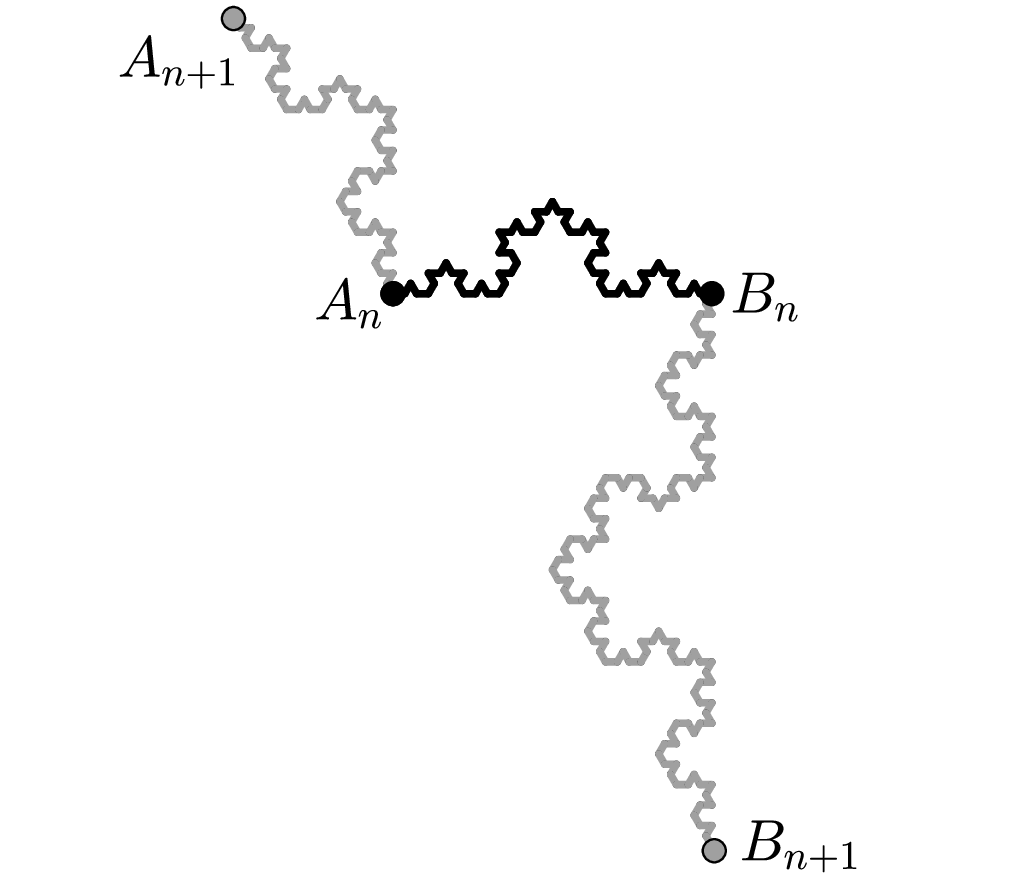}
			\includegraphics[width=.45\textwidth]{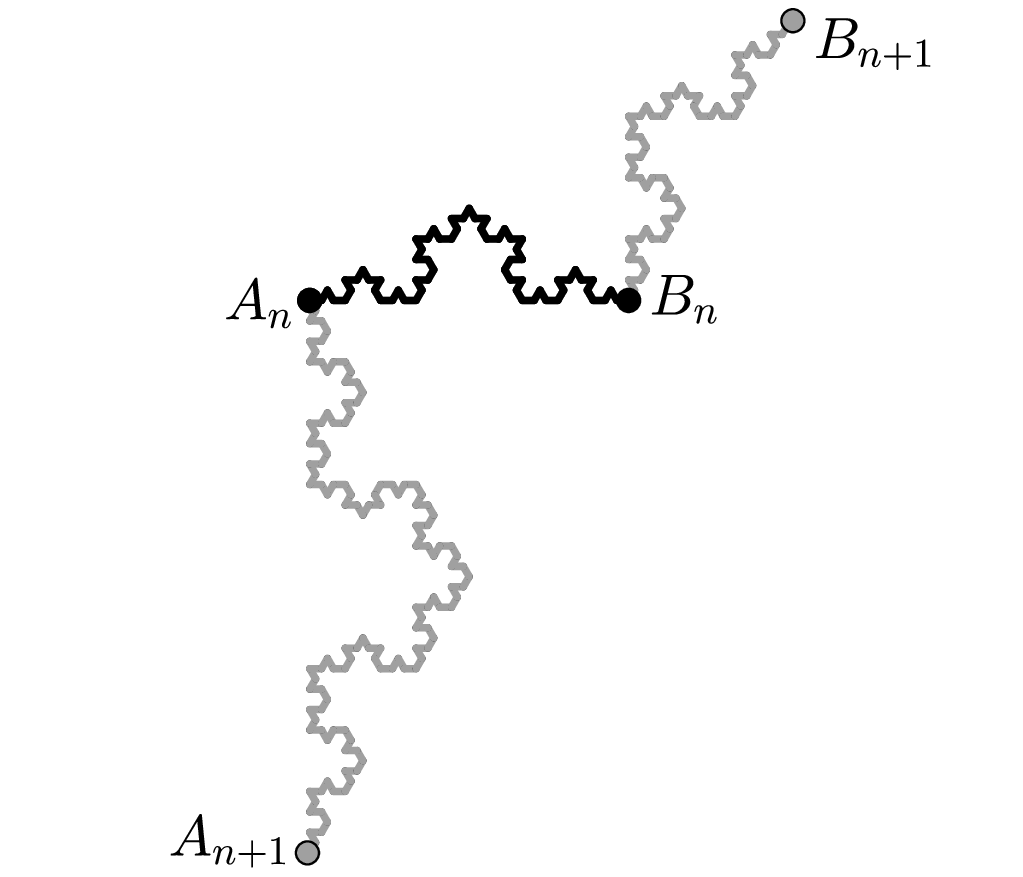}
			\caption{Four ways to attach 3 isometric copies to $\bs T_n$ in the construction of the unimodular discrete Koch snowflake, where each copy is a rotated/translated version of $\bs T_n$ (relative to $\bs A_n$ and $\bs B_n$). {Here, $\bs T_n$ is shown in black.}}
			\label{fig:koch}
		\end{center}
	\end{figure}

\end{example}

\begin{remark}
	\label{rem:selfsimilar-not-equal}
	If the $r_i$'s are not all equal, the guess is that there is no scaling of the sequence
	$[K_n, \bs o_n]$ that converges to a nontrivial unimodular discrete space (which is not a single point).
	This {has been} {verified by the authors} in the case $o\in V$. In this case, by letting $\bs a_n$ be the
	distance of $\bs o_n$ to its closest point in $K_n$, it is shown that for any $\epsilon>0$,
	$\myprob{\bs a_n/(\bar r)^n<\epsilon}\rightarrow\frac 1 2$ and
	$\myprob{\bs a_n /(\bar r)^n>\frac 1{\epsilon}}\rightarrow\frac 1 2$, where $\bar r$
	is the geometric mean of $r_1,\ldots, r_l$. This implies the claim {(note that the counting measure matters for convergence; e.g., $\{0,\frac 1 n\}$ does not converge to $\{0\}$)}.
\end{remark}

{To prove Theorem~\ref{thm:selfSimilar}, it is useful to consider the following nested version of the sets $K_n$} (note that $K_n$ is not necessarily contained in $K_{n+1}$, unless $o$ is a fixed point of some $f_i$).
Let $\bs u_1,\bs u_2,\ldots$ be i.i.d. uniform random numbers in
$\{1,\ldots,l\}$ and $\bs \delta_n:=(\bs u_n,\ldots, \bs u_1)$. Let
$
\bs o'_n:=f_{\bs \delta_n}(o).
$
Let $\hat{\bs K}_n:=f_{\bs \delta_n}^{-1} K_n ={f_{\bs u_1}^{-1}\cdots f_{\bs u_n}^{-1}} K_n$.
The chosen order of the indices in $\bs \delta_n$ ensures that 
$
\hat{\bs K}_n\subseteq \hat{\bs K}_{n+1}
$
for all $n$.
It is easy to see that $[\hat{\bs K}_n, o]$ has the same distribution as $[r^{-n} K_n, \bs o'_n]$.
For $v\in \hat{\bs K}_n$, let
\[
\bs w_n(v):=\card{\{\sigma: \norm{\sigma}=n,  f_{\sigma}(o)=f_{\bs \delta_n}(v) \}}.
\]
One has $\bs w_n(v)\leq \bs w_{n+1}(v)$. Note that in the case $o\in V$, $\bs w_n(\cdot)=1$ and the arguments are much simpler. The reader can assume this at first reading.

In the following,
for $x\in\mathbb R^k$, $B_r(x)$ represents the closed ball of radius $r$ centered at $x$ in $\mathbb R^k$.

\begin{lemma}
	\label{lem:selfSimilarConstruction}
	Let $\hat{\bs K}:=\cup_n \hat{\bs K}_n$ and $\bs w(v):=\lim_n \bs w_n(v)$ for $v\in \hat{\bs K}$.
	\begin{enumerate}[(i)]
		\item \label{lem:selfSimilarConstruction:1} $\bs w(\cdot)$ is uniformly bounded.
		\item Almost surely, $\hat{\bs K}$ is a discrete set.
		\item The distribution of $[\hat{\bs K}, o]$,
		biased by $1/\bs w(o)$, {is the limiting distribution alluded to in Theorem~\ref{thm:selfSimilar}.}
	\end{enumerate}
\end{lemma}

\begin{proof}
	\textit{(i).}	
	Assume $f_{\sigma_1(o)}=\cdots=f_{\sigma_l}(o)$ and $\norm{\sigma_j}=n$ for each $j\leq l$. Let $D$ be a fixed number such that $V$ intersects $B_{D}(o)$. Now, the sets $f_{\sigma_j}(V)$ for $1\leq j\leq l$ are disjoint and intersect a common ball of radius $Dr^n$. Moreover, each of them contains a ball of radius $ar^n$ and each is contained in a ball of radius $br^n$ (for some fixed $a,b>0$). Therefore, Lemma~2.2.5 of~\cite{bookBiPe17} implies that $l\leq (\frac {D+2b}{a})^k=:C$. This implies that $\bs w_n(\cdot)\leq C$ a.s., hence $\bs w(\cdot)\leq C$ a.s.
	
	\textit{(ii).} Let $D$ be arbitrary as in the previous part. Assume $f_{\bs \delta_n}^{-1}f_{\sigma_j(o)}\in B_D(o)$ for $j=1,\ldots,l$. Now, for $j=1,\ldots,l$, the sets $f_{\sigma_j}(V)$ are disjoint and intersect a common ball of radius $2Dr^n$. As in the previous part, one obtains $l\leq (\frac {2D+2b}{a})^k$. Therefore, $\card{N_D(o)\leq (\frac {2D+2b}{a})^k}$ a.s. Since this holds for all large enough $D$, one obtains that $\hat{\bs K}$ is a discrete set a.s.

	\textit{(iii).} Note that the distribution of $\bs o'_n$ is just the distribution of $\bs o_n$ biased by the \textit{multiplicities} of the points in $K_n$. It follows that biasing the distribution of $[\hat{\bs K}_n, o]$ by $1/{\bs w_n(o)}$ gives just the distribution of $[r^{-n}K_n, \bs o_n]$. The latter is unimodular since $\bs o_n$ is uniform in $K_n$. {So} the distribution of $[\hat{\bs K}, o]$ biased by $ 1/{\bs w(o)}$ is also unimodular and satisfies the claim of Theorem~\ref{thm:selfSimilar}. 
\end{proof}

\begin{proof}[Proof of Theorem~\ref{thm:selfSimilar}]
	Convergence is proved in Lemma~\ref{lem:selfSimilarConstruction}.
	The rest of the proof is base on the construction of a sequence of equivariant coverings of $\hat{\bs K}$. In this proof, with an abuse of notation, the dimension of ${\hat{\bs K}}$ means the dimension of the unimodular space obtained by biasing the distribution of $\hat{\bs K}$ by $1/\bs w(o)$ (see Lemma~\ref{lem:selfSimilarConstruction}).
	Let $D>\mathrm{diam}(K)$ be given, where $K$ is the attractor of $f_1,\ldots, f_l$.
	Let $m>0$ be large enough so that $\mathrm{diam}(K_m)<{D}$.
	Note that each element in $\hat{\bs K}$ can be  written as $f_{\bs \delta_n}^{-1}f_{\sigma}(o)$
	for some $n$ and some string $\sigma$ of length $n$. 
	Let $\bs \gamma_m$ be a string of length $m$ chosen uniformly at random and independently of other variables. 
	For an arbitrary $n$ and a string $\sigma$ of length $n$, let
	\begin{eqnarray*}
		\bs U_{\sigma} &:=& f_{\bs \delta_{n+m}}^{-1}f_{\sigma}(K_{m}),\\	
		\bs z_{\sigma} &:=& f_{\bs \delta_{n+m}}^{-1}f_{\sigma} {f_{\bs \gamma_m}(o)}. 
	\end{eqnarray*}
	Note that $\bs U_{\sigma}\subseteq\hat{\bs K}$ is always a scaling of $K_m$
	with ratio $r^{-m}$ and $\bs z_{\sigma}\in \bs U_{\sigma}$.
	Now, define the following covering of $\hat{\bs K}$:
	\[
	\bs R_m(v):=\left\{
	\begin{array}{ll}
	Dr^{-m}, & \text{if } v=\bs z_{\sigma} \text{ for some } \sigma,\\
	0, & \text{otherwise}.
	\end{array}
	\right.
	\]
	It can be seen that $\bs R_m$ gives an equivariant covering. 
	Also, note that $\bs R_m(o)>0$ if and only if $f_{\sigma}f_{\bs \gamma_m}(o) = f_{\bs \delta_{n+m}}(o)$
	for some $n$ and some string $\sigma$ of length $n$. 
	Let $A_{n,m}(o)$ be the set of possible outcomes for $\bs \gamma_m$ such that there exists
	a string $\sigma$ of length $n$ such that the last equation holds. One can see that this
	set is increasing with $n$ and deduce that $\bs w_m(o)\leq \card{A_{n,m}(o)}\leq \bs w_{n+m}(o)$.
	By letting $\bs w'_m(o):=\card{\cup_n A_{n,m}(o)}$, it follows that 
	$\bs w_m(o)\leq \bs w'_m(o)\leq \bs w(o)$.
	According to the above discussion, $\bs R_m(o)>0$ if and only if $\bs \gamma_m \in  \cup_n A_{n,m}(o)$. So
	\[
	\probCond{\bs R_m(o)>0}{\bs u_0, \bs u_1,\ldots} = \bs w'_m(o) r^{m\alpha}. 
	\]
	Therefore, by considering the biasing that makes $\hat{\bs K}$ unimodular, one gets   
	\begin{equation}
	\label{eq:thm:selfSimilar:1}
	\omid{\frac 1{\bs w(o)}\identity{\{ \bs R_m(o)>0 \}} } {= \omid{\frac{\bs w'_m(o)r^{m\alpha}}{\bs w(o)}}} \leq r^{m\alpha}.
	\end{equation}
	Since the balls in the covering have radius $Dr^{-m}$, one gets $\dimMl{\hat{\bs K}}\geq\alpha$. 
	
	On the other hand, by~\eqref{eq:thm:selfSimilar:1} and monotone convergence, one finds that
	\[
	\omid{\frac 1{\bs w(o)}\identity{\{ \bs R_m(o)>0 \}} } \geq \frac 12 r^{m\alpha},
	\]
	for large enough $m$. 
	Similar to the proof of part~\eqref{lem:selfSimilarConstruction:1} of Lemma~\ref{lem:selfSimilarConstruction},
	one can show that the sequence of coverings $\bs R_m$  (for $m=1,2,\ldots$) is uniformly bounded. 
	Therefore, Lemma~\ref{lem:nearlyDisjoint} implies that $\dimM{\hat{\bs K}}=\alpha$. 
	Moreover, since $\omid{\bs R_m(o)^{\alpha}/\bs w(o)}$ is bounded (by $D^{\alpha}$),
	one can get that $\measH{\alpha}(\hat{\bs K})>0$.
	
	Lemma~\ref{lem:mdp-simple} will be used to bound the Hausdorff dimension.
	Let $D>1$ be arbitrary. Choose $m$ such that $r^{-m}\leq D < r^{-m-1}$.
	By Lemma~\ref{lem:selfSimilarConstruction}, there are finitely many points 
	in $\hat{\bs K}\cap B_D(o)$. Therefore, one finds $n$ such that $\hat{\bs K}\cap B_D(o)=\hat{\bs K}_{n+m}\cap B_D(o)$. 
	It follows that the sets $\{\bs U_{\sigma}:\norm{\sigma}=n\}$ cover $\hat{\bs K}_{n+m}$.
	Now, assume $\sigma_1,\ldots,\sigma_k$ are strings of length $n$ such that
	$\bs U_{\sigma_i}$ are distinct and intersects $B_D(o)$. One obtains that
	\begin{equation}
	\label{eq:thm:selfSimilarDim:1}
	\card{B_D(o)\cap\hat{\bs K}}\leq \sum_{j=1}^k\card{B_D(o)\cap \bs U_{\sigma_j}} \leq k l^m = k r^{-\alpha m} \leq k D^{\alpha}. 
	\end{equation}
	Consider the sets $\bs V_{\sigma_j}:=f_{\bs \delta_{n+m}}^{-1}f_{\sigma_j}(V)$ which are disjoint
	(since $\sigma_j$'s have the same length). Note that if $\epsilon>\mathrm{diam}(V\cup\{o\})$ is fixed,
	then the $\epsilon$-neighborhood of $V$ contains $K_m$. Therefore, all $\bs V_{\sigma_j}$'s
	intersect a common ball of radius $D+\epsilon r^{-m} \leq (1+\epsilon) D$. Moreover, each of them contains
	a ball of radius $ar^{-m}\geq a r D$ and is contained in a ball of radius $br^{-m} \leq b D$
	(for some $a,b>0$ not depending on  $D$). Therefore, Lemma~2.2.5 of~\cite{bookBiPe17}
	implies that $k\leq(\frac {(1+\epsilon) + 2b}{ar})^k$. Therefore, \eqref{eq:thm:selfSimilarDim:1} implies that
	\[
	\card{B_D(o)\cap\hat{\bs K}}\leq C D^{\alpha}, \quad \text{a.s.}
	\]
	Therefore, Lemma~\ref{lem:mdp-simple} implies that $\dimH{\hat{\bs K}}\leq \alpha$. 
	Moreover, the proof of the lemma shows that $\measH{\alpha}(\hat{\bs K})<\infty$. This completes the proof. 
\end{proof}

\begin{remark}
	\label{rem:alg:selfSimilar}
	{Motivated by Examples~\ref{ex:cantor} and~\ref{ex:koch}, it can be seen that every unimodular self similar discrete space can be constructed by successively attaching copies of a set to itself. This is expressed in the following algorithm.}
	\begin{algorithm}
		$\hat{\bs K}_0:=\{o\}$\;
		Let $g_0$ be the identity map\;
		Choose i.i.d. random numbers $i_1,i_2,\ldots$ uniformly in $\{1,\ldots,l\}$\;
		\For{$n=1,2,\ldots$}{
			let $\hat{\bs K}_n$ consist of $l$ isometric copies of $\hat{\bs K}_{n-1}$ as follows
			\[\hat{\bs K}_{n}:= \bigcup_{j=1}^{l} g_{n-1} f_{i_n}^{-1} f_j g_{n-1}^{-1} (\hat{\bs K}_{n-1});\]
			Let $g_{n}:=g_{n-1} f_{i_n}^{-1}$\;
		}
	\end{algorithm}
\end{remark}

\subsection{Notes and Bibliographical Comments}
Some of the examples in this section, listed below, are motivated by analogous examples in the continuum setting. In fact, the unimodular dimensions of these examples are equal to the ordinary dimensions of the analogous continuum examples. This connection will be discussed further in \cite{III} via scaling limits.

Proposition~\ref{prop:EGWdimension} is inspired by the dimension of the Brownian continuum random tree (see\formir{~\cite{HaMi04} or }Theorem~5.5 of~\cite{DuLega05}),
which is the scaling limit of Galton-Watson trees conditioned to be large.
The zero set of the simple random walk (Proposition~\ref{prop:zeros}) is analogous to
the zero set of Brownian motion. Self-similar unimodular discrete spaces are inspired by
continuum self-similar sets (see e.g., Section~2.1 of~\cite{bookBiPe17}) as discussed
in Subsection~\ref{subsec:selfsimilar}. 

\section{\formir{The Unimodular Mass Distribution Principle and Billingsley Lemma}}
\label{sec:volumeGrowth}

Let 
$D$ be a discrete space and $o\in D$. The \defstyle{upper and lower (polynomial) \formir{volume} growth rates} of $D$ are 
\begin{eqnarray*}
	\growthu{\card{N_r(o)}} &=& \limsup_{r\rightarrow\infty} {\log \card{N_r(o)}}/{\log r},\\
	\growthl{\card{N_r(o)}} &=& \liminf_{r\rightarrow\infty} {\log \card{N_r(o)}}/{\log r}.
\end{eqnarray*}
$D$ has \defstyle{polynomial growth} if $\growthu{\card{N_r(o)}}<\infty$.
\formir{These definitions have various other names in the literature (e.g.,
mass dimension~\cite{BaTa89}, fractal dimension, or growth degree);
volume growth will be used in the present paper since it is common
in the context of graphs and discrete groups.}

If the upper and lower \formir{volume} growth rates are equal, the common value is
called the \defstyle{\formir{volume} growth rate} of $D$. 
Note that for $v\in D$, one has $N_r(o)\subseteq N_{r+c}(v)$ and $N_r(v)\subseteq N_{r+c}(o)$, where $c:=d(o,v)$.
This implies that $\growthu{\card{N_r(o)}}$ and $\growthl{\card{N_r(o)}}$ do not depend on the choice of the point $o$. 

In various situations in this paper, some \textit{weight} in $\mathbb R^{\geq 0}$ will
be assigned to each point of $D$. In these cases, it is natural to redefine the
\formir{volume} growth rate by considering the weights; i.e., by replacing $\card{N_r(o)}$
with the sum of the weights of the points in $N_r(o)$. This will be formalized below using the notion of
\textit{equivariant~processes}.  
Recall that an equivariant process should be defined for all discrete spaces $D$.
However, if a random \rooted{} discrete space $[\bs D, \bs o]$ is considered,
it is enough to define weights in almost every realization
(see Subsection~\ref{subsec:process} for more on the matter).
Also, given $D$, the weights are allowed to be random.

\begin{definition}
\label{def:weight}
An \defstyle{equivariant weight function} $\bs w$ is an equivariant process (Definition \ref{subsec:process})
with values in $\mathbb R^{\geq 0}$. For all discrete spaces $D$ and $v\in D$, the (random)
value $\bs w(v):=\bs w_{D}(v)$ is called the \defstyle{weight} of $v$.
Also, for $S\subseteq D$, let  $\bs w(S):=\bs w_D(S):=\sum_{v \in S}\bs w(v).$
\end{definition}
The last equation shows that one could also call $\bs w$ an \textit{equivariant measure}.

Assume $[\bs D, \bs o]$ is a unimodular discrete space (Subsection~\ref{subsec:unimodular}).
Lemma~\ref{lem:equivProcess} shows that $[\bs D, \bs o; \bs w_{\bs D}]$ is a random \rooted{}
marked discrete space and is unimodular.

%

\subsection{Unimodular Mass Distribution Principle}

\begin{theorem}[Mass Distribution Principle]
	\label{thm:mdp-simple}
	Let $[\bs D, \bs o]$ be a unimodular discrete space. 
	\begin{enumerate}[(i)]
		\item Let $\alpha,c,M>0$ and assume there exists an equivariant weight function $\bs w$ such that
		$\forall r\geq M: \bs w(N_r(\bs o))\leq c r^{\alpha}$, a.s.
		Then, $\contentH{\alpha}{M}(\bs D)$ defined in~\xeq{\ref{eq:hcontent}} satisfies
		$\contentH{\alpha}{M}(\bs D)\geq \frac 1 c \omid{\bs w(\bs o)}.$
		\item If in addition, 
		\formir{$\bs w_{\bs D}(\bs o)>0$ with positive probability,}
		then $\dimH{\bs D}\leq\alpha.$
	\end{enumerate}
\end{theorem}

\begin{proof}
	Let $\bs R$ be an arbitrary equivariant covering such that $\bs R(\cdot)\in \{0\}\cup [M,\infty)$ a.s.
	By the assumption {on $\bs w$}, $\bs R(\bs o)^{\alpha}\geq \frac 1 c {\bs w(N_{\bs R}(\bs o))}$ a.s. Therefore, 	
	\begin{equation}
	\label{eq:thm:mdp-simple}
	\omid{\bs R(\bs o)^{\alpha}}\geq \frac 1{c} \omid{\bs w(N_{\bs R}(\bs o))}.
	\end{equation}
	Consider the \textit{independent coupling} of $\bs w$ and $\bs R$; i.e.,
        for each deterministic discrete space $G$, choose $\bs w_G$ and $\bs R_G$ 
        independently (see Definition~\ref{def:equivProcess}). Then, it can be seen 
        that the pair $(\bs w, \bs R)$ is an equivariant process.
        So by Lemma~\ref{lem:equivProcess}, $[\bs G, \bs o; (\bs w, \bs R)]$ is unimodular.
        Now, the mass transport principle~(\ref{eq:unimodularMarked})
        can be used for  $[\bs G, \bs o; (\bs w, \bs R)]$.
	By letting $g(u,v):= \bs w(v)\identity{\{v\in N_{\bs R}(u)\}}$,
	one gets $g^+(\bs o) = \bs w(N_{\bs R}(\bs o))$. Also, 
	$g^-(\bs o) = \bs w(\bs o) \sum_{u\in\bs D} \identity{\{\bs o\in N_{\bs R}(u)\}} \geq \bs w(\bs o)$ a.s.,
	where the last inequality follows from the fact that $\bs R$ is a covering.
	Therefore, the {mass transport principle} implies that
	$\omid{\bs w(N_{\bs R}(\bs o))} \geq \omid{\bs w(\bs o)}$
	(recall that by convention, $N_{\bs R}(\bs o)$ is the empty set when $\bs R(\bs o)=0$).
	So by~\eqref{eq:thm:mdp-simple}, one gets $\omid{\bs R(\bs o)^{\alpha}} \geq \frac 1 c \omid{\bs w(\bs o)}$.
	Since this holds for any $\bs R$, one gets that 
	$\contentH{\alpha}{M}(\bs D)\geq \frac 1{c} \omid{\bs w(\bs o)}$ and the first claim is proved. 
	If \formir{in addition, $\bs w(\bs o)>0$} with positive probability, 
        then $\omid{\bs w(\bs o)}>0$. Therefore,
	$\contentH{\alpha}{1}(\bs D)>0$ and the second claim is proved.
\end{proof}

\subsection{Unimodular Billingsley Lemma}
\label{sec:bounds}

The main result of this subsection is Theorem~\ref{thm:billingsley}.
It is based on Lemmas~\ref{lem:MDPstronger2} and~\ref{lem:lowerboundiid} below. Lemma~\ref{lem:MDPstronger2}
is a stronger version of the mass distribution principle (Theorem~\ref{thm:mdp-simple}).

\begin{lemma}[{An Upper Bound}]
\label{lem:MDPstronger2}
Let $[\bs D, \bs o]$ be a unimodular discrete space and $\alpha\geq 0$.
\begin{enumerate}[(i)]
\item If there exist $c\geq 0$ and $\bs w$ is an equivariant weight function such that 
$
{\limsup_{r\rightarrow\infty}{\bs w(N_r(\bs o))}/{r^{\alpha}} \leq c, \quad \mathrm{a.s.},}
$
then
$
\contentH{\alpha}{\infty}(\bs D) \geq {\frac 1{2^{\alpha}c} \omid{\bs w(\bs o)}}.
$
\item \label{lem:MDPstronger2:2}
In addition, if \formir{$\bs w_{\bs D}(\bs o)>0$ with positive probability,}
then $\dimH{\bs D}\leq\alpha$.
\end{enumerate}
\end{lemma}

\begin{proof}
	Let $c'>c$ be arbitrary. The assumption implies that 
	$ \sup \{r\geq 0: \bs w(N_r(\bs o))>c'r^{\alpha} \}<\infty $ a.s.
	For $m\geq 1$, let $A_m:=\{v\in \bs D: \forall r\geq m: \bs w(N_r(v))\leq c' r^{\alpha} \},$
	which is an increasing sequence of equivariant subsets. So
	\begin{equation}
	\label{eq:limeqp2}
	\lim_{m\rightarrow\infty} \myprob{\bs o\in A_m} {= 1}.
	\end{equation}
	
	Let $\bs R$ be an equivariant covering such that $\bs R(\cdot)\in \{0\}\cup [m,\infty)$ a.s. One has
	\begin{equation}
	\label{eq:lem:MDPstronger2:1}
	\omid{\bs R(\bs o)^{\alpha}}\geq \omid{\bs R(\bs o)^{\alpha} \identity{\{N_{\bs R}(\bs o)\cap A_m \neq \emptyset \}}}.
	\end{equation}
	If $N_{\bs R}(\bs o)\cap A_m \neq \emptyset$, then $\bs R(\bs o)\neq 0$ and hence $\bs R(\bs o)\geq m$. 
	In the next step, assume that this is the case.
	Let $v$ be an arbitrary point in $N_{\bs R}(\bs o)\cap A_m$. By the definition of $A_m$,
	one gets that for all  $r\geq m$, $\bs w(N_r(v))\leq c' r^{\alpha}.$
	Since $N_{\bs R(\bs o)}(\bs o)\subseteq N_{2\bs R(\bs o)}(v)$, it follows that
	$\bs w(N_{\bs R}(\bs o))\leq \bs w(N_{2\bs R(\bs o)}(v))\leq 2^{\alpha}c'\bs R(\bs o)^{\alpha}$.
	Therefore, \eqref{eq:lem:MDPstronger2:1} gives
	\begin{equation}
	\label{eq:lem:MDPstronger2:2}
	\omid{\bs R(\bs o)^{\alpha}}\geq \frac 1{2^{\alpha}c'} \omid{\bs w(N_{\bs R}(\bs o))\identity{\{N_{\bs R}(\bs o)\cap A_m \neq \emptyset \}}}.
	\end{equation}
	
	By letting $g(u,v):= \bs w(v)\identity{\{v\in N_{\bs R}(u)\}} \identity{\{N_{\bs R}(u)\cap A_m \neq \emptyset\}}$,
	one gets that $g^+(\bs o) = \bs w(N_{\bs R}(\bs o))\identity{\{N_{\bs R}(\bs o)\cap A_m \neq \emptyset\}}$.
	Also, since there is a ball $N_{\bs R}(u)$ that covers $\bs o$ a.s., one has
	$g^-(\bs o) \geq \bs w(\bs o)\identity{\{\bs o \in A_m\}}$ a.s.
	Therefore, the mass transport principle~(\ref{eq:unimodularMarked})
	and~\eqref{eq:lem:MDPstronger2:2} imply that
	$
	\omid{\bs R(\bs o)^{\alpha}}\geq \frac 1{2^{\alpha}c'} \omid{\bs w(\bs o)\identity{\{\bs o\in A_m\}}}.
	$
	This implies that 
	$\contentH{\alpha}{m}(\bs D)\geq \frac 1{2^{\alpha}c'} \omid{\bs w(\bs o)\identity{\{\bs o\in A_m\}}}$. 
	Using~\eqref{eq:limeqp2} and letting $m$ tend to infinity gives
	$\contentH{\alpha}{\infty}(\bs D) \geq \frac 1{2^{\alpha}c'} \omid{\bs w(\bs o)}$.
	Since $c'>c$ is arbitrary, the first claim is proved.
	
	Part~\eqref{lem:MDPstronger2:2} {\formir{is proved by the same argument as the corresponding statement
			in Theorem~\ref{thm:mdp-simple}. The proof leverages Lemma \ref{lem:Hmeas}.}}
\end{proof}

\begin{lemma}[{Lower Bounds}]
	\label{lem:lowerboundiid}
	Let $[\bs D, \bs o]$ be a unimodular discrete space,  $\alpha\geq 0$ and $c>0$.
	Let $\bs w$ be an arbitrary equivariant weight function such that $\omid{\bs w(\bs o)}<\infty$.
	\begin{enumerate}[(i)]
		\item \label{part:lem:lowerbound:all}
		If $\exists r_0: \forall r\geq r_0: \bs w(N_r(\bs o))\geq c r^{\alpha}$ a.s.,
		then $\dimMl{\bs D}\geq \alpha$. 
		\item \label{part:lem:lowerbound:liminf}
		If {$\growthl{\bs w(N_r(\bs o))}\geq \alpha$}
		a.s., then $\dimH{\bs D}\geq \alpha$.
		\item \label{part:lem:lowerbound:E} 
		If {$\lim_{\delta \downarrow 0} \liminf_{r\rightarrow\infty} \myprob{\bs w(N_r(\bs o))\leq \delta r^{\alpha}}=0$,}
		then $\dimH{\bs D}\geq \alpha$. 
		\item \label{part:lem:lowerbound:exp}
		If $\decayl{\omid{\mathrm{exp}\left(-\frac{\bs w(N_n(\bs o))}{n^{\alpha}}\right)}}\geq \alpha$,
		then $\dimMl{\bs D}\geq \alpha$.
	\end{enumerate}
\end{lemma}

\begin{proof}
	The proofs of the first two parts are very similar. The second part is proved first.
	
	\eqref{part:lem:lowerbound:liminf}.
	Let $\beta$, $\gamma$ and $\kappa$ be such that $\gamma < \beta < \kappa < \alpha$. 
	Fix $n\in\mathbb N$. Let $\bs S= \bs S_{\bs D}$ be the equivariant subset obtained by selecting
	each point $v\in\bs D$ with probability {$1\wedge (n^{-\beta}\bs w(v))$} (the selection variables are assumed to be conditionally independent given $[\bs D, \bs o; \bs w]$). Let $\bs R_n(v)=n$ if $v\in \bs S_{\bs D}$,
	$\bs R_n(v)=1$ if $N_n(v)\cap \bs S_{\bs D}=\emptyset$, and $\bs R_n(v)=0$ otherwise.
	Then $\bs R_n$ is an equivariant covering. It is shown below that 
	$\omid{\bs R_n(\bs o)^{\gamma}}\rightarrow 0$.
	Let $M:={\sup}\{r\geq 0: \bs w(N_r(\bs o))< r^{\kappa}\}$.
	By the assumption, $M<\infty$ a.s.  One has 
	\begin{eqnarray*}
		\omid{\bs R_n(\bs o)^{\gamma}} &=& n^{\gamma}\myprob{\bs o\in \bs S_{\bs D}} + \myprob{N_n(\bs o)\cap \bs S_{\bs D}=\emptyset}\\
		&=&	n^{\gamma}{\omid{1\wedge n^{-\beta}\bs w(\bs o)}} + \omid{\prod_{v\in N_n(\bs o)}\left(1-(1\wedge {n^{-\beta}}{\bs w(v)})\right)}\\
		&\leq & {n^{\gamma-\beta}\omid{\bs w(\bs o)} + \omid{\exp\left(-n^{-\beta}{\bs w(N_n(\bs o))}  \right)}} \\
		&=& n^{\gamma-\beta}\omid{\bs w(\bs o)} + \omidCond{{\exp\left(-n^{-\beta}{\bs w(N_n(\bs o))}  \right)}}{M<n}\myprob{M<n} \\ & & \quad\quad\quad\quad\quad\quad\hspace{0.2mm}  + \, \omidCond{{\exp\left(-n^{-\beta}{\bs w(N_n(\bs o))}  \right)}}{M\geq n}\myprob{M\geq n}\\
		&\leq & n^{\gamma-\beta}\omid{\bs w(\bs o)} + \exp\left(-n^{\kappa-\beta} \right) + \myprob{M\geq n},
	\end{eqnarray*}
	{where the first inequality holds because $1-(1\wedge x)\leq e^{-x}$ for all $x\geq 0$.}
	Therefore, $\omid{\bs R_n(\bs o)^{\gamma}}\rightarrow 0$ when $n\to\infty$. It follows that $\dimH{\bs D}\geq \gamma$.
	Since $\gamma$ is arbitrary, this implies $\dimH{\bs D}\geq \alpha$.
	
	\eqref{part:lem:lowerbound:all}. Only a small change is needed in the above proof.
	For $n\geq r_0$, let $\bs R_n(v)=n$ if either $v\in \bs S_{\bs D}$ or $N_n(v)\cap \bs S_{\bs D}=\emptyset$,
	and let $\bs R_n(v)=0$ otherwise. Note that $\bs R_n$ is a covering by balls of equal radii. 
	By the same computations and the assumption  {$M\leq r_0$}, one gets
	\[
	\myprob{\bs R_n(\bs o)\neq 0} \leq {n^{-\beta}\omid{\bs w(\bs o)} + \exp\left(-n^{\kappa-\beta}\right),}
	\]
	which is of order $n^{-\beta}$ for large $n$.
	This implies that $\dimMl{\bs D}\geq \beta$.
	Since $\beta$ is arbitrary, one gets $\dimMl{\bs D}\geq \alpha$ and the claim is proved.
	
	\eqref{part:lem:lowerbound:E}. 
	Let $\beta<\alpha$. It will be proved below that under the assumption of \eqref{part:lem:lowerbound:E}, there is a sequence
	$r_1,r_2,\ldots$ such that $\omid{\exp\left(-r_n^{-\beta}{\bs w(N_{r_n}(\bs o))}  \right)}\rightarrow 0$.
	If so, by a slight modification of the proof of part~\eqref{part:lem:lowerbound:liminf},
	one can find a sequence of equivariant coverings $\bs R_n$ such that
	$\omid{\bs R_n(\bs o)^{\beta}}<\infty$ and \eqref{part:lem:lowerbound:E} is proved.
	
	Let $\epsilon>0$ be arbitrary. By the assumption, there is $\delta>0$
	and $r\geq 1$ such that $\myprob{\bs w(N_r(\bs o))\leq \delta r^{\alpha}}<\epsilon$. So
	\begin{eqnarray*}
		\omid{\exp\left(-r^{-\beta}{\bs w(N_{r}(\bs o))}  \right)} &\leq& \omidCond{\exp\left(-r^{-\beta}{\bs w(N_{r}(\bs o))}\right)}{\bs w(N_r(\bs o))>\delta r^{\alpha}}\\
		&& + \myprob{\bs w(N_r(\bs o))\leq \delta r^{\alpha}}\\
		&\leq & \mathrm{exp}(-\delta r^{\alpha-\beta}) + \epsilon.
	\end{eqnarray*}
	Note that for fixed $\epsilon$ and $\delta$ as above, $r$ can be arbitrarily large.
	Now, choose $r$ large enough for the right hand side to be at most $2\epsilon$.
	This shows that $\omid{\exp\left(-r^{-\beta}{\bs w(N_{r}(\bs o))}  \right)}$ can be arbitrarily small and the claim is proved.

	\eqref{part:lem:lowerbound:exp}.
	As before, let $\bs R_n(v)=n$ if either $v\in \bs S_{\bs D}$ or $N_n(v)\cap \bs S_{\bs D}=\emptyset$,
	and let $\bs R_n(v)=0$ otherwise. The calculations in the proof of part~\eqref{part:lem:lowerbound:liminf} show that 
	\[
	\myprob{\bs R_n(\bs o)\neq 0} \leq n^{-\beta}\omid{\bs w(\bs o)} + \omid{\exp\left(-n^{-\beta}\bs w(N_n(\bs o))\right)}.
	\]
	Now, the assumption implies the claim.
\end{proof}

\begin{remark}
	The assumption in part~\eqref{part:lem:lowerbound:E} of Lemma~\ref{lem:lowerboundiid} is
	equivalent to the condition that 
	{there exists a sequence $r_n\to \infty$ such that the family of random variables $r_n^{\alpha}/\bs w(N_{r_n}(\bs o))$ is tight.}
	Also, 
	from the proof of the lemma, one can see that this assumption is equivalent to
	\[\liminf_{n\rightarrow\infty} \omid{\mathrm{exp}\left(-\frac{\bs w(N_n(\bs o))}{n^{\alpha}}\right)} =0.\] 
\end{remark}

\begin{theorem}[Unimodular Billingsley Lemma] 
	\label{thm:billingsley}
	Let $[\bs D, \bs o]$ be a unimodular discrete metric space. Then, for all equivariant weight functions $\bs w$ such that $0<\omid{\bs w(\bs o)}<\infty,$ one has
	\begin{eqnarray*}
		\essinf \left(\growthl{\bs w(N_r(\bs o))}\right) 
		&\leq& \dimH{\bs D} \\ 
		&\leq& \essinf \left(\growthu{\bs w(N_r(\bs o))}\right)\\
		&\leq& \growthu{\omid{\bs w(N_r(\bs o))}}.
	\end{eqnarray*}
\end{theorem}

\begin{proof}
	The first inequality
	is implied by part~\eqref{part:lem:lowerbound:liminf} of Lemma~\ref{lem:lowerboundiid}. 
	For the second inequality, assume that ${\growthu{\bs w(N_r(\bs o))}}<\alpha$ with positive probability.
	On this event, one has $\bs w(N_r(\bs o))\leq r^{\alpha}$ {for large $r$}; i.e.,
	$\limsup_r \bs w(N_r(\bs o))/r^{\alpha}\leq 1$. Now, Lemma~\ref{lem:MDPstronger2} implies that
	$\dimH{\bs D}\leq \alpha$. This proves the second inequality. The last claim follows
	{from Lemma~\ref{lem:logbound}.}
\end{proof}

\begin{corollary}
\label{cor:billingsley-ergodic}
Under the assumptions of Theorem~\ref{thm:billingsley}, 
if the upper and lower growth rates of $\bs w(N_r(\bs o))$ are almost surely constant
(e.g., when $[\bs D,\bs o; \bs w]$ is \textit{ergodic}), then, 
\begin{eqnarray}
\label{eq:billingsly-simple}
{\growthl{\bs w(N_r(\bs o))}} \leq &{\dimH{\bs D}} &\leq \growthu{\bs w(N_r(\bs o))} \quad a.s.
\end{eqnarray}
In particular, if $\growth{\bs w(N_r(\bs o))}$ exists and is constant a.s., then 
\[\dimH{\bs D} = \growth{\bs w(N_r(\bs o))}.\]
\end{corollary}

In fact, without the assumption of this corollary, an inequality similar to~\eqref{eq:billingsly-simple}
is valid for the \textit{sample Hausdorff dimension} of $\bs D$, which will be studied in~\cite{III}.

\formir{
\begin{remark}
In many examples, in the unimodular Billingsley lemma, it is enough to take $\bs w$ equal to 
the counting measure, i.e., $\forall v: \bs w(v)=1$ and $\bs w(N_r(\bs o))=\card{N_r(\bs o)}$.
Analogously, for many natural fractals in the continuum setting,
there is a natural \textit{mass measure} that can be used in Billingsley's lemma.
\end{remark}
}

\begin{remark}
\label{rem:bil-nec}
In fact, the assumption $\omid{\bs w(\bs o)}<\infty$ in Theorem~\ref{thm:billingsley}
is only needed for the lower bound while the assumption $\omid{\bs w(\bs o)}>0$ is only needed for the upper bound.
These assumptions are also necessary as shown below.
\\	
For example, assume {$\Phi$} is a point-stationary point process in $\mathbb R$
(see Example~\ref{ex:point-stationary}).
For $v\in \Phi$, let $\bs w(v)$ be the sum of the distances of $v$ to its next and previous points in $\Phi$.
This equivariant weight function satisfies $\bs w(N_r(v))\geq 2r$ for all $r$, and hence
$\growthl{\bs w(N_r(\bs o))}\geq 1$. But $\dimH{\Phi}$ can be strictly less than 1
as shown in Subsection~\ref{subsec:image}.
\\
Also, the condition that \formir{$\omid{\bs w_{\bs D}(\bs o)}>0$} is trivially necessary for the upper bound.
\end{remark}

\begin{corollary}
	\label{cor:graphs-lowerbound}
	Let $[\bs G, \bs o]$ be a unimodular random graph equipped with the graph-distance metric.
	If $\bs G$ is infinite almost surely, then $\dimMl{\bs G}\geq 1$ and else, $\dimM{\bs G}=\dimH{\bs G}=0$.
\end{corollary}

\begin{proof}
	If $\bs G$ is infinite a.s., then for $\bs w_{\bs G}\equiv 1$, one has $\bs w(N_r(\bs o))\geq r$ for all $r$.
	So part~\eqref{part:lem:lowerbound:all} of Lemma~\ref{lem:lowerboundiid} implies the first claim.
	{The second claim is implied by Example~\ref{ex:hausdorff-basic} (this can be deduced from the unimodular Billingsley lemma as well).}
\end{proof}

\begin{corollary}
The unimodular Minkowski and Hausdorff dimensions of any unimodular two-ended tree are equal to one.
\end{corollary}
This result has already been shown in Theorem~\ref{thm:twoEnded}, but can also be deduced
from the unimodular Billingsley lemma directly. For this, let $\bs w(v)$ be 1 if 
$v$ belongs to the trunk of the tree and 0 otherwise.

\begin{problem}
\label{prob:uppergrowth}
In the setting of Corollary~\ref{cor:billingsley-ergodic}, is it always the case
that $\dimH{\bs D}=\growthu{\bs w(N_r(\bs o))}$?
\end{problem}

The claim of this problem holds in all of the examples in which both quantities are computed in this work.
This problem is a corollary of Problem~\ref{prob:growth} and the unimodular Frostman lemma
(Theorem~\ref{thm:frostmanGeneral}) below. Note that there are examples where
$\growthl{\cdot}\neq \growthu{\cdot}$ as shown in 
Subsections~\ref{subsec:canopy-generalized} and~\ref{subsec:digits}.

\subsection{Bounds for Point Processes}
\label{subsec:euclidean}

\formir{Example~\ref{ex:point-stationary}, explains that for point processes containing the origin, unimodularity is, roughly speaking, equivalent to point-stationarity. To study the dimension of such processes, the following covering is used in the next results.}
Let $\varphi$ be a discrete subset of $\mathbb R^k$
equipped with the $l_{\infty}$ metric and $r\geq 1$.
Let $C:=C_r:=[0,r)^k$, $\bs U:=\bs U_r$ be a point chosen uniformly at random in $-C$,
and consider the partition $\{C+\bs U+z: z\in r\mathbb Z^k\}$ of $\mathbb R^k$ {by cubes}.
Then, for each $z \in r\mathbb Z^k$, choose a random element in $(C+\bs U+z)\cap \varphi$ independently
(if the intersection is nonempty). The distribution of this random element should depend on the set 
$(C+\bs U+z)\cap \varphi$ in a translation-invariant way (e.g., choose with the uniform distribution
or choose the least point in the lexicographic order).
Let ${\bs R=}\bs R_{\varphi}$ assign the value  $r$ to the selected points and zero to the other points of $\varphi$.
Then, $\bs R$ is an equivariant covering. 
Also, each point is covered at most $3^k$ times. So $\bs R$ is $3^k$-bounded (Definition~\ref{def:nearlyDisjoint}).

\begin{theorem}[Minkowski Dimension in the Euclidean Case]
	\label{thm:lowerBoundR^d}
	Let $\Phi$ be a point-stationary point process in $\mathbb R^k$
	and assume the metric on $\Phi$ is equivalent to the Euclidean metric.
	Then, for all equivariant weight functions $\bs w$ such that ${\bs w_{\Phi}(0)}>0$ a.s., one has 
	\begin{eqnarray*}
		\dimMu{\Phi}  = \decayu {\omid{\bs w(0)/ \bs w(C_r+\bs U_r)}} &\leq& \decayu{\omid{\bs w(0)/ \bs w(N_r(0))}}\\
		& \leq & \growthu{\omid{\bs w(N_r(0))}},\\
		\dimMl{\Phi}  = \decayl {\omid{\bs w(0)/ \bs w(C_r+\bs U_r)}} &\leq& \decayl{\omid{\bs w(0)/ \bs w(N_r(0))}}\\
		& \leq & \growthl{\omid{\bs w(N_r(0))}},
	\end{eqnarray*}
	where $\bs U_r$ is a uniformly at random point in $-C_r$ independent of $\Phi$ and $\bs w$. 
\end{theorem}

\begin{proof}
	By Theorem~\ref{thm:metricChange}, one may assume the metric on $\Phi$ is the $l_{\infty}$ metric without loss of generality. 
	Given any $r>0$, consider the equivariant covering $\bs R$ described above, but when choosing a random 
	element of $(C_r+\bs U_r+z)\cap \varphi$, choose point $v$ with probability
	$\bs w_{\varphi}(v)/\bs w_{\varphi}(C_r+\bs U_r+z)$ (conditioned on $\bs w_{\varphi}$). One gets
	$
	\myprob{0 \in \bs R} = \omid{ {\bs w(0)}/{\bs w(C_r+\bs U_r)}}.
	$
	As mentioned above,
	$\bs R$ is equivariant and uniformly bounded (for all $r>0$).
	So Lemma~\ref{lem:nearlyDisjoint} implies both equalities in the claim.
	The inequalities are implied by the facts that $\bs w(C_r+\bs U_r)\leq \bs w(N_r(0))$ and
	\[
	\omid{\frac {\bs w(0)}{\bs w(\nei{r}{0})}} \omid{\bs w(\nei{r}{0})} \geq \omid{\sqrt{\bs w(0)}}^2>0,
	\]
	which is implied by the Cauchy-Schwartz inequality.
\end{proof}

\formir{
	\begin{example}
		The right-most inequalities in the above theorem can be strict. For example, let $T>0$ be a random number and let $\Phi:=\frac 1 T \mathbb Z$. Then $\card{N_r(0)}\sim 1+T/r$. So, $\decay{1/\card{N_r(0)}}=1$, but it might be the case that $\omid{\card{N_r(0)}}=\infty$. For an ergodic example, let $1\leq T_i\in\mathbb Z$ be i.i.d. with finite mean but infinite variance (for $i\in\mathbb Z$). In each interval $[i,i+1]$, put $T_i-1$ equidistant points and let $\Phi_0$ be the union of these points together with $\mathbb Z$. Bias the distribution of $\Phi_0$ by $T_0$ (Definition~\ref{def:bias}) and then translate $\Phi_0$ by moving a random point in $\Phi_0\cap [0,1)$ to the origin. Let $\Phi$ be the resulting point process. It can be seen that $\Phi$ is unimodular and point-stationary. Since $\omid{T_0^2}=\infty$, one gets $\omid{N_r(0)}\geq \omid{N_1(0)}=\infty$. But one can show that $\decay{\omid{1/N_r(0)}}=1$.
	\end{example}
}

\begin{proposition}
	\label{prop:upperbound-Rd}
	If $\Phi$ is a point-stationary point process in $\mathbb R^k$ and the metric on $\Phi$ is equivalent 
	to the Euclidean metric, then $\dimH{\Phi}\leq k$.
\end{proposition}
\begin{proof}
	One may assume the metric on $\Phi$ is the $l_{\infty}$ metric without loss of generality.
	Let $C:=[0,1)^k$ and $\bs U$ be a random point in $-C$ chosen uniformly. For all discrete subsets
	$\varphi\subseteq\mathbb R^k$ and $v\in \varphi$, let $\bs C(v)$ be the cube containing $v$ of the
	form $C+\bs U+z$ (for $z\in \mathbb Z^k$) and $\bs w_{\varphi}(v):= 1/\card{(\varphi\cap \bs C(v))}$.
	Now, $\bs w$ is an equivariant weight function. The construction readily 
        implies that $\bs w(N_r(\bs o))\leq (2r+1)^k$.
	Moreover, by $\bs w\leq 1$, one has $\omid{\bs w(0)}<\infty$. Therefore, the unimodular Billingsley
	lemma (Theorem~\ref{thm:billingsley}) implies that $\dimH{\Phi}\leq k$.
\end{proof}


\begin{proposition}
\label{prop:palm}
If $\Psi$ is a stationary point process in $\mathbb R^k$ {with finite intensity} and $\Psi_0$
is its Palm version, then
$\dimM{\Psi_0}=\dimH{\Psi_0}=k.$
Moreover, the modified unimodular \formir{Hausdorff size} of $\Psi_0$, defined in Section~\ref{subsec:otherSets}, satisfies
$
\mathcal M'_k(\Psi_0)= {2^k}{\rho(\Psi)},
$
where $\rho(\Psi)$ is the intensity of $\Psi$. 
\end{proposition}
Notice that if $\Psi_0\subseteq \mathbb Z^k$, {then the claim is directly implied by Theorem~\ref{thm:subsetDimension}.} The general case is treated below. 
\begin{proof}
	For the first claim, by Proposition~\ref{prop:upperbound-Rd} and Theorem~\ref{thm:comparison},
	it is enough to prove that $\dimMl{\Psi_0}\geq k$. 
	Let $\Psi'$ be a shifted square lattice independent of $\Psi$ (i.e., $\Psi'=\mathbb Z^k+\bs U$, 
	where $\bs U\in [0,1)^k$ is chosen uniformly, independently of $\Psi$).
	Let $\Psi'':=\Psi\cup \Psi'$. Since $\Psi''$ is a superposition of two independent stationary point processes,
	it is a stationary point process itself.
	By letting $p:=\rho(\Psi)/(\rho(\Psi)+1)$, the Palm version $\Psi_0''$ of $\Psi''$ is obtained 
	by the superposition of $\Psi_0$ and an independent stationary lattice with probability $p$ (heads),
	and the superposition of $\mathbb Z^k$ and $\Psi$ with probability $1-p$ (tails). So 
	part~\eqref{part:lem:lowerbound:all} of Lemma~\ref{lem:lowerboundiid} implies that $\dimMl{\Psi_0''}\geq k$.
	Note that $\Psi_0''$ has two natural equivariant subsets which, after conditioning to contain the origin,
	have the same distributions as $\Psi_0$ and $\mathbb Z^k$ respectively.
	Therefore, one can use Theorem~\ref{thm:subsetDimension} to deduce that $\dimMl{\Psi_0}\geq \dimMl{\Psi_0''}=k$.
	Therefore, Proposition~\ref{prop:upperbound-Rd} implies that $\dimH{\Psi_0} = \dimM{\Psi_0}=k$.
	
	Also, by using Theorem~\ref{thm:subsetDimension} twice, one gets $\mathcal M'_k(\Psi_0)= p \mathcal M'_k(\Psi_0'')$
	and $\mathcal M'_k(\mathbb Z^k) =  (1-p)\mathcal M'_k(\Psi_0'')$. Therefore, 
	$
	\mathcal M'_k(\Psi_0) = p/{(1-p)} \mathcal M'_k(\mathbb Z^k). 
	$
	By the definition of $\mathcal M'_k$, one can directly show that $\mathcal M'_k(\mathbb Z^k)=2^k$ (see also Proposition~\ref{prop:lattice-Hmeasure}). This implies the claim.
\end{proof}

The last claim of Proposition~\ref{prop:palm} suggests the following, which is verified when $k=1$ in the next proposition.

\begin{conjecture}
	\label{conj:point-stationary}
	If $\Phi$ is a point-stationary point process in $\mathbb R^k$ which is not the Palm version of
	any stationary point process, then $\measH{k}(\Phi)=0$. 
\end{conjecture}

\begin{proposition}
	\label{prop:conj:point-stationary}
	Conjecture~\ref{conj:point-stationary} is true when $k=1$.
\end{proposition}
\begin{proof}
	Denote $\Phi$ as $\Phi=\{S_n: n\in\mathbb Z\}$ such that $S_0=0$ and $S_n<S_{n+1}$ for each $n$.
	Then, the sequence {$T_n:=S_{n+1}-S_n$} is stationary under shifting the indices (see Example~\ref{ex:point-stationary}). 
	The assumption that $\Phi$ is not the Palm version of a stationary point process is equivalent
	to $\omid{S_1}=\infty$ (see~\cite{bookDaVe03II} or Proposition~6 of~\cite{shift-coupling}).
	Indeed, if $\omid{S_1}<\infty$, then one could bias the probability measure by $S_1$ 
	(Definition~\ref{def:bias}) 
	and then shift the whole process by $-\bs U$, where $\bs U\in [0,S_1]$
	is chosen uniformly and independently.
	
	Since $\omid{S_1}=\infty$, Birkhoff's pointwise ergodic theorem~\cite{bookPe89} implies
	that {$\lim_n (T_1+\cdots+T_n)/n=\infty$}. This in turn implies that $\lim_r \card{N_r(0)}/r = 0$.
	Therefore, Lemma~\ref{lem:MDPstronger2} gives that $\contentH{1}{\infty}(\Phi)=\infty$; i.e., $\measH{1}(\Phi)=0$.
\end{proof}

\subsection{Connections to Birkhoff's Pointwise Ergodic Theorem}

The following corollary of the unimodular Billingsley lemma is of independent interest.
Note that the statement does not involve dimension. 

\begin{theorem}
\label{thm:birkhoff}
Let $[\bs D, \bs o]$ be a unimodular discrete space. 
For any two equivariant weight functions $\bs w_1$ and $\bs w_2$
\formir{such that $\myprob{\exists v\in \bs D: \bs w_{2}(v)\neq 0}=1$}
and $\omid{\bs w_1(\bs o)}<\infty$, 
one has
\[
\growthl{\bs w_1(N_r(\bs o))}\leq \growthu{\bs w_2(N_r(\bs o))}, \quad a.s.
\]
In particular, if $\bs w_1(N_r(\bs o))$ and $\bs w_2(N_r(\bs o))$ have well defined growth rates,
then their growth rates are equal.
\end{theorem}

Note that the condition $\omid{\bs w_1(\bs o)}<\infty$ is necessary as shown in Remark~\ref{rem:bil-nec}.

\begin{proof}
	Let $\epsilon>0$ be arbitrary and
	$$A:=\{[D,o]\in\dstar: \growthl{\bs w_1(N_r(o))}> \growthu{\bs w_2(N_r(o))}+\epsilon\}.$$ 
	It can be seen that $A$ is a measurable subset of $\dstar$. Assume $\myprob{[\bs D,\bs o]\in A}>0$. 
	Denote by $[\bs D',\bs o']$ the random \rooted{} discrete space obtained by conditioning $[\bs D,\bs o]$ on $A$.
	Since $A$ does not depend on the root (i.e., if $[D,o]\in A$, then $\forall v\in D: [D,v]\in A$),
	by a direct verification of the mass transport principle~(\ref{eq:unimodular}),
	one can show that $[\bs D',\bs o']$ is unimodular.
	So by using the unimodular Billingsley
	lemma (Theorem~\ref{thm:billingsley}) twice, one gets
	\[
	\essinf \left(\growthl{\bs w_1(N_r(\bs o'))}\right)\leq \dimH{\bs D'}\leq \essinf \left(\growthu{\bs w_2(N_r(\bs o'))}\right).
	\]
	By the definition of $A$, this contradicts the fact that $[\bs D',\bs o']\in A$ a.s.
	So $\myprob{[\bs D,\bs o]\in A}=0$ and the claim is proved.
\end{proof}

\begin{remark}
	Theorem~\ref{thm:birkhoff} is a generalization of a weaker form of Birkhoff's pointwise ergodic
	theorem as explained below. In the cases where $\bs D$ is either $\mathbb Z$, the Palm version
	of a stationary point process in $\mathbb R^k$ or a point-stationary point process in $\mathbb R$,
	Birkhoff's pointwise ergodic theorem (or its generalizations) implies that
	$\lim {\bs w_1(N_r(\bs o))}/{\bs w_2(N_r(\bs o))} = \omid{\bs w_1(0)}/\omid{\bs w_2(0)}$ a.s.
	This is stronger than the claim of Theorem~\ref{thm:birkhoff}.
	Note that Theorem~\ref{thm:birkhoff} implies nothing about 
	$\lim {\bs w_1(N_r(\bs o))}/{\bs w_2(N_r(\bs o))}.$ 
	On the other side, note that \textit{amenability} is not assumed in this Theorem,
	which is a general requirement in the study of ergodic theorems.
	However, it will be proved in~\cite{III} that, roughly speaking, non-amenability implies 
	$\growthu{\bs w_2(N_r(\bs o))}=\infty$, which makes the claim of Theorem~\ref{thm:birkhoff} trivial
	in this case. 
	In this case, using exponential gauge functions seems more interesting.
\end{remark}

{
	\begin{problem}
		\label{prob:growth}
		Is it true that for every unimodular discrete space $[\bs D, \bs o]$, the growth rates 
		$\growthu{\bs w(N_r(\bs o))}$ and $\growthl{\bs w(N_r(\bs o))}$ do not depend on  $\bs w$ as long as $0<\omid{\bs w(\bs o)}<\infty$?
	\end{problem}
}

\subsection{Notes and Bibliographical Comments}

As already mentioned, the unimodular mass distribution principle and the unimodular Billingsley lemma have analogues in the continuum setting (see e.g., \cite{bookBiPe17}) and are named accordingly.  
Note however that there is no direct or systematic reduction to these continuum results. 
For instance, in the continuum setting, one should assume that the space under study is a subset of the Euclidean space, or more generally, satisfies the \textit{bounded subcover property} ({see e.g.,~\cite{bookBiPe17}}). Theorem~\ref{thm:billingsley} does not require such assumptions.
Note also that the term $\growthu{\bs w(N_r(\bs o))}$ in Theorem~\ref{thm:billingsley}
does not depend on the \rooot{} in contrast to the analogous term in the continuum version.
Similar observations can be made on Theorem \ref{thm:mdp-simple}.

\section{Examples \formir{Continued}}
\label{sec:examples2}

This section presents further 
examples for illustrating the results of the previous section.

\subsection{Remaining Proofs from Section~\ref{sec:examples}}
\label{subsec:remainingproofs}

\formir{The unimodular Billingsley lemma can be used to complete the computation of
the unimodular Hausdorff dimension in the examples of Section~\ref{sec:examples}.
These examples include Eternal \gw{} trees, the image of a random walk and the drainage network model.}

\begin{proof}[\formir{Proof of Proposition~\ref{prop:EGWdimension} (second part)}]
	\formir{The equality $\dimM{\bs T}=2$ is proved in Subsection~\ref{subsec:egw}.}
	So it remains to prove $\dimH{\bs T}\leq 2$.
	By the unimodular Billingsley lemma, it is enough to show that 
	$\omid{\card{N_n(\bs o)}}\leq cn^2$ for a constant $c$.
	Recall from Subsection~\ref{subsec:one-ended} that $F(v)$ represents the parent of vertex $v$ and $D(v)$ 
	denotes the subtree of descendants of $v$. Write $N_n(\bs o)=Y_0\cup Y_1\cup \cdots \cup Y_n$,
	where $Y_n:={N_n(\bs o) \cap D(\bs o)}$ and 
	$Y_i:={N_n(\bs o) \cap D(F^{n-i}(\bs o))\setminus D(F^{n-i-1}(\bs o))}$ for $0\leq i<n$. 
	By the explicit construction of \egw{} trees in~\cite{eft}, 
	$Y_n$ is a critical Galton-Watson tree up to generation $n$. Also, for $0\leq i<n$, 
	$Y_i$ has the same structure up to generation $i$, except that the distribution of the
	first generation is \textit{size-biased minus one} (i.e., $(np_{n+1})_n$
        with the notation of Subsection~\ref{subsec:ugw}).
	So the assumption of finite variance implies that the first generation in each
        $Y_i$ has finite mean, namely $m'$.
	Now, one can inductively show that $\omid{\card{Y_n}}= n$ and $\omid{\card{Y_i}}=i m'$,
	for $0\leq i<n$. It follows that $\omid{\card{N_n(\bs o)}}\leq (1+m')n^2$ and the claim is proved.
\end{proof}

\begin{proof}[\formir{Proof of Proposition~\ref{prop:image} (second part)}]
\formir{In Subsection~\ref{subsec:image}, it is proved that $\dimMl{\Phi} \geq 1\wedge \decayl{\myprob{S_1>r}}$.
So part~\eqref{part:thm:image:1} is implied by part~\eqref{part:thm:image:h}, which is proved below.}
Since $\Phi$ is a point-stationary point process in $\mathbb R$ (see Subsection~\ref{subsec:image}),
Proposition~\ref{prop:upperbound-Rd} implies that $\dimH{\Phi}\leq 1$.
Now, assume $\decayu{\myprob{S_1>r}}< \beta$.
Then, there exists $c>0$ such that $\myprob{S_1>r}> c r^{-\beta}$ for all $r\geq 1$.
This implies that there exists $C<\infty$ and a random number $r_0>0$ such that for all $r\geq r_0$,
one has $\card{N_r(\bs o)}\leq Cr^{\beta} \log \log r$ a.s. ({see Lemma~\ref{lem:BaumKatz1} in the appendix} or Theorem~4 of~\cite{FrPr71}).
Therefore, the unimodular Billingsley lemma (Theorem~\ref{thm:billingsley}) implies that
$\dimH{\Phi}\leq \beta+\epsilon$ for every $\epsilon>0$, which in turn implies that $\dimH{\Phi}\leq \beta$.
\end{proof}

\begin{example}[Infinite \formir{H-Size}]
\label{ex:infiniteMeasure}
In Proposition~\ref{prop:image}, assume that
$\myprob{S_1>r}= 1/\log r$ for large enough $r$.
Then, part~\eqref{part:thm:image:h} of the proposition implies that $\dimH{\Phi}=0$.
However, since $\Phi$ is infinite a.s., it has infinite \formir{0-dim H-size}
(Proposition~\ref{prop:finite-HausMeas}).
\end{example}

\begin{example}[Zero \formir{H-Size}]
\label{ex:zeroMeasure}
In Proposition~\ref{prop:image}, assume $\myprob{S_1>r}= 1/ r$ for large enough $r$.
Then, part~\eqref{part:thm:image:h} of the proposition implies that $\dimH{\Phi}=1$.
Since $\omid{S_1}=\infty$, $\Phi$ is not the Palm version of any stationary point process
(see Proposition~\ref{prop:conj:point-stationary}). Therefore, 
Proposition~\ref{prop:conj:point-stationary} implies that $\measH{1}(\Phi)=0$.
\end{example}

\begin{proof}[\formir{Proof of Proposition~\ref{prop:drainage} (second part)}]
	\formir{The equality $\dimM{\bs T}=\frac 3 2$ is proved in Subsection~\ref{subsec:drainage}. So it remains} to prove $\dimH{\bs T}\leq \frac 3 2$.
	To use the unimodular Billingsley lemma, an upper bound on $\omid{\card{N_n(\bs o)}}$ is derived. Let 
	$e_{k,l}:=\card{\left(F^{-k}(F^l(\bs o))\setminus F^{-(k-1)}(F^{l-1}(\bs o))\right)}$
	be the number of descendants of order $k$ of $F^l(\bs o)$ which are not a descendant
	of $F^{l-1}(\bs o)$ (for $l=0$, let it be just $\card{F^{-k}(\bs o)}$).
	One has $\card{N_n(\bs o)} = \sum_{k,l} e_{k,l}\identity{\{k+l\leq n\}}$.
	It can be seen that $\omid{e_{k,l}}$ is equal to the probability that two independent
	paths of length $k$ and $l$ starting both at $\bs o$ do not collide at another point. 
	Therefore, $\omid{e_{k,l}}\leq c (k\wedge l)^{-\frac 12}$ for some $c$ and all $k,l$.
	This implies that (in the following, $c$ is updated at each step to a new constant without changing the notation)
	\begin{eqnarray*}
		\omid{\sum_{k,l\geq 0} e_{k,l} \identity{\{k+l\leq n\}}} &\leq & \sum_{k=0}^{\floor{\frac n 2}} ck^{-\frac 12}(n-k)
		\le  cn\sum_{k=0}^{\floor{\frac n 2}} k^{-\frac 12}
		\le cn^{\frac 3 2}.
	\end{eqnarray*}
	The above inequalities imply that $\omid{\card{N_n(\bs o)}}\leq cn^{\frac 3 2}$ for some $c$ and all $n$.
	Therefore, the unimodular Billingsley lemma (Theorem~\ref{thm:billingsley}) implies that $\dimH{\bs T} \leq \frac 32$.
	So the claim is proved.
\end{proof}

\subsection{General Unimodular Trees \formir{Continued}}
\label{subsec:trees2}
{The following is a direct corollary of Theorem~\ref{thm:one-ended} and the unimodular Billingsley lemma. Since the statement does not involve dimension, it is of independent interest and believed to be new.
	
	\begin{corollary}
		For every unimodular one-ended tree $[\bs T, \bs o]$ and every equivariant weight function $\bs w$, almost surely,
		\[
		\decayu{\myprob{h(\bs o)=n}} \leq \growthu{\bs w(N_r(\bs o))} \leq \growthu{\omid{\bs w(N_r(\bs o))}}.
		\]
	\end{corollary}
	
	The rest of this subsection is focused on unimodular trees with infinitely many ends.}

\begin{proposition}
	\label{prop:tree-expGrowth}
	Let $[\bs T, \bs o]$ be a unimodular tree with infinitely many ends such that $\omid{\mathrm{deg}(\bs o)}<\infty$. Then $\bs T$ has exponential \formir{volume} growth a.s. {and $\dimH{\bs T}=\infty$}.
\end{proposition}

In fact, the graph-distance metric on $\bs T$ can be replaced by an arbitrary equivariant metric.
This will be proved in~\cite{III}.

The following proof uses the definitions and results of~\cite{processes}, but they are not recalled for brevity.
\begin{proof}[Proof of Proposition~\ref{prop:tree-expGrowth}]
By Corollary~8.10 of~\cite{processes}, $[\bs T, \bs o]$ is \textit{non-amenable}
(this will be discussed further in~\cite{III}). 
So Theorem~8.9 of~\cite{processes} implies that the critical probability $p_c$ of percolation on $\bs T$ is
less than one with positive probability. In fact, it can be shown that $p_c<1$ a.s. 
(if not, condition on the event $p_c=1$ to get a contradiction).
For any tree, $p_c$ is equal to the inverse of the \textit{branching number}.
So the branching number is more than one, which implies that the tree has exponential \formir{volume} growth. 
Finally, the unimodular Billingsley lemma (Theorem~\ref{thm:billingsley}) implies that $\dimH{\bs T}=\infty$.
\end{proof}

The following example shows that the Minkowski dimension can be finite.
\begin{example}
Let $T$ be the 3-regular tree. Split each edge $e$ by adding a random number $\bs l_e$ of
new vertices and let $\bs T_0$ be the resulting tree. Let $v_e$ be the middle vertex in
this edge (assuming $\bs l_e$ is always odd) and assign marks by $\bs m_0(v_e):=\bs l_e$. 
Assume that the random variables $\bs l_e$ are i.i.d. If $\omid{\bs l_e}<\infty$, then one
can bias the probability measure and choose a new root to obtain a unimodular marked tree,
namely $[\bs T,\bs o;\bs m]$ (see Example~9.8 of~\cite{processes} or~\cite{shift-coupling}). 
It will be shown below that $\dimMu{\bs T}$ may be finite. 
		\\
Let $\bs R$ be an arbitrary equivariant $r$-covering of $\bs T$. Consider the set of middle
vertices $\bs A_r:=\{v\in\bs T: \bs m(v)\geq r\}$. Since these vertices have pairwise distance at least $r$,
they belong to different balls in the covering. So, by the mass transport principle, one can
show that $\rho(\bs R)\geq \rho(\bs A_r)$, where $\rho(\cdot)=\myprob{\bs o\in \cdot}$ denotes 
the intensity. On the other hand, let $\bs S$ be the equivariant subset of vertices with degree 3.
Send unit mass from every point of $\bs A_r$ to its two closest points in $\bs S$.
Then the mass transport principle implies that $2\rho(\bs A_r)=3\rho(\bs S)\myprob{\bs l_e\geq r}$.
Hence, $\rho(\bs R)\geq \frac 3 2 \rho(\bs S)\myprob{\bs l_e\geq r}$. This gives that
$\dimMu{\bs T}\leq \decayu{\myprob{\bs l_e\geq r}}$, which can be finite.
In fact, if $\decay{\myprob{\bs l_e\geq r}}$ exists, Proposition~\ref{prop:infEndslowerbound}	 
below implies that $\dimM{\bs T}=\decay{\myprob{\bs l_e\geq r}}$.
\end{example}

	The following proposition gives a lower bound on the Minkowski dimension. 
	
	\begin{proposition}
		\label{prop:infEndslowerbound}
		Let $[\bs T, \bs o]$ be a unimodular tree with infinitely many ends and without leaves. Let $\bs S$ be the equivariant subset of vertices of degree at least 3. For every $v\in\bs S$, let $\bs w(v)$ be the sum of the distances of $v$ to its neighbors in $\bs S$. If $\omid{\bs w(\bs o)^{\alpha}}<\infty$, then $\dimMl{\bs T}\geq \alpha$.
	\end{proposition}
	
	The proof is based on the following simpler result. This will be used in Subsection~\ref{subsec:pwit} as well.
	
\begin{proposition}
\label{prop:regree-distorted}
Let $[\bs T,\bs o]$ be a unimodular tree such that the degree of every vertex is at least 3.
Let $\bs d'$ be an equivariant metric on $\bs T$. Let $\bs w(v):=\sum_u \bs d'(v,u)$,
where the sum is over the 3 neighbors of $v$ which are closest to $v$ under the metric $\bs d'$.
If $\omid{\bs w(\bs o)^{\alpha}}<\infty$, then $\dimMl{\bs T,\bs d'}\geq\alpha$.
\end{proposition}

\begin{proof}
Define $\bs w'(v):=\sum_{u} \bs d'(u,v)^{\alpha}$, where the sum is over the three
closest neighbors of $v$. It is enough to assume that $\bs d'$ is generated by equivariant
edge lengths since increasing the edge lengths does not increase the dimension
(by Theorem~\ref{thm:metricChange}). By the same argument, it is enough to assume
$\bs d'(u,v)\geq 1$ for all $u\sim v$. Then,
it can be seen that there exists a constant $c$, that depends only on $\alpha$, 
such that $\bs w'(\nei{r}{v})\geq c r^{\alpha}$ for all $v\in\bs T$ and $r\geq 0$ ({Lemma~\ref{lem:regtree-weight}}). 
Also, the assumption implies that $\omid{\bs w'(\bs o)}<\infty$.
So Lemma~\ref{lem:lowerboundiid} implies that $\dimMl{T_3,\bs d'}\geq \alpha$ and the claim is proved.
\end{proof}

\begin{proof}[Proof of Proposition~\ref{prop:infEndslowerbound}]
For $v\in \bs S$, let $\bs w'(v):=\sum_u d(u,v)^{\alpha}$, where the sum is over
the neighbors of $v$ in $\bs S$. For $v\in\bs T\setminus\bs S$, if $u_1$ and $u_2$
are the two closest points of $\bs S$ to $v$, let $g(v,u_i):=d(u_i,v)^{\alpha-1}$
and $\bs w'(v):=g(v,u_1)+g(v,u_2)$.
The assumption implies that $\omid{\bs w'(\bs o)}<\infty$ (use the mass transport principle for $g$ defined above).
Similarly to Proposition~\ref{prop:regree-distorted}, there exists $c=c(\alpha)$,
such that $\bs w'(\nei{r}{v})\geq c r^{\alpha}$ for all $v\in\bs T$ and $r\geq 0$
({Lemma~\ref{lem:regtree-weight}}) and the claim is proved.
\end{proof}

\subsection{Instances of Unimodular Trees \formir{Continued}}

\subsubsection{A Unimodular Tree With No \formir{Volume} Growth Rate}
\label{subsec:canopy-generalized}
Recall the generalized canopy tree $[\bs T, \bs o]$ from Subsection~\ref{subsec:canopyGeneralized}. Here, it is shown that $\growthl{\bs T}\neq \growthu{\bs T}$ if the parameters are suitably chosen. Similarly, it provides an example where the exponential growth rate does not exist. The existence of unimodular trees without exponential growth rate is already proved in~\cite{Ti14}, but with a more difficult construction.

Choose the sequence $(p_n)_n$ in the definition of $[\bs T, \bs o]$ such that $p_n=c2^{-q_n}$ and $\sum_n p_n=1$, where $c$ is constant and $q_0\leq q_1\leq \cdots$ is a sequence of integers. In this case, $\bs T$ is obtained by splitting the edges of the canopy tree by adding new vertices or concatenating them, depending only on the \textit{level} of the edges. It can be seen that if $v$ is a vertex in the $n$-th level of $\bs T$, then the number of descendants of $v$ is $(p_0+\cdots+p_n)/p_n$. It follows that $\growthl{\bs T}=\decayl{p_n}$ and $\growthu{\bs T}=\decayu{p_n}$. So, by choosing $(p_n)_n$ appropriately, $\bs T$ can have no polynomial (or exponential) \formir{volume} growth rate. This proves the claim. 
Note also that the unimodular Billingsley lemma and  Theorem~\ref{thm:one-ended} imply that $\dimH{\bs T}=\growthu{\bs T}$ here.

\subsubsection{Unimodular Galton-Watson Trees}
\label{subsec:ugw}

Here, it is shown that the unimodular Galton-Watson tree~\cite{processes} is infinite dimensional 
(note that this tree differs from the Eternal Galton-Watson tree of Subsection~\ref{subsec:egw} which is a directed tree). 
Consider an ordinary Galton-Watson tree with offspring distribution $\mu=(p_0,p_1,\ldots)$, where $\mu$ is a probability measure on $\mathbb Z^{\geq 0}$.
The {unimodular Galton-Watson tree} $[\bs T, \bs o]$ has a similar construction with the
difference that the offspring distribution of the \rooot{} is different from that of the other vertices:
It has for distribution {the size-biased version} $\hat{\mu}=(\frac n m p_n)_n$, where $m$ is the mean of $\mu$ (assumed to be finite).

In what follows, the trivial case $p_1=1$ is excluded.
If $m\leq 1$, then $\bs T$ is finite a.s.; i.e., there is extinction a.s.
Therefore, $\dimH{\bs T}=0$.
So assume the \textit{supercritical case}, namely $m>1$.
If $p_0>0$, then $\bs T$ is finite with positive probability.
So $\dimH{\bs T}=0$ for the same reason. Nevertheless, one can condition on non-extinction as follows.

\begin{proposition}
	Let $[\bs T, \bs o]$ be a supercritical unimodular Galton-Watson tree conditioned on non-extinction. Then,
	$
	\dimM{\bs T}=\dimH{\bs T}=\infty.
	$
\end{proposition}
\begin{proof}
	The result for the Hausdorff dimension follows from the unimodular Billingsley lemma
	(Theorem~\ref{thm:billingsley}) and the Kesten-Stigum theorem~\cite{KeSt66},
	which implies that $\lim_n \card{N_n(\bs o)}m^{-n}$ exists and is positive a.s. 
	Computing the Minkowski dimension is more difficult. By part~\eqref{part:lem:lowerbound:exp} of Lemma~\ref{lem:lowerboundiid},
	it is enough to prove that $\omid{(1-n^{-\alpha})^{\card{\nei{n}{\bs o}}}}$ has infinite decay rate for every $\alpha\geq 0$.
	Denote by $[\widetilde{\bs T},\widetilde{\bs o}]$ the Galton-Watson tree with the same parameters.
	Using the fact that $\card{N_n(\bs o)}$ is stochastically larger than $\card{N_{n-1}(\widetilde{\bs o})}$,
	one gets that it is enough to prove the last claim for $[\widetilde{\bs T},\widetilde{\bs o}]$. 
	
	For simplicity, the proof is given for the case $p_0=0$ only. By this assumption, the probability of extinction is zero. The general case can be proved with 
	similar arguments and by using the decomposition theorem of supercritical Galton-Watson trees
	(see e.g.,  Theorem~5.28 of~\cite{bookLyPe16}). In fact, the following proof implies the general claim by the fact that the \textit{trunk}, conditioned on non-extinction, is another supercritical unimodular Galton-Watson tree. The latter can be proved similarly to the decomposition theorem.
	
	Let $f(s):=\sum_n p_n s^n$ be the generating function of $\mu$.
	By classical results of the theory of branching processes, for all $s\leq 1$, 
	$\omid{s^{d_n(\widetilde{\bs o})}}=f^{(n)}(s),$
	where $d_n(\tilde{\bs o}):=\card{N_n(\widetilde{\bs o})}-\card{N_{n-1}(\widetilde{\bs o})}$
	and $f^{(n)}$ is the $n$-fold composition of $f$ with itself. 
	Let $a>0$ be fixed and $g(s):=\frac {as}{-s+a+1}$ (such functions are frequently used in
	the literature on branching processes; see, e.g., \cite{bookAsHe83}).
	One has $f(0)=g(0)=0$, $f(1)=g(1)=1$, $f'(1)=m>1$, $g'(1)=(1+a)/a$, and $f$ is convex.
	Therefore, $a$ can be chosen large enough such that $f(s)\leq g(s)$ for all $s\in[0,1]$. So
	\begin{eqnarray*}
		f^{(n)}(s) \leq g^{(n)}(s) = \frac {a^n s}{a^n + (a+1)^n(1-s)},
	\end{eqnarray*}
	where the last equality can be checked by induction. Therefore, 
	\[
	f^{(n)}(1-n^{-\alpha}) \leq \frac {a^n}{a^n + n^{-\alpha}(a+1)^n }.
	\]
	It follows that $\decay{f^{(n)}(1-n^{-\alpha})}=\infty$. So the above discussion gives
	that $\omid{(1-n^{-\alpha})^{\card{\nei{n}{\bs o}}}}$ has infinite decay rate and the claim is proved.
\end{proof}

\subsubsection{The Poisson Weighted Infinite Tree}
\label{subsec:pwit}
The Poisson Weighted Infinite Tree (\texttt{PWIT}) is defined 
as follows (see e.g., \cite{objective}). It is a rooted tree $[\bs T, \bs o]$ such that the degree
of every vertex is infinite. Regarding $\bs T$ as a family tree with progenitor $\bs o$,
the edge lengths are as follows. For every $u\in \bs T$, the set $\{d(u,v): v\text{ is an offspring of } u\}$
is a Poisson point process on $\mathbb R^{\geq 0}$ with intensity function $x^k$, where $k>0$ is a given integer.
Moreover, for different vertices $u$, the corresponding Poisson point processes are jointly independent.
It is known that the \texttt{PWIT} is unimodular (notice that although each vertex has infinite degree, the \texttt{PWIT} is boundedly-finite as a metric space). See for example~\cite{objective} for more details.

\begin{proposition}
The \texttt{PWIT} satisfies $ \dimM{\texttt{PWIT}} = \dimH{\texttt{PWIT}} = \infty.$
\end{proposition}

\begin{proof}
Denote the neighbors of $\bs o$ by $v_1,v_2,\ldots$ such that $d(\bs o,v_i)$ is increasing in $i$.
It is straightforward that all moments of $d(\bs o,v_3)$ are finite. Therefore,
Proposition~\ref{prop:regree-distorted} implies that $\dimMl{\bs T}=\infty$
(see also {Lemma~\ref{lem:regtree-weight}}). This proves the claim.
\end{proof}

\subsection{The Graph of the Simple Random Walk}
\label{subsec:srw-graph}

\formir{As in Subsection~\ref{subsec:randomwalk}, consider the simple random walk $(S_n)_{n\in \mathbb Z}$
in $\mathbb R^k$, where $S_0=0$ and the increments $S_n-S_{n-1}$ are i.i.d.}
The graph of the random walk $(S_n)_{n\in\mathbb Z}$ is 
$ \Psi:=\{(n,S_n):n\in\mathbb Z\}\subseteq\mathbb R^{k+1}.$
It can be seen that $\Psi$ is a point-stationary point process, and hence,
$[\Psi,0]$ is unimodular (see Subsection~\ref{subsec:image}).

Since $\card{\Psi\cap [-n,n]^{k+1}}\leq 2n+1$, 
the mass distribution principle (Theorem~\ref{thm:mdp-simple}) implies that $\dimH{\Psi}\leq 1$.
In addition, if $S_1$ has finite first moment, then 
the strong law of large numbers implies that $\lim_n \frac 1 n S_n=\omid{S_1}$.
This implies that $\liminf_n \frac 1 n \card{\left(\Psi\cap [-n,n]^{k+1}\right)} >0$.
Therefore, the unimodular Billingsley lemma (Theorem~\ref{thm:billingsley})
implies that $\dimH{\Psi}\geq 1$. Hence, $\dimH{\Psi}=1$. 
\formir{This matches the result of~\cite{KhoXi17}
that the \textit{macroscopic dimension} of the graph of the Brownian motion is 1,
while its \textit{microscopic dimension} is $3/2$ when $k=1$ (see Subsection~\ref{subsec:connections}).}

Below, the focus is on the case $k=1$ and on the following metric:
\begin{equation}
\label{eq:anothermetric}
d((x,y),(x',y')):=\max \{ \sqrt{\norm{x-x'}}, {\norm{y-y'}}\}.
\end{equation}

Theorem~\ref{thm:metricChange} implies that, by considering this metric,
unimodularity is preserved and dimension is not decreased. 
Under this metric, the ball $N_n(0)$ is $\Psi\cap [-n^2,n^2]\times [-n, n]$.
{It is straightforward that $\mathbb Z^2$ has \formir{volume} growth rate 3 and also} Minkowski and Hausdorff dimension 3 under this metric.

\begin{proposition}
	If the jumps are $\pm 1$ uniformly, under the metric~\eqref{eq:anothermetric},
	the graph $\Psi$ of the simple random walk satisfies $\dimM{\Psi}= \dimH{\Psi}=2.$
\end{proposition}

\begin{proof}
	Let $n\in \mathbb N$. The ball $N_n(0)$ has at most $2n^2+1$ elements. So the mass distribution principle
	(Theorem~\ref{thm:mdp-simple}) implies that $\dimH{\Psi}\leq 2$. For the other side, let $\mathcal C$ be
	the equivariant disjoint covering of $\mathbb Z^2$ by translations of the rectangle $[-n^2,n^2]\times [-n,n]$
	(similar to Example~\ref{ex:lattice-Minkowski}). For each rectangle $\sigma\in \mathcal C$,
	select the right-most point in $\sigma\cap \Psi$ and let $\bs S=\bs S_{\Psi}$ be the set of selected points.
	By construction, $\bs S$ gives an $n$-covering of $\Psi$ and it can be seen that it is an equivariant covering.
	Let $\sigma_0$ be the rectangle containing the origin. By construction, $0\in \bs S$ if and only if it is
	either on a right-edge of $\sigma_0$ or on a horizontal edge of $\sigma_0$ and the random walk stays outside $\sigma_0$. 
	The first case happens with probability $1/(2n^2+1)$. By classical results concerning the hitting time of random walks,
	one can obtain that the probability of the second case lies between two constant multiples of $n^{-2}$.
	It follows that $\myprob{0\in \bs S}$ lies between two constant multiples of $n^{-2}$. 
	Therefore, $\dimMl{\Psi}\geq 2$. This proves the claim.
\end{proof}

\subsection{\formir{Other} Self Similar Unimodular Spaces}
\label{subsec:selfsimilar2}

In this subsection, two examples are presented which have some kind of self-similarity heuristically,
but do not fit into the framework of Subsection~\ref{subsec:selfsimilar}.

\subsubsection{Unimodular Discrete Spaces Defined by Digit Restriction}
\label{subsec:digits}
Let $J\subseteq\mathbb Z^{\geq 0}$. For $n\geq 0$, consider the set of natural numbers with expansion
$(a_na_{n-1}\ldots a_0)$ in base 2 such that $a_i=0$ for every $i\not \in J$. {Similarly to the examples in Subsection~\ref{subsec:selfsimilar},} one can shift this set 
randomly and take a limit to obtain a unimodular discrete space.
This can be constructed in the following way as well: Let $\bs T_0:=\{0\}$.
If $n\in J$, let $\bs T_{n+1}:=\bs T_n\cup (\bs T_n\pm 2\times 2^n)$,
where the sign is chosen i.i.d., each sign with probability $1/2$.
If $n\not\in J$, let $\bs T_{n+1}:=\bs T_n$. Finally, let $\Psi:=\cup_n \bs T_n$.

The upper and lower asymptotic densities of $J$ in $\mathbb Z^{\geq 0}$ are
defined by $\densityU{}(J):=\limsup_n \frac 1 nJ_n$ and
$\densityL{}(J):=\liminf_n \frac 1nJ_n$, where $J_n:= \card{J\cap \{0,\ldots,n\}}$.

\begin{proposition}
	\label{prop:digits}
	Almost surely,
	\begin{eqnarray*}
		\dimH{\Psi} = \dimMu{\Psi} = \growthu{\card{N_n(\bs o)}} &=& \densityU{}(J),\\
		\dimMl{\Psi} = \growthl{\card{N_n(\bs o)}} &=& \densityL{}(J).
	\end{eqnarray*}
\end{proposition}
{In particular, this provides another example of a unimodular discrete space where the (polynomial) \formir{volume} growth rate does not exist.}
\begin{proof}
	Let $n\geq 0$ be given. Cover $\bs T_n$ by a ball of radius $2^n$ centered at the minimal element of $\bs T_n$.
	By the same recursive definition, one can cover $\bs T_{n+1}$ by either 1 or 2 balls of the same radius. 
	Continuing the recursion, an equivariant $2^n$-covering $\bs R_n$ is obtained.
	It is straightforward to see that $\myprob{\bs R_n(\bs o)>0} = 2^{-J_n}$.
	Since these coverings are uniformly bounded (Definition~\ref{def:nearlyDisjoint}),
	Lemma~\ref{lem:nearlyDisjoint} implies that $\dimMl{\Psi} = \densityL{}(J)$ and $\dimMu{\Psi} = \densityU{}(J)$. 
	One has 
	\begin{equation}
	\label{eq:prop:digits}
	\card{\bs T_m} = 2^{J_m	}.
	\end{equation}
	This implies that $\card{N_{2^n}(\bs o)}\leq 2^{J_n+1}$. One can deduce that $\growthu{\card{N_n(\bs o)}} \leq \densityU{}(J)$.
	So the unimodular Billingsley lemma (Theorem~\ref{thm:billingsley}) gives $\dimH{\Psi}\leq \densityU{}(J)$. This proves the claim.
\end{proof}

\subsubsection{Randomized Discrete Cantor set}
\label{subsec:cantor-randomized}
This  subsection proposes a unimodular discrete analogue of the \textit{random Cantor set}, recalled below. Let $0\leq p\leq 1$ and $b>1$.  The {random Cantor set} in $\mathbb R^k$~\cite{Ha81} (see also~\cite{bookBiPe17})
is defined by $\Lambda_k(b,p):=\cap_n E_n$, where $E_n$ is defined by the following
random algorithm: Let $E_0:=[0,1]^k$. For each $n\geq 0$ and each \textit{$b$-adic}
cube of edge length $b^{-n}$ in $E_n$, divide it into $b^k$ smaller $b$-adic cubes
of edge length $b^{-n-1}$. Keep each smaller $b$-adic cube with probability $p$
and delete it otherwise independently from the other cubes. Let $E_{n+1}$ be
the union of the kept cubes. It is shown in Section~3.7 of~\cite{bookBiPe17}
that $\Lambda_k(b,p)$ is empty for $p\leq b^{-k}$ and otherwise,
has dimension $k+\log_b p$ conditioned on being non-empty.

For each $n\geq 0$, let $\bs K_n$ be the set of lower left corners of the $b$-adic cubes
forming $E_n$. It is easy to show that $\bs K_n$ tends to $\Lambda_k(b,p)$ a.s. under the Hausdorff metric. 

\begin{proposition}
\label{prop:randomCantor}
Let $\bs K'_n$ denote the random set obtained by biasing the distribution
of $\bs K_n$ by $\card{\bs K_n}$ {\formir{(Definition~\ref{def:bias})}}.
Let $\bs o'_n$ be {a point chosen uniformly at random in $\bs K'_n$.}
\begin{enumerate}[(i)]
\item $[b^n \bs K'_n, \bs o'_n]$ converges {weakly} to some unimodular discrete space $[\hat{\bs K}, \hat{\bs o}]$. 
\item If $p<b^{-k}$, then $\hat{\bs K}$ is finite a.s., hence, $\dimH{\hat{\bs K}}=0$ a.s.
\item If $p\geq b^{-k}$, then $\hat{\bs K}$ is infinite a.s. and 
\[ \dimH{\hat{\bs K}} = \dimM{\hat{\bs K}} =k+\log_b p, \quad a.s.  \]
\end{enumerate}
\end{proposition}

Note that in contrast to the continuum analogue~\cite{Ha81},
for $p=b^{-k}$, the set is non-empty and even infinite,
though still zero dimensional. Also, for $p<b^{-k}$ the set is non-empty as well.

To prove the above proposition, the following construction of $\hat{\bs K}$ will be used.
First, consider the usual nested sequence of partitions $\Pi_n$ of $\mathbb Z^k$ 
by translations of the cube $\{0,\ldots,b^n-1\}^k$, where $n\geq 0$. To make it \textit{stationary}, 
shift each $\Pi_n$ randomly as follows. Let $a_0,a_1,\ldots\in \{0,1,\ldots, b-1\}^k$
be i.i.d. uniform numbers and let $\bs U_n=\sum_{i=0}^n a_ib^i \in \mathbb Z^k$.
Shift the partition $\Pi_n$ by the vector $\bs U_n$ to form a partition denoted by $\Pi'_n$.
It is easy to see that $\Pi'_n$ is a nested sequence of partitions.

\begin{lemma}
	\label{lem:randomCantor}
	Let $(\Pi'_n)_n$ be the stationary nested sequence of partitions of $\mathbb Z^k$ defined above.
	For each $n\geq 0$ and each cube $C\in \Pi'_n$ that does not contain the origin, with probability $1-p$
	(independently for different choices of $C$), mark {all points in} $C\cap \mathbb Z^k$ for deletion. Then,
	the set of the unmarked points of $\mathbb Z^k$, \rooted{} at the origin,
	has the same distribution as $[\hat{\bs K}, \hat{\bs o}]$ defined in Proposition~\ref{prop:randomCantor}.
\end{lemma}

\begin{proof}[Proof of Lemma~\ref{lem:randomCantor}]
	Let $\Phi$ be the set of unmarked points in the algorithm.
	For $n\geq 0$, let $\bs C_n$ be the cube in $\Pi'_n$ that contains the origin.
	It is proved below that $\bs C_n\cap \Phi$ has the same distribution as $b^n(\bs{K}'_n -\bs o_n)$. This implies the claim.
	
	Let $A_n\subseteq [0,1]^k$ be the set of possible outcomes of $\bs o'_n$. One has $\card{A_n} = b^{kn}$.
	For $v\in A_n$, it is easy to see that the distribution of $b^n(\bs{K}'_n -\bs o_n)$, conditioned
	on $\bs o'_n=v$, coincides with the distribution of $\bs C_n\cap \Phi$ conditioned on $\bs C_n = b^n([0,1)^k-v)$.
	So it remains to prove that $\myprob{\bs o'_n=v}= \myprob{\bs C_n = b^n([0,1)^k-v)}$, which is left to the reader.
\end{proof}

Here is another description of $\hat{\bs K}$. The nested structure of $\bigcup_n \Pi'_n$ defines a tree as follows.
The set of vertices is $\bigcup_n \Pi'_n$. For each $n\geq 0$, connect (the vertex corresponding to) every
cube in  $\Pi'_n$ to the unique cube in $\Pi'_{n+1}$ that contains it. This tree is the canopy tree
(Subsection~\ref{subsec:canopy}) with offspring cardinality $N:=b^k$, except that the root
(the cube $\{0\}$) is always a leaf. Now, keep each vertex with probability $p$ and remove it
with probability $1-p$ in an i.i.d. manner. Let $\bs T$ be the connected component of the
remaining graph that contains the root. Conditioned on the event that $\bs T$ is infinite,
$\hat{\bs K}$ corresponds to the set of leaves in the connected component of the root.

\begin{proof}[Proof of Proposition~\ref{prop:randomCantor}]
	The unimodular Billingsley lemma is used to get an upper bound on the Hausdorff dimension.
	For this $\omid{\card{N_{b^n}(\bs o)}}$ is studied. Consider the tree $[\bs T, \bs o]$ defined above and
	obtained by the percolation process on the canopy tree with offspring cardinality $N:=b^k$. Let $C$ be any cube
	in $\Pi'_i$ that does not contain the origin. Note that the subtree of descendants of $C$ in the percolation cluster
	(conditioned on keeping $C$) is a Galton-Watson tree with binomial offspring distribution with parameters $(N,p)$.
	Classical results on branching processes say $\omidCond{\card{C\cap \hat{\bs K}}}{\Pi'_i} = p m^i$, where $m:=p b^k$.
	So the construction implies that
	\[
	\omid{\card{\bs C_n\cap \hat{\bs K}}} = 1+ p(N-1)\left(m^{n-1}+m^{n-2}+\cdots + 1 \right).
	\]
	For $m>1$, the latter is bounded by $l m^n$
	for some constant $l$ not depending on $n$. Note that $N_{b^n}(\bs o)$ is contained 
	in the union of $\bs C_n$ and $3^k-1$ other cubes in $\Pi'_n$. It follows that 
	$\omid{\card{N_{b^n}(\bs o)}}\leq l' m^n$, where $l'=l+(3^k-1)p$. So the unimodular Billingsley lemma
	(Theorem~\ref{thm:billingsley}) implies that $\dimH{\hat{\bs K}}\leq k+\log_b p$. The claim for $m=1$ and $m<1$ are similar.
	
	Consider now the Minkowski dimension. {As above, we assume $m>1$ and the proofs for the other cases are similar.}
	Let $n\geq 0$ be given. By considering the partition $\Pi'_n$ by cubes, one can construct a $b^n$-covering
	$\bs R_n$ as in Theorem~\ref{thm:lowerBoundR^d}. This covering satisfies
	$
	\myprob{\bs R_n(\bs o)\geq 0} = \omid{ 1/{\card{(\bs C_n \cap \hat{\bs K})}}}.
	$
	Let $[\bs T', \bs o']$ be the eternal Galton-Watson tree of Subsection~\ref{subsec:egw} 
	with binomial offspring distribution with parameters $(N,p)$. By regarding $\bs T'$ as a family tree,
	it is straightforward that $[\bs T, \bs o]$ has the same distribution as the part of $[\bs T', \bs o']$,
	up to the generation of the root (see~\cite{eft} for more details on eternal family trees).
	Therefore, Lemma~5.7 of~\cite{eft} implies that 
	$
	\omid{1/{\card{(\bs C_n \cap \hat{\bs K})}}} = m^{-n} \myprob{h(\bs o')\geq n}.
	$
	Since $m>1$, $\myprob{h(\bs o')\geq n}$ tends to the non-extinction probability of the descendants of the root, which is positive.
	By noticing the fact that the radii of the balls are $b^n$ and the covering is uniformly bounded,
	one gets that $\dimM{\hat{\bs K}}= \log_b m = k+\log_b p$.
	
	Finally, it remains to prove that $\hat{\bs K}$ is infinite a.s. when $p=b^{-k}$.
	In this case, consider the eternal Galton-Watson tree $[\bs T',\bs o']$ as above.
	Proposition~6.8 of~\cite{eft} implies that the generation of the root is infinite a.s. This proves the claim.
\end{proof}

\subsection{Cayley Graphs}
\label{subsec:cayley}

As mentioned in Subsection~\ref{subsec:metricChange},  the dimension of a Cayley graph depends
only on the group and not on the generating set. The following result connects it to the \formir{volume} growth
rate of the group. Note that 
Gromov's theorem~\cite{Gr81} implies that the polynomial growth degree exists and is
either an integer or infinity.

\begin{theorem}
\label{thm:Cayley}
For every finitely generated group $H$
with polynomial growth {degree} $\alpha\in [0,\infty]$, one has $ \dimM{H} = \dimH{H} = \alpha $.
\formir{Also, if $\alpha<\infty$, then} $\measH{\alpha}(H)<\infty$.
\end{theorem}
\begin{proof}
	First, assume $\alpha<\infty$. The result of Bass~\cite{Ba72} implies that there are constants $c,C>0$
	such that $\forall r\geq 1: cr^{\alpha}<\card{N_r(o)} \leq Cr^{\alpha}$, where $o$ is an arbitrary element of $H$.
	So the mass distribution principle (Theorem~\ref{thm:mdp-simple}) and part~\eqref{part:lem:lowerbound:all}
	of Lemma~\ref{lem:lowerboundiid} imply that $\dimM{H} = \dimH{H}=\alpha$. 
	{In addition, \eqref{eq:thm:mdp-simple} in the proof of Theorem~\ref{thm:mdp-simple} implies that $\contentH{\alpha}{M}(H)\geq 1/C$ for all $M\geq 1$, which implies that $\measH{\alpha}(H)\leq C<\infty$.}
	
	Second, assume $\alpha=\infty$. The result of \cite{VaWi84} shows that for any $\beta<\infty$,
	$\card{N_r(o)}>r^{\beta}$ for sufficiently large $r$.
	Therefore, part~\eqref{part:lem:lowerbound:all} of Lemma~\ref{lem:lowerboundiid} implies
	that $\dimMl{H}\geq \beta$. Hence, $\dimM{H}=\dimH{H}=\infty$ and the claim is proved.
\end{proof}

{
	It is natural to expect that $\measH{\alpha}(H)>0$ as well, but only a weaker inequality will be proved in Proposition~\ref{prop:Cayley}. 
}


\subsection{Notes and Bibliographical Comments}

{The proof of {Proposition~\ref{prop:tree-expGrowth}} was suggested by R. Lyons.} 
Bibliographical comments on {some of} the examples discussed in this
section can be found at the end of Section \ref{sec:examples}. {The example defined by digit restriction (Subsection~\ref{subsec:digits}) is inspired by an example in the continuum setting (see e.g., Examples~1.3.2 of~\cite{bookBiPe17}). The randomized discrete Cantor set (Subsection~\ref{subsec:cantor-randomized}) is inspired by the \textit{random cantor set} (see e.g., Section~3.7 of~\cite{bookBiPe17}).}

\section{Frostman's Theory}
\label{sec:frostman}

This section provides a unimodular version of Frostman's lemma and some of its applications.
In a sense to be made precise later, this lemma gives converses to the mass distribution principle.
It is a powerful tool {in the theoretical analysis of} the unimodular Hausdorff dimension.
For example, it is used in this section to derive inequalities for the dimension of product spaces 
and \textit{embedded spaces} (Subsections~\ref{subsec:product} and~\ref{subsec:embedded}).
It is also the basis of many of the results in~\cite{III}.


\subsection{Unimodular Frostman Lemma}
\label{subsec:frostman-general}

The statement of the unimodular Frostman lemma requires the definition
of \textit{weighted Hausdorff} content. The latter is based on the notion of
\textit{equivariant weighted collections of balls} as follows.
For this, the following mark space is needed. Let $\Xi$ be the set of functions
$c:\mathbb R^{\geq 0}\to \mathbb R^{\geq 0}$ which are positive in only finitely many points;
i.e., $c^{-1}((0,\infty))$ is a finite set.  
Remark~\ref{rem:frostman-markspace} below defines a metric on $\Xi$, so that
the notion of $\Xi$-valued equivariant processes (Definition~\ref{def:equivProcess}) is well defined.
Such a process $\bs c$ is called an \textbf{equivariant weighted collection of balls}
\footnote{The term `weighted' refers to the weighted sums in Definition~\ref{def:xi^alpha}
and should not be confused with equivariant weight functions of Definition~\ref{def:weight}}.
Consider a unimodular discrete space $[\bs D, \bs o]$ with distribution $\mu$.
For $v\in \bs D$, the reader can think of the value $\bs c_r(v):=\bs c(v)(r)$, if positive, to indicate that
there is a ball in the collection, with radius $r$, centered at $v$,
and with {\textbf{cost} (or weight)} $\bs c_r(v)$.
Note that extra randomness is allowed in the definition.
A ball-covering $\bs R$ can be regarded a special case of this construction by letting
$\bs c_r(v)$ be 1 when $r=\bs R(v)$ and 0 otherwise.

\begin{definition}
\label{def:xi^alpha}
Let $f:\mathcal D_*\rightarrow\mathbb R$ be a measurable function and $M\geq 1$. 
An equivariant weighted collection of balls $\bs c$ is called a \defstyle{$(f,M)$-covering} if 
\begin{equation}
\label{eq:weightedCovering}
\forall v\in \bs D: f(v) \leq \sum_{u\in\bs D}\sum_{r\geq M} \bs c_r(u)\identity{\{v\in N_r(u)\}}, \quad a.s.,
\end{equation}
where $f(v):=f[\bs D,v]$ for $v\in\bs D$. For $\alpha\geq 0$, define 
\begin{eqnarray*}
\xi^{\alpha}_M(f)&:=&\inf\left\{ \omid{\sum_{r} \bs c_r(\bs o)r^{\alpha}}:
\bs c \text{ is a } (f,M)\text{-covering}\right\},\\
\xi^{\alpha}_{\infty}(f)&:=& \lim_{M\rightarrow\infty} \xi^{\alpha}_M(f).
\end{eqnarray*}
\end{definition}

It is straightforward that every equivariant ball-covering of Definition~\ref{def:convering}
gives a $(1,1)$-covering, where \formir{the first $1$} is regarded as the constant function $f\equiv 1$ on $\dstar$.
This gives (see also Conjecture~\ref{conj:frostman} below) \formir{that for all $M\ge 1$},
\begin{equation}
\label{eq:frostman-xi<H}
\xi^{\alpha}_M(1) \leq \contentH{\alpha}{M}(\bs D).
\end{equation}

Also, by considering the case $\bs c_M(v) := f(v)\vee 0$, one can see that if $f\in L_1(\mathcal D_*,\mu)$, then
\[
\xi^{\alpha}_M(f)\leq M^{\alpha} \omid{f(\bs o)\vee 0} <\infty.
\]
In the next theorem, to be consistent with the setting of the paper,
the following notation is used: $w(u):=w([D,u])$ for $u\in D$, and $w(N_r(v))=\sum_{u \in N_r(v)} w(u)$.
Also, recall that a deterministic equivariant weight function is given by a measurable function
$w:\dstar\to\mathbb R^{\geq 0}$ (see Example~\ref{ex:equiv-basic}).  

\begin{theorem}[Unimodular Frostman Lemma]
\label{thm:frostmanGeneral}
Let $[\bs D, \bs o]$ be a unimodular discrete space, $\alpha\geq 0$ and $M\geq 1$.
\begin{enumerate}[(i)]
\item \label{thm:frostmanGeneral:1}
There exists a bounded measurable weight function $w:\mathcal D_*\rightarrow\mathbb R^{\geq 0}$
such that {$\omid{w(\bs o)} = \xi^{\alpha}_M(1)$ and} almost surely,
\begin{equation}
\label{eq:frostman:1}
\forall v\in \bs D,\  \forall r\geq M,  w(N_r(v))\leq r^{\alpha}.
\end{equation}
\item \label{thm:frostmanGeneral:3}
In addition, if either
$\dimH{\bs D}<\alpha$ or $\alpha=\dimH{\bs D}$ and $\measH{\alpha}(\bs D)<\infty$, 
then $w[\bs D,\bs o]\neq 0$ with positive probability.
\end{enumerate}
\end{theorem}

The proof is given later in this subsection.

\begin{remark}
	\label{rem:frostman-ineq}
	One can show that if $\bs w$ is an equivariant weight function satisfying~\eqref{eq:frostman:1},
	then $\omid{\bs w(\bs o)}  \leq  \xi^{\alpha}_M(1)$ and $\omid{\bs w(\bs o)h(\bs o)} \leq  \xi^{\alpha}_M(h)$.
	Therefore, the (deterministic) weight function $w$ given in the unimodular Frostman lemma
	is a maximal equivariant weight function satisfying~\eqref{eq:frostman:1}
	(it should be noted that such maximal functions are not unique in general). 
	The proof is similar to that of the mass distribution principle (Theorem~\ref{thm:mdp-simple})
	and is left to the reader.
\end{remark}

\begin{conjecture}
	\label{conj:frostman}
	One has $\contentH{\alpha}{M}(\bs D)=\xi^{\alpha}_M(1)$.
\end{conjecture}

Here are a few comments on this conjecture. \formir{An analogous equality holds in the continuum setting (see 2.10.24 in~\cite{bookFe69}).}
{Lemma~\ref{lem:frostmanAuxiliary} below proves a weaker inequality.}
{It can be seen that the conjecture holds for $\mathbb Z^k$ (with the $l_{\infty}$ metric) and for the non-ergodic example of  Example~\ref{ex:nonergodic}. In the former case, this is obtained by considering the constant weight function $w(\cdot)\equiv \left(\frac{M}{2M+1}\right)^k$, which satisfies the claim of the unimodular Frostman lemma. The latter case is similar by letting $w[\mathbb Z,0]:=\frac M{2M+1}$ and $w[\mathbb Z^2,0]:=0$.}

\begin{remark}
	\label{rem:frostman-mdp-billingsley}
	{The unimodular Frostman lemma implies that, in theory, the mass distribution principle (Theorem~\ref{thm:mdp-simple}) is enough for bounding the Hausdorff dimension from above. However, there are very few examples in which the function $w$ given by the unimodular Frostman lemma can be explicitly computed (in some of the examples, a function $w$ satisfying {only}~\eqref{eq:frostman:1} can be found; e.g., for two-ended trees). Therefore, in practice, the unimodular Billingsley lemma is more useful than the mass distribution principle.}
\end{remark}

{\formir{The following lemma is needed to prove Theorem~\ref{thm:frostmanGeneral}.}}

\begin{lemma}
	The function $\xi^{\alpha}_M: L_1(\mathcal D_*,\mu)\rightarrow\mathbb R$ is continuous.
	In fact, it is $M^{\alpha}$-Lipschitz; i.e., 
	$
	\norm{\xi^{\alpha}_M(f_1) - \xi^{\alpha}_M(f_2)} \leq M^{\alpha} \omid{\norm{f_1(\bs o)-f_2(\bs o)}}.
	$
\end{lemma}
\begin{proof}
	Let $\bs c$ be an equivariant weighted collection of balls satisfying~\eqref{eq:weightedCovering} for $f_1$.
	Intuitively, add a ball of radius $M$ at each point $v$ with cost $\norm{f_2(v)-f_1(v)}$.
	More precisely, let $\bs c'_r(v):=\bs c_r(v)$ for $r\neq M$ and $\bs c'_M(v):=\bs c_M(v)+\norm{f_2(v)-f_1(v)}$.
	This definition implies that $\bs c'$ satisfies~\eqref{eq:weightedCovering} for $f_2$. Also,
	\[
	\xi^{\alpha}_M(f_2)\leq \omid{\sum_{i} \bs c'_i(\bs o)i^{\alpha}} = \omid{\sum_{r} \bs c_r(\bs o)i^{\alpha}} + M^{\alpha}\omid{\norm{f_2(\bs o)-f_1(\bs o)}}.
	\]
	Since $\bs c$ is arbitrary, one obtains 
	$
	\xi^{\alpha}_M(f_2)\leq \xi^{\alpha}_M(f_1) + M^{\alpha}\omid{\norm{f_2(\bs o)-f_1(\bs o)}},
	$
	which implies the claim.
\end{proof}

\begin{proof}[Proof of Theorem~\ref{thm:frostmanGeneral}]
	The theorem is a special case of Proposition~\ref{prop:frostmangeneral} which is proved below.
\end{proof}

\begin{proposition}
	\label{prop:frostmangeneral}
	In the setting of Theorem~\ref{thm:frostmanGeneral}, let $h\in L_1(\mathcal D_*,\mu)$
	be any given function such that $h> 0$ a.s. Then, one can replace the condition
	$\omid{w(\bs o)} = \xi^{\alpha}_M(1)$ in Theorem~\ref{thm:frostmanGeneral} by 
	$\omid{w(\bs o)h(\bs o)} = \xi^{\alpha}_M(h)$, and the conclusions of the theorem are valid.
	%
	%
\end{proposition}

\begin{proof}
	It is easy to see that $\xi^{\alpha}_M(tf)=t\xi^{\alpha}_M(f)$ for all $f$ and $t\geq 0$ and also 
	$
	\xi^{\alpha}_M(f_1+f_2) \leq \xi^{\alpha}_M(f_1)+\xi^{\alpha}_M(f_2)
	$
	for all $f_1,f_2$. Let $h\in L_1(\mathcal D_*,\mu)$ be given. By the Hahn-Banach theorem
	(see Theorem~3.2 of~\cite{bookRu73}), there is a linear functional $l:L_1(\mathcal D_*,\mu)\rightarrow\mathbb R$ such that 
	$
	l(h)=\xi^{\alpha}_M(h)
	$
	and
	$
	-\xi^{\alpha}_M(-f) \leq l(f) \leq \xi^{\alpha}_M(f),
	$
	{for all $f\in L_1$.}
	Since $l$ is sandwiched between two functions which are continuous  at $0$ and are equal at $0$ 
	(since $\xi^{\alpha}_M(0)=0$),
	one gets that $l$ is continuous at 0. Since $l$ is linear, 
	this implies that $l$ is continuous. Since the dual of $L_1(\mathcal D_*, \mu)$ is $L_{\infty}(\mathcal D_*, \mu)$,
	one obtains that there is a function $w\in L_{\infty}(\mathcal D_*, \mu)$ such that  
	$
	l(f) = \omid{f(\bs o)w(\bs o)}, 
	$
	{for all $f\in L_1$.}
	Note that if $f\geq 0$, then $\xi^{\alpha}_M(-f) = 0$ and so $l(f)\geq 0$.
	This implies that $w(\bs o)\geq 0$ a.s. (otherwise, let $f(\bs o):=\identity{\{w(\bs o)<0\}}$ to get a contradiction). 
	Consider a version of $w$ which is nonnegative everywhere.
	The claim is that $w$ satisfies the requirements. 
	
	Let $r\geq M$ be fixed.
	For all discrete spaces $D$, let $\bs S:=\bs S_D:=\{v\in D: w(N_r(v))>r^{\alpha}\}$.
	By the definition of $\bs S_D$, one has
	\begin{equation}
	\label{eq:thm:frostmanGeneral:1}
	\omid{w(N_r(\bs o))\identity{\{\bs o\in \bs S\}}} \geq r^{\alpha}\myprob{\bs o \in \bs S}.
	\end{equation}	
	Moreover, if $\myprob{\bs o \in \bs S}>0$, then the inequality is strict.
	Let $f_r(v):= \card{N_r(v)\cap \bs S}$. By the mass transport principle for the function
	$(v,u)\mapsto w(u)\identity{\{v\in \bs S\}} \identity{\{u\in N_r(v) \}}$, one gets
	\begin{eqnarray*}
		\omid{w(N_r(\bs o))\identity{\{\bs o\in \bs S\}}} &=& \omid{w(\bs o) \card{N_r(\bs o)\cap \bs S}}\\
		&=& \omid{w(\bs o)f_r(\bs o)}\\
		&=& l(f_r)\\
		&\leq & \xi^{\alpha}_M(f_r)\\
		&\leq & r^{\alpha} \myprob{\bs o \in \bs S},
	\end{eqnarray*}
	where the last inequality is implied by considering the following weighted collection of balls 
	for $f_r$: put balls of radius $r$ with cost 1 centered at the points in $\bs S$. 
	More precisely, let $\bs c_r(v):=\identity{\{v\in \bs S\}}$ and $\bs c_s(v):=0$ for $s\neq r$.
	It is easy to see that this satisfies~\eqref{eq:weightedCovering} for $f_r$,
	which implies the last inequality by the definition of $\xi^{\alpha}_M(\cdot)$.
	Thus, equality holds in~\eqref{eq:thm:frostmanGeneral:1}. Hence, $\myprob{\bs o\in \bs S}=0$;
	i.e., $w(N_r(\bs o))\leq r^{\alpha}$ a.s. Lemma~\ref{lem:happensAtRoot} implies that almost surely,
	$\forall v\in \bs D: w(N_r(v))\leq r^{\alpha}$. So the same holds for all rational $r\geq M$ simultaneously.
	By monotonicity {of $w(N_r(v))$ w.r.t. $r$}, one gets that the latter almost surely holds for all $r\geq M$ as desired.
	Also, one has 
	$
	\omid{w(\bs o)h(\bs o)} = l(h) = \xi^{\alpha}_M(h).
	$
	Thus, $w$ satisfies the desired requirements.
	
	To prove~\eqref{thm:frostmanGeneral:3}, assume $\measH{\alpha}(\bs D)<\infty$.
	By Lemma~\ref{lem:Hmeas-elementary}, one has $\contentH{\alpha}{M}(\bs D)>0$.
	So Lemma~\ref{lem:frostmanAuxiliary} below implies that $0<\xi^{\alpha}_M(h)=\omid{w(\bs o)h(\bs o)}$.
	This implies that $w$ is not identical to zero.
\end{proof}

{\formir{The above proof uses the following lemma.}}

\begin{lemma}
	\label{lem:frostmanAuxiliary}
	
	Let $[\bs D, \bs o]$ be a unimodular discrete metric space.
	\begin{enumerate}[(i)]
		\item \label{lem:frostmanAuxiliary:1}
		By letting $b:=\xi^{\alpha}_1(1)$, one has 
		$
		b\leq \contentH{\alpha}{1}(\bs D)\leq b + {b}{\norm{\log b}}.
		$
		\item \label{lem:frostmanAuxiliary:2}
		Let 
		$h\in L_1(\mathcal D_*,\mu)$ be a non-negative function. 
		For $M\geq 1$, one has
		\begin{equation}
		\label{eq:frostmanAuxiliary}
		\contentH{\alpha}{M}(\bs D)\leq \inf_{a\geq 0}\left\{M^{\alpha}\omid{e^{-a h(\bs o)}}  + a \xi^{\alpha}_M(h)\right\}.
		\end{equation}
		\item \label{lem:frostmanAuxiliary:3}
		In addition, if $h>0$ a.s., then $\xi^{\alpha}_M(h)=0$ if and only if $\contentH{\alpha}{M}(\bs D)=0$.	
	\end{enumerate}
\end{lemma}
\begin{proof}
	\eqref{lem:frostmanAuxiliary:1}. 
	The first inequality is easily obtained from the definition of $\xi^{\alpha}_1(1)$
	by considering the cases where $\bs c(\cdot)\in \{0,1\}$. In particular, this implies that $b\leq 1$.
	The second inequality is implied by part~\eqref{lem:frostmanAuxiliary:2} by letting $h(\cdot):=1$ and $a:=-\log b\geq 0$.
	
	\eqref{lem:frostmanAuxiliary:2}. 
	Let $b'>\xi^{\alpha}_M(h)$ be arbitrary. So there exists an equivariant weighted collection
	of balls $\bs c$ 
	that satisfies~\eqref{eq:weightedCovering} for $h$ and 
	$
	\omid{\sum_{r\geq M} \bs c_r(\bs o)r^{\alpha}} \leq b'.
	$
	Next, given $a\geq 0$, define an equivariant covering $\bs R$ as follows.
	For each $v\in \bs D$ and $r\geq M$ such that $\bs c_r(v)>0$, put a ball of radius $r$ at $v$ with probability
	$a \bs c_r(v) \wedge 1$. Do this independently for all $v$ and $r$
	(one should condition on $\bs D$ first). If more than one ball is put at $v$,
	keep only the one with maximum radius. Let $\bs S$ be the union of the chosen balls.
	For $u\in\bs D\setminus \bs S$, put a ball of radius $M$ at $u$. This gives an equivariant covering,
	namely $\bs R$, by balls of radii at least $M$. Then, one gets
	\begin{equation}
	\label{eq:lem:frostmanGeneral:2}
	\omid{\bs R(\bs o)}^{\alpha} \leq M^{\alpha}\myprob{\bs o\not\in \bs S} + \omid{\sum_{r\geq M} (a\bs c_r(\bs o)\wedge 1)r^{\alpha}} \leq M^{\alpha} \myprob{\bs o\not\in \bs S} + ab'.
	\end{equation}
	
	To bound $\myprob{\bs o\not\in \bs S}$, consider a realization of $[\bs D, \bs o]$.
	First, if for some $v\in \bs D$ and $r\geq M$, one has $a \bs c_r(v)>1$ and $\bs o\in N_r(v)$,
	then $\bs o$ is definitely in $\bs S$. Second, assume this is not the case.
	By~\eqref{eq:weightedCovering}, one has $\sum_{u\in\bs D}\sum_{r\geq M} \bs c_r(u)\identity{\{\bs o\in N_r(u)\}}\geq h(\bs o)$.
	This implies that the probability that $\bs o\not\in \bs S$ in this realization is
	\begin{eqnarray*}
		\prod_{(v,r): \bs o\in N_r(v)} (1-a \bs c_r(v))
		\leq  \mathrm{exp}\left( -\sum_{(v,r): \bs o\in N_r(v)} a\bs c_r(v) \right)
		\leq e^{-a h(\bs o)}.
	\end{eqnarray*}
	In both cases, one gets $\myprob{\bs o\not\in \bs S} \leq \omid{e^{-a h(\bs o)}}$. 
	Thus, \eqref{eq:lem:frostmanGeneral:2} implies 
	that $\omid{R(\bs o)}^{\alpha} \leq M^{\alpha} \omid{e^{-a h(\bs o)}} + ab'$.
	Since $a\geq 0$ and $b'>b$ are arbitrary, the claim follows.
	
	\eqref{lem:frostmanAuxiliary:3}. 
	Assume $\xi^{\alpha}_M(h)=0$. By letting $a\rightarrow\infty$ in~\eqref{eq:frostmanAuxiliary}
	and using dominated convergence,
	one obtains that $\contentH{\alpha}{M}(\bs D)=0$. Conversely, assume $\contentH{\alpha}{M}(\bs D)=0$. 
	The first inequality in~\eqref{lem:frostmanAuxiliary:1} gives that $\xi^{\alpha}_M(a)=0$ for any constant $a$.
	Therefore, $\xi^{\alpha}_M(h)\leq \xi^{\alpha}_M(a)+\xi^{\alpha}_M((h-a)\vee 0))\leq M^{\alpha}\omid{(h-a)\vee 0}$.
	By letting $a$ tend to infinity, one gets $\xi^{\alpha}_M(h)=0$.
\end{proof}

\begin{remark}
	\label{rem:frostman-markspace}
	In this subsection, the following metric is used on the mark space $\Xi$.
        Let $\Xi'$ be the set of finite measures on $\mathbb R^2$.
	By identifying $c\in \Xi$ with the counting measure on the finite set
        $\{(x,c(x)): x\in\mathbb R^{\geq 0}, c(x)>0\}$,
	one can identify $\Xi$ with a Borel subset of $\Xi'$. It is well known that $\Xi'$ is a
        complete separable metric space under the Prokhorov metric (see e.g., \cite{bookDaVe03I}).
	So one can define the notion of $\Xi'$-valued equivariant processes as in
        Definition~\ref{def:equivProcess}. Therefore, $\Xi$-valued equivariant processes also make sense.
\end{remark}

\subsection{Max-Flow Min-Cut Theorem for Unimodular One-Ended Trees}
\label{subsec:maxflow}

The result of this subsection is used in the next subsection for a Euclidean version of the unimodular Frostman lemma, but is of independent interest as well.

The max-flow min-cut theorem is a celebrated result in the field of graph theory (see e.g., \cite{FoFu62}).
In its simple version, it studies the minimum number of edges in a \textit{cut-set} in a finite graph;
i.e., a set of edges the deletion of which disconnects two given subsets of the graph.
A generalization of the theorem in the case of trees is obtained by considering cut-sets separating a given 
finite subset from the set of ends of the tree. This generalization is used to prove a version of
Frostman's lemma for compact sets in the Euclidean space (see e.g., \cite{bookBiPe17}).

This subsection presents an analogous result for unimodular one-ended trees.
It discusses cut-sets separating the set of leaves from the end of the tree. 
Since the tree has infinitely many leaves a.s. (see e.g., \cite{eft}), infinitely many edges
are needed in any such cut-set. Therefore, cardinality cannot be used to study minimum cut-sets.
The idea is to use unimodularity for a quantification of the size of a cut-set. 

Let {$[\bs T, \bs o; \bs c]$} be a unimodular {marked} one-ended tree with mark space $\mathbb R^{\geq 0}$.  
Assume the mark $\bs c(e)$ of each edge $e$ is well defined and call it
the \textbf{conductance} of $e$. Let $\bs L$ be the set of leaves of $\bs T$. As in Subsection~\ref{subsec:one-ended},
let $F(v)$ be the parent of vertex $v$ and $D(v)$ be the descendants subtree of $v$.

\begin{definition}
	A \textbf{legal equivariant flow} on $[\bs T;\bs c]$ is an equivariant way of 
	assigning extra marks $\bs f(\cdot)\in\mathbb R$ to the edges (see Definition~\ref{def:equivProcess}
	and Remark~\ref{rem:equivProcessExtraMarks}), such that almost surely,
	\begin{enumerate}[(i)]
		\item for every edge $e$, one has $0\leq \bs f(e)\leq \bs c(e)$,
		\item for every vertex $v\in \bs T\setminus\bs L$, one has 
		\begin{equation}
		\label{eq:flowcons}
		\bs f(v,F(v))=\sum_{w\in F^{-1}(v)}\bs f(w,v).
		\end{equation}
	\end{enumerate}
	Also, an \textbf{equivariant cut-set} is an equivariant subset $\Pi$ of the edges of $[\bs T; \bs c]$
	that separates the set of leaves $\bs L$ from the end in $\bs T$. 
\end{definition}

Note that extra randomness is allowed in the above definition. 
The reader can think of the value $\bs f(v,F(v))$ as the \textit{flow} from $v$ to $F(v)$.
So~\eqref{eq:flowcons} can be interpreted as \textit{conservation of flow} at the vertices except the leaves.
Also, the leaves are regarded as the \textit{sources} of the flow.

Since the number of leaves is infinite a.s., the sum of the flows exiting the
leaves might be infinite. In fact, it can be seen that unimodularity implies
that the sum is always infinite a.s. The idea is to use unimodularity to quantify how \textit{large} is the flow.
Similarly, in any equivariant cut-set, the sum of the conductances of the edges is infinite a.s.
Unimodularity is also used to quantify the \textit{conductance} of an equivariant cut-set.
These are done in Definition~\ref{def:flowNorm} below.

Below, since each edge of $\bs T$ can be uniquely represented as $(v,F(v))$, the following convention is helpful.

\begin{convention}
	For the vertices $v$ of $\bs T$, the symbols $\bs f(v)$ and $\bs c(v)$ are used as
	abbreviations for $\bs f(v,F(v))$ and $\bs c(v,F(v))$, respectively.
	Also, by $v\in \Pi$, one means that the edge $(v,F(v))$ is in $\Pi$.
\end{convention}

\begin{definition}
	\label{def:flowNorm}
	The \textbf{norm} of the legal equivariant flow $\bs f$ is defined as
	\[
	\norm{\bs f}:=\omid{\bs f(\bs o) \identity{\{\bs o\in \bs L \}}}.
	\]
	Also, for the equivariant cut-set $\Pi$, define
	\[
	\bs c(\Pi):=\omid{\bs c(\bs o)\identity{\{\bs o\in \Pi \}}} = \omid{\sum_{w\in F^{-1}(\bs o)} \bs c(w) \identity{\{w\in \Pi \}} },
	\]
	where the last equality follows from the mass transport principle~(\ref{eq:unimodularMarked}).
\end{definition}

An equivariant cut-set $\Pi$ is called \defstyle{equivariantly minimal} if there is no other
equivariant cut-set which is a subset of $\Pi$ a.s. 
If so, it can be seen that it is \defstyle{almost surely minimal} as well;
i.e., in almost every realization, it is a minimal cut set ({Lemma~\ref{lem:minimalCut}}).

\begin{lemma}
\label{lem:maxFlowMinCut}
If $\bs f$ is a legal equivariant flow and $\Pi$ is an equivariant cut-set, 
then $\norm{\bs f}\leq \bs c(\Pi)$. Moreover, if the pair $(\bs f,\Pi)$ is equivariant, then 
\[
\norm{\bs f} \leq \omid{\bs f(\bs o)\identity{\{\bs o\in \Pi \}}} \leq \bs c(\Pi).
\] 
In addition, if $\Pi$ is minimal, then equality holds in the left inequality.
\end{lemma}
\begin{proof}
One can always consider an \textit{independent coupling} of $\bs f$ and $\Pi$
(as in the proof of Theorem~\ref{thm:mdp-simple}). So assume $(\bs f, \Pi)$ is 
equivariant from the beginning.
Note that the whole construction (with conductances, the flow and the cut-set)
is unimodular (Lemma~\ref{lem:equivProcess}).
For every leaf $v\in \bs L$, let $\bs \tau(v)$ be the first ancestor of $v$
such that $(v,F(v))\in \Pi$. Then, send mass $\bs f(v)$ from each leaf $v$ to $\bs \tau(v)$.
By the mass transport principle~(\ref{eq:unimodularMarked}), one gets 
\begin{eqnarray*}
\omid{\bs f(\bs o) \identity{\{\bs o\in \bs L \}}} &=&
\omid{ \identity{\{\bs o\in \Pi \}} \sum_{v\in \bs \tau^{-1}(\bs o)}\bs f(v) }\\
&\leq & \omid{ \identity{\{\bs o\in \Pi \}} \sum_{v\in D(\bs o)\cap \bs L}\bs f(v) }
= \omid{\bs f(\bs o)\identity{\{\bs o\in \Pi \}}},
\end{eqnarray*}
where the last equality holds because $\bs f$ is a flow. Moreover, if $\Pi$ is minimal,
then the above inequality becomes an equality and the claim follows.
\end{proof}

The main result is the following converse to the above lemma.
\begin{theorem}[Max-Flow Min-Cut for Unimodular One-Ended Trees]
	\label{thm:maxFlowMinCut}
	For every unimodular {marked} one-ended tree $[\bs T, \bs o;\bs c]$ equipped 
        with conductances $\bs c$ as above, if $\bs c$ is bounded on the set of leaves, then 
	\[
	\max_f \norm{\bs f} = \inf_{\Pi} \bs c(\Pi),
	\]
	where the maximum is over all legal equivariant flows $\bs f$ and the
        infimum is over all equivariant cut-sets $\Pi$.
\end{theorem}

\begin{remark}
	\label{rem:maxflow-infinite}
	The claim of Theorem~\ref{thm:maxFlowMinCut} is still valid if the probability
	measure (the distribution of $[\bs T, \bs o; \bs c]$) is replaced by any (possibly infinite)
	measure $\mathcal P$ on $\mathcal D'_*$ supported on one-ended trees,
	such that $\mathcal P(\bs o\in \bs L)<\infty$ and the mass transport principle~\xeq{\ref{eq:unimodularMarked}} holds.
	The same proof works for this case as well. This will be used in Subsection~\ref{subsec:frostman-euclidean}.
\end{remark}

\begin{proof}[Proof of Theorem~\ref{thm:maxFlowMinCut}]
	For $n\geq 1$, let $\bs T_n$ be the sub-forest of 
	$\bs T$ {obtained by keeping only vertices of height at most $n$ in $\bs T$}.
	Each connected component of $\bs T_n$ is a finite tree which contains some leaves of $\bs T$.
	For each such component, namely $T'$, do the following: if $T'$ has more than one vertex,
	consider the maximum flow on $T'$ between the
	leaves and the \textit{top} vertex (i.e., the vertex with maximum height in $T'$).
	If there is more than one maximum flow, choose one of them randomly and uniformly. Also, choose a
	minimum cut-set in $T'$ randomly and uniformly. Similarly, if $T'$ has a single vertex $v$,
	do the same for the subgraph with vertex set $\{v,F(v)\}$ and the single edge adjacent to $v$.
	By doing this for all components of $\bs T_n$, a (random) function $\bs f_n$ on the edges and
	a cut-set $\Pi'_n$ are obtained (by letting $\bs f_n$ be zero on the other edges).
	$\Pi'_n$ is always a cut-set, but $\bs f_n$ is not a flow. However, $\bs f_n$ 
	satisfies (\ref{eq:flowcons}) for vertices of $\bs T_n\setminus\bs L$,
        except the top vertices of the connected components of $\bs T_n$.
	Also, it can be seen that $\bs f_n$ and $\Pi'_n$ are equivariant. 
	
	For each component $T'$ of $\bs T_n$, the set of leaves of $T'$, excluding the top vertex,
	is $\bs L\cap T'$. So the max-flow min-cut theorem of Ford-Fulkerson \cite{FoFu62} 
	(see e.g., Theorem~3.1.5 of~\cite{bookBiPe17}) gives
	that, for each component $T'$ of $\bs T_n$, one has 
	\[
	\sum_{v\in \bs L\cap T'} \bs f_n(v) = \sum_{e\in\Pi'_n\cap T'} \bs c(e).
	\] 
	If $u$ is the top vertex of $T'$, let $h(u)$ be the common value in the above equation. 
	By using the mass transport principle~\xeq{\ref{eq:unimodularMarked}}
	for each of the two representations of $\omid{h(\bs o)}$, one can obtain
	\[
	\omid{\bs f_n(\bs o) \identity{\{\bs o\in \bs L \}}}= {\omid{h(\bs o)}=} \omid{\bs c(\bs o)\identity{\{\bs o\in \Pi'_n \}}} {=\bs c(\Pi'_n)}.
	\]
	Since $0\leq \bs f_n(\cdot)\leq \bs c_n(\cdot)$, one can see that the distributions of $\bs f_n$ 
	are tight (the proof {is similar to Lemma~\ref{lem:equivariant-tight} and is}
	left to the reader).
	Therefore, there is a sequence $n_1,n_2,\ldots$ and an equivariant process $\bs f'$
	such that $\bs f_{n_i}\rightarrow \bs f'$ (weakly).
	It is not hard to deduce that $\bs f'$ is a legal equivariant flow.
	Also, since {$\bs f'(\bs o)$} and $\identity{\{\bs o\in \bs L \}}$ are
	continuous functions of {$[\bs T, \bs o; \bs f']$} and their product is bounded (by the assumption on $\bs c$),
	one gets that 
	\[
	\norm{\bs f'} {=\omid{\bs f'(\bs o)\identity{\{\bs o\in \bs L \}}} = \lim_i \omid{\bs f_{n_i}(\bs o)\identity{\{\bs o\in \bs L \}}} } = \lim_i \bs c(\Pi'_{n_i}).
	\]
	Therefore,
	$
	\max_f \norm{\bs f} \geq \inf_{\Pi} \bs c(\Pi).
	$
	Note that the maximum of $\norm{\bs f}$ is attained by the same tightness argument as above.
	So Lemma~\ref{lem:maxFlowMinCut} implies that equality holds and the claim is proved.
\end{proof}

\subsection{A Unimodular Frostman Lemma for Point Processes}
\label{subsec:frostman-euclidean}

In the Euclidean case, another form of the unimodular Frostman lemma is given below. 
Its proof is based on the max-flow min-cut theorem of Subsection~\ref{subsec:maxflow}.
As will be seen, the claim implies that in this case, Conjecture~\ref{conj:frostman} holds up to
a constant factor (Corollary~\ref{cor:frostman-Euclidean}). However, the weight function obtained
in the theorem needs extra randomness.

\begin{theorem}
	\label{thm:frostman-euclidean}
	Let $\Phi$ be a point-stationary point process in $\mathbb R^k$ endowed with the $l_{\infty}$ metric,
	and let $\alpha\geq 0$. Then, there exists an equivariant weight function $\bs w$
        on $\Phi$ such that, almost surely,
	\begin{equation}
	\label{eq:thm:frostman-euclidean:1}
	\forall v\in \Phi,\ \forall r\geq 1: \bs w(N_r(v))\leq r^{\alpha}
	\end{equation}	
	and
	\begin{equation}
	\label{eq:thm:frostman-euclidean:2}
	{\omid{\bs w(0)} \geq 3^{-k} \contentH{\alpha}{1}(\Phi).}
	\end{equation}
	In particular, if $\contentH{\alpha}{1}(\Phi)>0$, then $\bs w(0)$ is not identical to zero.
\end{theorem}

\formir{A similar result holds for the Euclidean metric or other equivalent metrics by just changing the constant $3^{-k}$ in~\eqref{eq:thm:frostman-euclidean:2}.}

{In the following proof, $\Phi$ is regarded as a counting measures; i.e., for all $A\subseteq \mathbb R^d$, $\Phi(A):=\card{(\Phi\cap A)}$.}

\begin{proof}
	Let $b>1$ be an arbitrary integer (e.g., $b=2$). 
	For every integer {$n\geq 0$,} let $\bs Q_n$ be the stationary partition of $\mathbb R^k$
	by translations of the cube $[0,b^n)^k$ as in Subsection~\ref{subsec:euclidean}.
	Consider the nested coupling of these partitions for $n\geq 0$
	(i.e., every cube of $\bs Q_n$ is contained in some cube of $\bs Q_{n+1}$ for every $n\geq 0$) independent of $\Phi$.
	Let $\bs T_0$ be the tree whose vertices are the cubes in $\cup_n \bs Q_n$ and the edges are between
	all pairs of nested cubes in $\bs Q_n$ and $\bs Q_{n+1}$ for all $n$.
	Let $\bs T\subseteq \bs T_0$ be the subtree consisting of the cubes $\bs q_n(v)$ 
	for all $v\in\Phi$ and $n\geq 0$. The set $\bs L$ of the leaves of $\bs T$ consists of
	the cubes $\bs q_{0}(v)$ for all $v\in \Phi$. Let $\bs \sigma:=\bs q_{0}(0)\in\bs L$. 
	Note that in the correspondence $v\mapsto \bs q_0(v)$, each cube $\sigma\in\bs L$ corresponds to $\Phi(\sigma)\geq 1$ points of $\Phi$. Therefore, by verifying the mass transport principle, it can be seen that the distribution of $[\bs L,\bs \sigma]$, biased by $1/\Phi(\bs \sigma)$, is unimodular; i.e., 
	\[
	\omid{\frac 1{\Phi(\bs \sigma)}\sum_{\sigma'\in\bs L} g(\bs L,\bs \sigma, \sigma')} = 
	\omid{\frac 1{\Phi(\bs \sigma)}\sum_{\sigma'\in\bs L} g(\bs L,\sigma', \bs \sigma)},
	\]
	for every measurable $g\geq 0$.
	In addition, $g$ can be allowed to depend on $\bs T$ in this equation (but the sum is still on $\sigma'\in\bs L$). Therefore, one can assume the metric on $\bs L$ is the graph-distance metric induced from $\bs T$
	(see Theorem~\ref{thm:metricChange}). Moreover, Theorem~5 of~\cite{shift-coupling} implies that by a further biasing and choosing a new root for $\bs T$, one can make $\bs T$ unimodular. More precisely, the following (possibly infinite) 
	measure on $\mathcal D_*$ is unimodular:
	\begin{equation}
	\label{eq:thm:frostman-euclidean:P}
	\mathcal P[ A]:=  \omid{\sum_{{n\geq 0}} \frac 1{e_n} \identity{A}[\bs T, {\bs q_n(0)}]},
	\end{equation}
	where $e_n:=\Phi(\bs q_n(0))$. 
	Let $\mathcal E$ denote the integral operator w.r.t. the measure $\mathcal P$.
	For any equivariant flow $\bs f$ on $\bs T$, the norm of $\bs f$ w.r.t. the measure $\mathcal P$
	(see Remark~\ref{rem:maxflow-infinite}) satisfies
	\begin{eqnarray*}
		\norm{\bs f} = \mathcal E[\bs f \cdot \identity{\bs L}]
		= \omid{\sum_{{n\geq 0}} \frac 1{e_n} \bs f(\bs q_n(0)) \identity{\{\bs q_n(0)\in \bs L \}} }
		= \omid{\frac 1{\Phi(\bs \sigma)}\bs f(\bs \sigma)},
	\end{eqnarray*} 
	where the second equality is by~\eqref{eq:thm:frostman-euclidean:P}.
	Consider the conductance function $\bs c(\tau):=b^{n\alpha}$ for all cubes $\tau$
	of edge length $b^n$ in $\bs T$ and all $n$. Therefore, Theorem~\ref{thm:maxFlowMinCut}
	and Remark~\ref{rem:maxflow-infinite} imply that 
	the maximum of $\omid{\bs f(\bs{\sigma})}$ over all equivariant legal flows $\bs f$
        on $[\bs T, \bs{\sigma}]$ is attained (note that $[\bs T, \bs \sigma]$ is not unimodular,
        but the theorem can be used for $\mathcal P$).
	Denote by $\bs f_0$ the maximum flow.
	Let $\bs w$ be the weight function on $\Phi$ defined by
        $\bs w(v)=\delta \bs f_0(\bs q_{0}(v)){/\Phi(\bs q_0(v))}$,
	for all $v\in \Phi$, where $\delta:=(b+1)^{-k}$.
	The claim is that $\bs w$ satisfies the requirements~\eqref{eq:thm:frostman-euclidean:1}
        and~\eqref{eq:thm:frostman-euclidean:2}.
	Since $\bs f_0$ is a legal flow, it follows that for every cube $\sigma\in\bs T$, one has
	\[
	\bs w(\sigma) = \delta \bs f_0(\sigma) \leq \delta \bs c(\sigma) = {\delta b^{n\alpha}}.
	\]
	
	Each cube $\sigma$ of edge length $r\in [b^n,b^{n+1})$ in $\mathbb R^k$ can be covered with at most
	$(b+1)^k$ cubes of edge length $b^n$ in $\bs T_0$. If $n\geq 0$, the latter are either
        in $\bs T$ or do not intersect $\Phi$.
	So the above inequality implies	that $\bs w(\sigma)\leq r^{\alpha}$. 
        So \eqref{eq:thm:frostman-euclidean:1} is proved for $\bs w$.
	
	To prove \eqref{eq:thm:frostman-euclidean:2}, given any equivariant cut-set $\Pi$ of $\bs T$,
	a covering of $\Phi$ can be constructed as follows:
	For each cube $\sigma\in \Pi$ of edge length say \formir{$b^n$}, let $\bs \tau(\sigma)$ be one of the points
	in $\sigma\cap\Phi$ chosen uniformly at random and put a ball of radius $b^n$ 
        centered at $\bs \tau(\sigma)$. 
	Note that this ball contains $\sigma$. Do this independently for all cubes in $\bs T$.
	If a point in $\Phi$ is chosen more than once, consider only the largest radius assigned to it.
        It can be seen that this gives an equivariant covering of $\Phi$, namely $\bs R$. One has
	\begin{eqnarray*}
		\omid{\bs R(0)^{\alpha}} &\leq& \omid{\sum_{n\geq 0} b^{n\alpha} \identity{\{\bs q_n(0) \in\Pi \}} \identity{\{0 = \bs \tau(\bs q_n(0)) \}} }\\
		&=&  \omid{\sum_{n\geq 0} \frac{b^{n\alpha}}{e_n}  \identity{\{\bs q_n(0) \in\Pi \}}}.
	\end{eqnarray*}

	On the other hand, by~\eqref{eq:thm:frostman-euclidean:P}, one can see that 
	\[
	\bs c(\Pi) = \omid{\sum_{n\geq 0} \frac 1{e_n} \bs c(\bs q_n(0)) \identity{\{\bs q_n(0) \in \Pi\}}} = \omid{\sum_{n\geq 0} \frac {b^{n\alpha}}{e_n} \identity{\{\bs q_n(0) \in \Pi\}}}.
	\]
	Therefore, $\omid{\bs R(0)^{\alpha}} \leq \bs c(\Pi)$.
	So $\contentH{\alpha}{1}(\Phi) \leq \bs c(\Pi)$. Since $\Pi$ is an arbitrary equivariant cut-set,
	by the unimodular max-flow min-cut theorem established above (Theorem~\ref{thm:maxFlowMinCut})
	and the maximality of the flow $\bs f_0$, one gets that 
	$
	\contentH{\alpha}{1}(\Phi) \leq \norm{\bs f_0} = \omid{\bs f_0(\bs \sigma)/\Phi(\bs \sigma)} = \delta^{-1} \omid{\bs w(0)}.
	$
	So the claim is proved.
\end{proof}

The following corollary shows that in the setting of Theorem~\ref{thm:frostman-euclidean},
the claim of Conjecture~\ref{conj:frostman} holds up to a constant factor (compare this with Lemma~\ref{lem:frostmanAuxiliary}).

\begin{corollary}
	\label{cor:frostman-Euclidean}
	For all point-stationary point processes $\Phi$ in $\mathbb R^k$ endowed with the $l_{\infty}$ metric and all $\alpha\geq 0$, 
	$
	3^{-k} \contentH{\alpha}{1}(\Phi)\leq \xi^{\alpha}_1(\Phi)\leq \contentH{\alpha}{1}(\Phi).
	$
\end{corollary}
\begin{proof}
	The claim is directly implied by \eqref{eq:frostman-xi<H}, Theorem~\ref{thm:frostman-euclidean} and Remark~\ref{rem:frostman-ineq}.
\end{proof}

\subsection{Applications}
The following subsections give some basic applications of the unimodular Frostman lemma.
This lemma is also the basis of many results of~\cite{III}.

\subsubsection{Cayley Graphs}

\begin{proposition}
\label{prop:Cayley}
For every finitely generated group $H$ with polynomial growth {degree} $\alpha\in [0,\infty]$, one has 
$\xi^{\alpha}_{\infty}(H)<\infty$.
\end{proposition}
Note that if Conjecture~\ref{conj:frostman} holds, then this result implies
$\measH{\alpha}(H)>0$, as conjectured in Subsection~\ref{subsec:cayley}.
\begin{proof}
	By Theorem~\ref{thm:Cayley}, $\measH{\alpha}(H)<\infty$. So the unimodular Frostman lemma (Theorem~\ref{thm:frostmanGeneral}) implies that for every $M\geq 1$, there exists $w:\dstar\to\mathbb R^{\geq 0}$ such that $w(N_M(e))\leq M^{\alpha}$ and $\omid{w(e)}=\xi^{\alpha}_M(H)$, where $e$ is the neutral element of $H$. Since the Cayley graph of $H$ is transitive and $w$ is defined up to rooted isomorphisms, $w(H,\cdot)$ is constant. Hence, $w(H,v)=\xi^{\alpha}_M(H)$ for all $v\in H$. Therefore, $\xi^{\alpha}_M(H)\card{N_M(e)} \leq M^{\alpha}$. Thus, $\xi^{\alpha}_M(H)\leq 1/c$, where $c$ is as in the proof of Theorem~\ref{thm:Cayley}. By letting $M\to \infty$, one gets $\xi^{\alpha}_{\infty}(H)\leq 1/c<\infty$.
\end{proof}

\subsubsection{Dimension of Product Spaces}
\label{subsec:product}

Let $[\bs D_1,\bs o_1]$ and $[\bs D_2,\bs o_2]$ be independent unimodular discrete metric spaces.
By considering any of the usual product metrics; e.g., the sup metric or the $p$ product metric,
the \defstyle{independent product} $[\bs D_1\times \bs D_2, (\bs o_1,\bs o_2)]$ makes sense as
a random \rooted{} discrete space. It is not hard to see that the latter is also unimodular
(see also Proposition~4.11 of~\cite{processes}). 

\begin{proposition}
	\label{prop:product}
	Let $[\bs D_1\times \bs D_2, (\bs o_1,\bs o_2)]$ represent the independent product of
	$[\bs D_1,\bs o_1]$ and $[\bs D_2,\bs o_2]$ defined above.
	Then,
	\begin{equation}
	\label{eq:product}
	\dimH{\bs D_1}+\dimMl{\bs D_2} \leq \dimH{\bs D_1\times \bs D_2} \leq \dimH{\bs D_1}+\dimH{\bs D_2}.
	\end{equation}
\end{proposition}
\begin{proof}
	By Theorem~\ref{thm:metricChange}, one can assume the metric on $\bs D_1\times \bs D_2$ is
	the sup metric without loss of generality. So $N_r(v_1,v_2)=N_r(v_1)\times N_r(v_2)$.
	
	The upper bound is proved first. For $i=1,2$, let $\alpha_i>\dimH{\bs D_i}$ be arbitrary.
	By the unimodular Frostman lemma (Theorem~\ref{thm:frostmanGeneral}), there is a nonnegative measurable
	functions $w_i$ on $\dstar$ such that $\forall v \in \bs D_i: \forall r\geq 1: w_i(N_r(v))\leq r^{\alpha} \text{, a.s.}$
	In addition, $w_i(\bs o_i)\neq 0$ with positive probability. Consider the equivariant weight function
	$\bs w$ on $\bs D_1\times \bs D_2$ defined by
	$
	\bs w(v_1,v_2):=w_1[\bs D_1,v_1]\times w_2[\bs D_2,v_2].
	$
	It is left to the reader to show that $\bs w$ is  an equivariant weight function.
	One has $\bs w(N_r(v_1,v_2)) = w_1(N_r(v_1))w_2(N_r(v_2))\leq r^{\alpha_1+\alpha_2}$.
	Also, by the independence assumption, $\bs w(\bs o_1,\bs o_2)\neq 0$ with positive probability.
	Therefore, the mass distribution principle (Theorem~\ref{thm:mdp-simple})
	implies that $\dimH{\bs D_1\times\bs D_2}\leq \alpha_1+\alpha_2$. This proves the upper bound.
	
	For the lower bound in the claim, let $\alpha<\dimH{\bs D_1}$, $\beta<\dimMl{\bs D_2}$
	and $\epsilon>0$ be arbitrary. It is enough to find an equivariant covering $\bs R$ of
	$\bs D_1\times\bs D_2$ such that $\omid{\bs R(\bs o_1,\bs o_2)^{\alpha+\beta}}<\epsilon$. 
	One has $\decayl{\lambda_r(\bs D_2)}>\beta$, where $\lambda_r$ is defined in~(\ref{eq:I_r}).
	So there is $M>0$ such that $\forall r\geq M: \lambda_r(\bs D_2)<r^{-\beta}$. So for every $r\geq M$,
	there is an equivariant $r$-covering of $\bs D_2$ with intensity less than $r^{-\beta}$.
	On the other hand, since $\alpha <\dimH{\bs D_1}$, one has $\contentH{\alpha}{M}(\bs D_1)=0$
	(by Lemma~\ref{lem:Hmeas-elementary}). Therefore there is an equivariant covering $\bs R_1$
	of $\bs D_1$ such that $\omid{\bs R_1(\bs o_1)^{\beta}}<\epsilon$ and $\forall v\in \bs D_1: \bs R_1(v)\in \{0\}\cup [M,\infty)$ a.s. 
	Choose the extra randomness in $\bs R_1$ independently from $[\bs D_2,\bs o_2]$. 
	Given a realization of $[\bs D_1,\bs o_1]$ and $\bs R_1$, do the following: 
	Let $v_1\in \bs D_1$ such that $\bs R_1(v_1)\neq 0$ (and hence, $\bs R_1(v_1)\geq M$).
	One can find an equivariant subset $\bs S_{v_1}$ of $\bs D_2$ that gives a covering of
	$\bs D_2$ by balls of radius $\bs R_1(v_1)$ and has intensity less than $\bs R_1(v_1)^{-\beta}$.
	Do this independently for all $v_1\in \bs D_1$. Now, for all $(v_1,v_2)\in \bs D_1\times \bs D_2$, define
	\[
	\bs R(v_1,v_2):=\begin{cases}
	\bs R_1(v_1)& \text{if } \bs R_1(v_1)\neq 0 \text{ and } v_2\in \bs S_{v_1},\\
	0& \text{otherwise}.
	\end{cases}
	\]
	Now, $\bs R$ is a covering of $\bs D_1\times \bs D_2$ and it can be seen that 
	it is an equivariant covering. Also, given $[\bs D_1,\bs o_1]$ and $\bs R_1$,
	the probability that $\bs o_2\in \bs S_{\bs o_1}$ is less than $\bs R_1(\bs o_1)^{-\beta}$. So one gets 
	\begin{eqnarray*}
		\omid{\bs R(\bs o_1,\bs o_2)^{\alpha+\beta}} 
		&=& \mathbb E\Bigg[\mathbb E\Big[{\bs R(\bs o_1,\bs o_2)^{\alpha+\beta}}|{[\bs D_1,\bs o_1],\bs R_1}\Big]\Bigg]\\ 
		&< & \omid{\bs R_1(\bs o_1)^{\alpha+\beta} \bs R_1(\bs o_1)^{-\beta}}
		= \omid{\bs R_1(\bs o_1)^{\alpha}}
		<\epsilon.
	\end{eqnarray*}
	So the claim is proved.
\end{proof}

The following examples provide instances where the inequalities in~\eqref{eq:product} are strict. 
\begin{example}
	Assume $[\bs D_1,\bs o_1]$ and $[\bs D_2,\bs o_2]$ are unimodular discrete spaces such that
	$\dimMl{\bs G_1}<\dimH{\bs G_1}$ and $\dimMl{\bs G_2}=\dimH{\bs G_2}$. By Proposition~\ref{prop:product}, one gets
	\[
	\dimH{\bs G_1\times \bs G_2} \geq \dimH{\bs G_1}+\dimMl{\bs G_2} > \dimH{\bs G_2}+ \dimMl{\bs G_1}.
	\]
	So by swapping the roles of the two spaces, an example of strict inequality in the left hand side of~\eqref{eq:product} is obtained.
\end{example}

\begin{example}
	Let $J$ be a subset of $\mathbb Z^{\geq 0}$ such that $\densityU{}(J)=1$ and $\densityL{}(J)=0$ simultaneously
	(see Subsection~\ref{subsec:digits} for the definitions). Let $\Psi_1$ and $\Psi_2$ be defined
	as in Subsection~\ref{subsec:digits} corresponding to $J$ and $\mathbb Z^{\geq 0} \setminus J$ respectively.
	Proposition~\ref{prop:digits} implies that $\dimH{\Psi_1}=\dimH{\Psi_2}=1$. On the other hand, \eqref{eq:prop:digits} implies that
	\[
	\card{N_{2^n}(\bs o_1\times \bs o_2)} \leq 2^{J_n+1} \times 2^{(n+1-J_n)+1} = 2^{n+3}.
	\]
	This implies that $\growthu{N_r(\bs o)}\leq 1$. So the unimodular Billingsley lemma
	(Theorem~\ref{thm:billingsley}) implies that $\dimH{\Psi_1\times \Psi_2}\leq 1$
	(in fact, equality holds by Proposition~\ref{prop:embedded} below).
	So the rightmost inequality in~\eqref{eq:product} is strict here.
\end{example}

\subsubsection{Dimension of Embedded Spaces}
\label{subsec:embedded}

It is natural to think of $\mathbb Z$ as a subset of $\mathbb Z^2$. However, $[\mathbb Z,0]$ is
not an equivariant subspace of $[\mathbb Z^2,0]$.
By the following definition, $[\mathbb Z,0]$ is called \textit{embeddable in $[\mathbb Z^2,0]$}.
The dimension of embedded subspaces is studied in this subsection.
\formir{The analysis requires the unimodular Frostman lemma.}

\begin{definition}
	\label{def:embedded}
	Let $[\bs D_0, \bs o_0]$ and $[\bs D, \bs o]$ be random \rooted{} discrete spaces. 
	An \defstyle{embedding} of $[\bs D_0,\bs o_0]$ in $[\bs D, \bs o]$ is 
	a (not necessarily unimodular) random \rooted{} marked discrete space $[\bs D', \bs o'; \bs m]$
	with mark space $\{0,1\}$ such that
	\begin{enumerate}[(i)]
		\item $[\bs D', \bs o']$ has the same distribution as $[\bs D, \bs o]$.
		
		\item $\bs m(\bs o')=1$ a.s. and by letting $\bs S:=\{v\in \bs D': \bs m(v)=1 \}$ equipped with the
		induced metric from $\bs D'$, $[\bs S, \bs o']$ has the same distribution
		as $[\bs D_0,\bs o_0]$.
	\end{enumerate}
	If in addition, $[\bs D_0, \bs o_0]$ is unimodular, then
	$[\bs D',\bs o'; \bs m]$ is called an \defstyle{equivariant embedding} if
	\begin{enumerate}[(i)]
		\setcounter{enumi}{2}
		\item \label{def:embedding:4} The mass transport principle holds on $\bs S$; i.e.,
		\xeq{\ref{eq:unimodularMarked}} holds for functions $g(u,v):=g(\bs D',u,v; \bs m)$
		such that $g(u,v)$ is zero when $\bs m(u)=0$ or $\bs m(v)=0$.
	\end{enumerate}
	If an embedding (resp. an equivariant embedding) exists, $[\bs D_0, \bs o_0]$ 
	is called \defstyle{embeddable} (resp. \defstyle{equivariantly embeddable}) in $[\bs D, \bs o]$.
\end{definition}

It should be noted that $[\bs D', \bs o'; \bs m]$ is not an equivariant process on $\bs D$ {except in the trivial case where $\bs m(\cdot)=1$ a.s.}

\begin{example}
	\formir{Here are instances of Definition~\ref{def:embedded}.}
	\begin{enumerate}[(i)]

		\item Let $[\bs D_0,\bs o_0]:=[\mathbb Z,0]$ and $[\bs D, \bs o]:=[\mathbb Z^2,0]$
                equipped with the sup metric.
		Consider $m:\mathbb Z^2\rightarrow\{0,1\}$ which is equal to one
                on the boundary of the positive cone. 
		Then, $[\mathbb Z^2, 0; m]$ is an embedding of $[\mathbb Z,0]$ in $[\mathbb Z^2,0]$,
		but is not an equivariant embedding since it does not satisfy~\eqref{def:embedding:4}.

		\item A point-stationary point process in $\mathbb Z^k$ (\rooted{} at 0)
                is equivariantly embeddable in $[\mathbb Z^k,0]$.
		
		\item Let $H$ be a finitely generated group equipped with the graph-distance metric of an arbitrary
		Cayley graph over $H$. Then, any subgroup of $H$ {(equipped with the induced metric)}
		is equivariantly embeddable in $H$.
	\end{enumerate}
\end{example}

\begin{proposition}
\label{prop:embedded}
If $[\bs D_0, \bs o_0]$ and $[\bs D, \bs o]$ are unimodular discrete spaces and the
former is equivariantly embeddable in the latter, then
\begin{eqnarray}
\label{eq:thm:embedded:1}
\dimH{\bs D}&\geq& \dimH{\bs D_0},
\end{eqnarray}
\formir{and for all $\alpha\geq 0$ and $M\geq 1$, with $\xi^{\alpha}_M$ defined in Definition~\ref{def:xi^alpha},
\begin{eqnarray}
\label{eq:thm:embedded:2}
\xi^{\alpha}_M(\bs D, 1) &\leq & \contentH{\alpha}{M}(\bs D_0).
\end{eqnarray}}
\end{proposition}

\begin{proof}
	First, assume~\eqref{eq:thm:embedded:2} holds.
	For $\alpha>\dimH{\bs D}$, one has $\contentH{\alpha}{M}(\bs D)>0$ (Lemma~\ref{lem:Hmeas-elementary}).
	Therefore, Lemma~\ref{lem:frostmanAuxiliary} implies that $\xi^{\alpha}_M(\bs D, 1)>0$.
	Hence, \eqref{eq:thm:embedded:2} implies that $\contentH{\alpha}{M}(\bs D_0)>0$,
	which implies that $\dimH{\bs D_0}\leq\alpha$. So it is enough to prove~\eqref{eq:thm:embedded:2}.
	
	By the unimodular Frostman lemma (Theorem~\ref{thm:frostmanGeneral}), there is a bounded function
	$w:\mathcal D_*\rightarrow\mathbb R^{\geq 0}$ such that $\omid{w(\bs o)} = \xi^{\alpha}_M(\bs D, 1)$,
	and almost surely, $w(N_r(\bs o))\leq r^{\alpha}$ for all $r\geq M$. 
	Assume $[\bs D', \bs o'; \bs m]$ is an equivariant embedding as in Definition~\ref{def:embedded}.
	For $x\in \bs D'$, let $w'(x):=w'_{\bs D'}(x):=w[\bs D', x]$. 
	Consider the random \rooted{} marked discrete space $[\bs S, \bs o'; w']$
        obtained by restricting $w'$ to $\bs S$.
	By the definition of equivariant embeddings and by directly verifying the mass
        transport principle, the reader can obtain that $[\bs S, \bs o'; w']$ is unimodular.
	Since $[\bs S, \bs o']$ has the same distribution as
	$[\bs D_0, \bs o_0]$, there exists  
	an equivariant process
	$\bs w_0$ on $\bs D_0$ such that $[\bs S ,\bs o'; w']$ has the same distribution as
        $[\bs D_0, \bs o_0; \bs w_0]$ \formir{(see the converse of Lemma~\ref{lem:equivProcess} in Subsection~\ref{subsec:process})}. 
	According to the above discussion, one has 
	$$\forall r\geq M: w'(N_r(\bs S, \bs o'))\leq w'(N_r(\bs D',\bs o')) \leq r^{\alpha}, \quad a.s.$$
	This implies that $\bs w_0(N_r(\bs o_0))\leq r^{\alpha}$ a.s. Therefore, the mass distribution
        principle (Theorem~\ref{thm:mdp-simple})
	implies that $\omid{\bs w_0(\bs o_0)}\leq \contentH{\alpha}{M}(\bs D_0)$. One the other hand, 
        $$\omid{\bs w_0(\bs o_0)}= \omid{w'(\bs o')}=\omid{w(\bs o)} = \xi^{\alpha}_M(\bs D,1),$$ 
	where the last equality is by the assumption on $w$.
	This implies that $\contentH{\alpha}{M}(\bs D_0)\geq \xi^{\alpha}_M(\bs D,1)$ and the claim is proved.
\end{proof}

It is natural to expect that an embedded space has a smaller \formir{Hausdorff size}.
\formir{This is stated in {Conjecture~\ref{conj:embedding}}.}

\begin{remark}
	\label{rem:embedding-counter}
	Another possible way to prove Proposition~\ref{prop:embedded} and Conjecture~\ref{conj:embedding}
	is to consider an arbitrary equivariant covering of $\bs D_0$ and try to extend it to an equivariant 
	covering of $\bs D$ by adding some balls (without adding a ball centered at the root).
	More generally, given an equivariant processes $\bs Z_0$ on $\bs D_0$, one might try to extend
	it to an equivariant process on $\bs D$ without changing the mark of the root.
	But {at least} the latter is not always possible. A counter example is when $[\bs D_0,\bs o_0]$ is $K_2$
	(the complete graph with two vertices),  $[\bs D,\bs o]$ is $K_3$, $\bs Z_0(\bs o_0)=\pm 1$
	chosen uniformly at random, and the mark of the other vertex of $\bs D_0$ is $-\bs Z_0(\bs o_0)$.
\end{remark}

\subsection{Notes and Bibliographical Comments}

The unimodular Frostman lemma (Theorem~\ref{thm:frostmanGeneral}) is analogous to  Frostman's 
lemma in the continuum setting (see e.g., Thm~8.17 of~\cite{bookMa95}). The proof of
Theorem~\ref{thm:frostmanGeneral} is also inspired by that of~\cite{bookMa95}, but there
are substantial differences. For instance, the proof of Lemma~\ref{lem:frostmanAuxiliary} and
also the use of the duality of $L_1$ and $L_{\infty}$ in the proof of Theorem~\ref{thm:frostmanGeneral} are new.
The Euclidean version of the unimodular Frostman lemma (Theorem~\ref{thm:frostman-euclidean})
and its proof are inspired by the continuum analogue (see e.g.,~\cite{bookBiPe17}). 

As already explained, the unimodular max-flow min-cut theorem (Theorem~\ref{thm:maxFlowMinCut})
is inspired by the max-flow min-cut theorem for finite trees. 
Also, the results and examples of Subsection \ref{subsec:product} on product spaces 
are inspired by analogous in the continuum setting; e.g., 
Theorem~3.2.1 of~\cite{bookBiPe17}.

\section{\formir{Miscellaneous Topics}}
\label{s:MiTo}

\subsection{\formir{Connections with Other Notions of Dimension}}
\label{subsec:connections}

Several notions of dimension are already defined in the literature for discrete spaces
\formir{in special cases. A few of them are listed in this subsection together with
their connections to unimodular dimensions.

For subsets of $\mathbb Z^d$, the notions of \textit{upper and lower mass dimension}
are defined in~\cite{BaTa89}, which are just the volume growth rates defined in
Section~\ref{sec:volumeGrowth}. 
The paper~\cite{KhoXi17} extends the upper mass dimension to general subsets
$A\subseteq \mathbb R^d$ and calls it the \textit{macroscopic Minkowski dimension}
of $A$ (one may define lower macroscopic Minkowski dimension similarly). 
This extension is obtained by \textit{pixelizing} $A$ to get a subset of $\mathbb Z^d$. 
The unimodular Billingsley lemma states that for unimodular (i.e., point-stationary)
and ergodic subsets of $\mathbb Z^d$, the unimodular Hausdorff dimension is between 
the upper and lower mass dimension. A similar result holds in the non-integer case as
well:
\begin{corollary}
For ergodic point-stationary point processes in $\mathbb R^d$, the unimodular
Hausdorff dimension is between the upper and lower macroscopic Minkowski dimensions a.s.
\end{corollary}
This is a direct corollary of Billingsley's lemma applied to the pixelization by a randomly-shifted lattice.
It can also be proved by using weights in
Billingsley's lemma similarly to the proof of Proposition~\ref{prop:upperbound-Rd}.
}

Another notion is \formir{that of \textit{discrete (Hausdorff) dimension}}
\cite{BaTa89}, which uses the idea behind the definition of the \formir{classical}
Hausdorff dimension by considering coverings of $\Phi\subseteq\formir{\mathbb Z^d}$
by large balls and considering the cost $(\frac r {r+|x|})^{\alpha}$
for each ball in the covering, where $r$ and $x$ are the radius and the center of
the ball and $\alpha$ is a constant (in fact, this is a modified version of
the definition of~\cite{BaTa89} mentioned in~\cite{bookBiPe17}). 
\formir{In the future work~\cite{III}, it is shown that} the discrete
dimension is an upper bound for the unimodular
Hausdorff dimension, when both notions are defined (i.e., for point-stationary point processes).

The unimodular Hausdorff dimension can be connected to the classical Hausdorff
dimension via scaling limits. Such limits are random continuum metric spaces
and can be defined by weak convergence w.r.t. the Gromov-Hausdorff-Prokhorov metric~\cite{Kh19ghp}.
It is shown in the preprint~\cite{III} that if the unimodular discrete space
admits a scaling limit, then the ordinary Hausdorff dimension of the limit is
an upper bound for the unimodular Hausdorff dimension. 

\formir{The above inequalities are expected to be equalities in most examples.
The preprint~\cite{III} provides more discussion on the matter.} 
Note that these comparison results imply relations
between the \formir{volume} growth rate, scaling limits and discrete dimension,
which are of independent interest and which are new to the best of the authors' knowledge.

\formir{A problem of potential interest is the connection of unimodular dimensions 
to other notions of dimension. This includes
Gromov's notion of \textit{asymptotic dimension}~\cite{bookGr91},
the \textit{spectral dimension} 
of a graph (defined in terms of the return probabilities of the simple random walk), the
\textit{typical displacement exponent} of a graph (see~\cite{CuHuNa17} for both notions),
the \textit{isoperimetric dimension} of a graph~\cite{ChYa95}, the \textit{resistance growth exponent}
of a graph, the \textit{stochastic dimension}
of a partition of $\mathbb Z^d$~\cite{BeKePeSc11}, etc. 
In statistical physics, one also assigns dimension and various exponents to 
finite models. Famous examples are self-avoiding walks and the boundaries of large percolation clusters. }

\subsection{Gauge Functions and the Unimodular Dimension Function}
\label{subsec:gauge}
There 
exist unimodular discrete spaces $\bs D$ in which the $\dimH{\bs D}$-dimensional
\formir{Hausdorff size} is either zero or infinity
\formir{(e.g., Examples~\ref{ex:infiniteMeasure} and~\ref{ex:zeroMeasure})}.
For such spaces, it is convenient to generalize
the unimodular \formir{Hausdorff size} as follows. Consider an increasing function
$\varphi:\{0\}\cup[1,\infty)\to[0,\infty)$; e.g., $\varphi(r)=r^{\alpha}$, called
a \textit{gauge function}. Define $\contentH{\varphi}{M}(\bs D)$ by
$\inf_{\bs R}\{\omid{\varphi(\bs R(\bs o))} \}$ similarly to~\eqref{eq:hcontent}.
Then, define $\measH{\varphi}(\bs D)$ similarly to~\eqref{eq:hmeas}.
If $0<\measH{\varphi}(\bs D)<\infty$, then $\varphi$ is called a \defstyle{unimodular dimension function}
for $\bs D$.




In addition, given a family of gauge functions  $(\varphi_{\alpha})_{\alpha\geq 0}$ that
is increasing in $\alpha$ and such that 
$\forall \alpha>\beta: \lim_{r\rightarrow\infty}\varphi_{\alpha}(r)/\varphi_{\beta}(r)=\infty$, 
one can redefine the unimodular Hausdorff dimension by $\sup\{\alpha:\measH{\varphi_{\alpha}}(\bs D)=0\}$ 
(see e.g., the next paragraph). One can redefine the unimodular Minkowski dimension similarly.
\formir{The authors have verified that the results of the paper} can be extended to this setting
except that Theorem~\ref{thm:subsetDimension} and the results of Subsection~\ref{subsec:otherSets}
require the \textit{doubling condition} $\sup_{r\geq 1} \varphi(2r)/\varphi(r)<\infty$. 
The general result of Subsection~\ref{subsec:one-ended} can also be extended under the doubling condition.
Also, the upper bounds in the unimodular mass distribution principle, the unimodular Billingsley
lemma and the unimodular Frostman lemma hold in this more general setting
(some other results require the doubling condition).
However, for the ease of reading, the results are presented in the original setting of this \formir{paper}.

As an example of the above framework, one can define the \defstyle{exponential dimension}
by considering $\varphi_{\alpha}(r):=e^{\alpha r}$.
It might be useful for studying unimodular spaces with super-polynomial \formir{volume} growth,
which are more interesting in group theory (see Subsection~\ref{subsec:cayley}). 
\formir{Other gauge functions may also be useful for groups of \textit{intermediate growth}.}
Note that exponential gauge functions do not satisfy the doubling condition,
and hence, the reader should be careful about using the results of this work for such gauge functions.

\formir{
\subsection{Negative Dimensions}
\label{subsec:negative}

If a compact metric space $X$ is the union of $k$ disjoint copies of $\frac 1 r X$,
then the \textit{similarity dimension} of $X$ is $\log k/ \log r$ (\formir{see e.g., \cite{bookBiPe17}}).
This definition can also be used for some infinite discrete sets as well.
For instance, $\mathbb Z^d$ is a union of $2^d$ copies of $2\mathbb Z^d$.
So, it can be said that the \textit{similarity dimension} of $\mathbb Z^d$ is negative.
\formir{The (deterministic) discrete Cantor set (see e.g., \cite{bookBiPe17}) is also $(-\log 2/\log 3)$-dimensional}.
There are several further arguments, listed below, suggesting that one should
actually assign negative dimensions to unimodular discrete spaces.

First, this would be natural in terms of definition.
The unimodular Minkowski dimensions should be redefined
by $\dimMu{\bs D}=\growthu{\lambda_r}$ and $\dimMl{\bs D}=\growthl{\lambda_r}$.
Using growth instead of decay would then unify the definition of the ordinary Minkowski
dimension of compact sets and the unimodular Minkowski dimension. The former is microscopic
(i.e., when $r$ tends to 0), whereas the latter is macroscopic ($r\to \infty$).
{One may also replace} the unimodular Hausdorff
by the negative of the definitions given so far.

Secondly, this unification of the definitions would also take
care of the puzzling direction of certain inequalities discussed in the paper:
when adopting these negative unimodular dimensions,
(i) the classical and unimodular Minkowski and Hausdorff dimensions would be ordered in
the same way, i.e., $\dimH{\bs D}\leq \dimMl{\bs D}\leq \dimMu{\bs D}$;
(ii) an equivariant subset of a unimodular set would have a unimodular
Minkowski dimension smaller than or equal to that of the set, and possibly strictly smaller
(see Subsections~\ref{subsec:equiv-subspace} and~\ref{subsec:subspace-minkowski}),
a property that is expected to hold for any notion of \textit{dimension};
(iii) the mass distribution principle and the Billingsley lemma would provide lower bounds
on $\dimH{\bs D}$, while upper bounds would be obtained by constructing explicit coverings;
(iv) the dimension of non-ergodic examples (e.g., Example~\ref{ex:nonergodic}) would also be
the supremum dimension of the components, as one might expect (see also Remark~\ref{rem:nonergodic-justification}).

It should however be noted that by assigning such negative dimension, the dimension of
a non-equivariant subset (see Subsection~\ref{subsec:embedded}) would be larger than or equal to
that of the whole space (just as the similarity dimension of $\mathbb Z$ is larger than that of
$\mathbb Z^2$). One should not expect that non-equivariant subsets behave like equivariant subsets,
since they do not satisfy the main assumption of \textit{statistical homogeneity},
which is the basis of all of the definitions in this work.}

\formir{
\subsection{Problems and Conjectures}
\label{ss:prco}

This subsection gathers some problems and conjectures pertaining to the theory of
unimodular dimensions and to specific examples. Those are already stated
in the paper and are also briefly listed.

\subsubsection{Further conjectures and problems}
\label{subsec:furtherproblems}
	
\paragraph{1. Connections to other notions}.
The following conjectures connect unimodular dimensions to the properties of the simple random walk.

\begin{conjecture}
If $[\bs G, \bs o]$ is a unimodular graph such that $\dimH{\bs G}<2$,
then the simple random walk in $\bs G$ is recurrent a.s.
\end{conjecture}

Note that the converse of this conjecture does not hold. For instance, it
is not hard to show that any one-ended tree (e.g., the canopy tree) is recurrent.

In~\cite{BaTa89}, the discrete dimension of $A\subseteq\mathbb Z^d$ is connected
to whether the simple random walk in $\mathbb Z^d$ hits $A$ infinitely often or not.
Analogously, one can generalized the above conjecture as follows:

\begin{conjecture}
If $[\bs D, \bs o]$ is equivariantly embedded in the unimodular graph $[\bs G, \bs o]$
and $\dimH{\bs D}>\dimH{\bs G}-2$, then the simple random walk in $\bs G$ hits $\bs D$ infinitely often a.s.
\end{conjecture}



\paragraph{2. Dimension functions.}
Does there exist a unimodular discrete space without any dimension function?
The answer is not known yet. \cite{ElKe06} gives a positive answer to the analogous question
in the continuum setting, but the proof ideas don't seem to work in the unimodular discrete setting.
Also, by analogy with stable trees~\cite{Du10},
a possible candidate is unimodular \egw{} trees with infinite offspring variance.

\paragraph{3. Simple random walk.} By analogy with the image of
subordinator processes (see e.g., \cite{FrPr71}), one may guess the exact unimodular dimension
function for the image of the random walk under the assumptions of Proposition~\ref{prop:image}.
For example, consider the zero set $\Psi$ of the simple random walk (Proposition~\ref{prop:zeros}).
By analogy with the zero set of Brownian motion~\cite{TaWe66}, it is natural to guess that
$\measH{1/2}(\Psi)=\infty$ and $\sqrt{r\log\log r}$ is a dimension function for $\Psi$.
To prove this, one should strengthen the bound $Cr^{\beta} \log \log r$ in the proof
of Proposition~\ref{prop:image} and also construct a covering of the set which is better
than that of Proposition~\ref{prop:lowerBoundR}. For the former, one may use
Theorem~4 of~\cite{FrPr71} (it seems that the assumption of~\cite{FrPr71} on the 
tail of the jumps is not necessary for having only an upper bound).
For the latter, one might try to get ideas from~\cite{TaWe66}
(it is necessary to use intervals with different lengths).
\\
\formir{Another guess is that the image of the symmetric nearest-neighbor simple random walk
in $\mathbb Z^d$ is 2-dimensional when $d\geq 3$. More generally, if the jumps are in the domain
of attraction of a symmetric $\alpha$-stable process, then the image is $\alpha$-dimensional.
These might be proved similarly to the analogous results in~\cite{BaTa92}.}
\\
For the graph of the simple random walk equipped with the Euclidean metric
(Subsection~\ref{subsec:srw-graph}), the guess is that if the increments are in the
domain of attraction of an $\alpha$-stable distribution, where $0<\alpha\leq 2$, 
then $\dimM{\Psi}=\dimH{\Psi} = \min\{1, \max(0,2\alpha-1)/\alpha \}$ {(see Theorem~3.13 of~\cite{KhoXi17})}.
Also, the guess is that the zero set of the symmetric nearest-neighbor simple random walk
in $\mathbb Z^2$ is $\frac 1 4$-dimensional.


\paragraph{4. Eternal Galton-Watson trees.} 
For unimodular eternal Galton-Watson trees (Subsection~\ref{subsec:egw}),
a conjecture is that if the offspring distribution
is in the domain of attraction of an $\alpha$-stable distribution, where $\alpha\in [1,2]$,
then $\dimM{\bs T} = \dimH{\bs T} = \frac{\alpha}{\alpha-1}$ (see~\cite{HaMi04} or Theorem~5.5 of~\cite{DuLega05}).
The guess is that there is no \textit{regularly varying} dimension function (see~\cite{Du10}), except
in the finite-variance case ($\alpha=2$), where one may guess that the dimension function
is $r^2\log\log r$ (see \cite{DuLe06}).

\paragraph{5. Drainage networks.} One can ask about the dimension of other drainage network models.
In particular, the simple model of Subsection~\ref{subsec:drainage} can be extended to a model in $\mathbb Z^k$
for $k>2$ and the connected component containing the origin is unimodular.

\paragraph{6. Embedded spaces.} 
\begin{conjecture}
\label{conj:embedding}
Under the setting of Proposition~\ref{prop:embedded}, for all $\alpha>0$,
one has $\measH{\alpha}(\bs D)\geq \measH{\alpha}(\bs D_0)$.
\end{conjecture}

Note that in the case $\alpha=0$, the conjecture is implied by Proposition~\ref{prop:finite-HausMeas}.
Also, in the general case, the conjecture is implied by~\eqref{eq:thm:embedded:2}
and Conjecture~\ref{conj:frostman}. 
%
%
Another problem is the validity of Proposition~\ref{prop:embedded} under the weaker assumption
of being non-equivariantly embeddable. As a partial answer, if $\growth{\card{N_r(\bs o)}}$ exists,
then~\eqref{eq:thm:embedded:1} holds. This is proved as follows: 
\begin{eqnarray*}
	\dimH{\bs D_0} \leq  \essinf \growthu{\card{N_r(\bs o_0)}}
	\leq  \essinf \growthu{\card{N_r(\bs o)}}\:&&\\
	= \essinf \growth{\card{N_r(\bs o)}}
	= \dimH{\bs D},&&
\end{eqnarray*}
where the first inequality and the last equality are implied by 
the unimodular Billingsley lemma.

\subsubsection{List of conjectures and problems mentioned in the previous sections}

It is not known whether the lower bound~\eqref{eq:thm:eftMinkowski:h} for the Hausdorff dimension
of unimodular one-ended trees is always an equality or not. 
Problem~\ref{prob:uppergrowth} asks whether the equality $\dimH{\bs D}=\growthu{\bs w(N_r(\bs o))}$
always holds. This is implied by Problem~\ref{prob:growth},
which states that the upper and lower growth rates of $\bs w(N_r(\bs o))$
(used in Billingsley's lemma) do not depend on $\bs w$.

Conjecture~\ref{conj:point-stationary} states that every point-stationary point process has zero
Hausdorff size unless when it is the Palm version of some stationary point process.
Conjecture~\ref{conj:frostman} states that $\contentH{\alpha}{M}(\bs D)=\xi^{\alpha}_M(1)$, where $\xi^{\alpha}_M(1)$ is define in Subsection~\ref{subsec:frostman-general}. This implies the conjecture that Cayley graphs have positive Hausdorff size (see Subsection~\ref{subsec:cayley}).

It would be interesting to find connections between unimodular dimensions and other notions 
of dimension, some of which are discussed in Subsection~\ref{subsec:connections}. Also,
as mentioned in the introduction and Subsection~\ref{subsec:gauge}, the setting of this paper
might be useful in the study of examples pertaining to statistical physics or group theory.
}

\begin{appendices}
	\section{A Metric on the Space of \Rooted{} Discrete Spaces}
	\label{ap:metric}
	In this appendix, a metric is defined on $\dstar$ and it is shown that $\dstar$ is a Borel subset of some Polish space. The same is done for the sets $\dstarr, \dstar'$ and $\dstarr'$.

	\begin{definition}
		\label{def:r-embedding}
		Let $(D,o)$ and $(D',o')$ be \rooted{} discrete spaces. An \defstyle{$r$-embedding} between $(D,o)$ and $(D',o')$, where $r>0$, is an injective function $f:\nei{r}{o} \to D'$ such that $f(o)=o'$ and has \defstyle{distortion} at most $\frac 1 r$; i.e.,
		\begin{equation}
			\label{eq:distortion}
			\forall x,y\in\nei r o:\quad  |d(x,y)-d(f(x),f(y))|\leq \frac 1r.
		\end{equation}
		If such a function exists, then $(D,o)$ is \defstyle{$r$-embeddable} in $(D',o')$. If each one of $(D,o)$ and $(D',o')$ is $r$-embeddable in the other, then they are called \defstyle{$r$-similar}.
	\end{definition}
	
	Note that the image of $f$ is not necessarily contained in $\nei{r}{o'}$. However, by~\eqref{eq:distortion}, it is contained in $\nei{r+1/r}{o'}$. Note also that if $(D,o)$ is $r$-embeddable in $(D',o')$ and $0<s<r$, then the former is also $s$-embeddable in the latter.
	
	\begin{definition}
		\label{def:kappa}
		For two \rooted{} discrete spaces $(D,o)$ and $(D',o')$, define 
		\begin{equation*}
			\kappa((D,o),(D',o')):= 1 \wedge \inf\{\epsilon>0: (D,o) \text{ and } (D^\prime,o^\prime) \text{ are } \frac 1{\epsilon}\text{-similar}  \}.
		\end{equation*}
	\end{definition}

	The definition clearly depends only on the isomorphism classes of $(D,o)$ and $(D',o')$.
	So $\kappa$ is well defined as a function on $\dstar\times \dstar$. 
	{At the end of this subsection, $\kappa$ is extended to a distance function on the set $\widehat{\mathcal D}_*$, which heuristically is the set of pointed discrete spaces in which every point has some \textit{multiplicity}.}

	\begin{theorem}
		\label{thm:polish}
		{The function $\kappa$ is a metric on $\widehat{\mathcal D}_*$. Moreover, $\widehat{\mathcal D}_*$ is a Polish space and contains $\dstar$ as a Borel subset.}
	\end{theorem}

	{This theorem is proved later in this appendix.}	It should be noted that $\dstar$ is not complete itself. 
	
	Similar definitions and arguments can be proposed for $\dstarr, \mathcal D'_*$ and $\mathcal D'_{**}$.
	The latter is briefly explained below {(see~\cite{Kh19generalization} for a more general framework)}. Let $(D,o,p;m)$ and $(D',o',p';m')$ be two doubly-\rooted{} 
	marked discrete spaces. Define \textit{$r$-embeddings} similarly to Definition~\ref{def:r-embedding}
	by adding the constrains $f(p)=p'$ (which requires $r\geq d(o,p)$) and in addition,
	\[
	\forall x,y\in N_r(o): d\bigg(m(x,y),m'(f(x),f(y))\bigg)\leq \frac 1 r.
	\]
	Being \textit{$r$-similar} and $\kappa$ are also defined in the same way.
	The result analogous to Theorem~\ref{thm:polish} is the following theorem. {It can be proved with the same arguments and its proof is skipped.}
	
	\begin{theorem}
		\label{thm:polish2}
		The sets $\dstarr, \mathcal D'_*$ and $\mathcal D'_{**}$, equipped with the distance function $\kappa$,
		are metric spaces and are Borel subsets of some Polish spaces.
	\end{theorem}
	
	\del{It should be mentioned that the set $\mathcal D$ of non-\rooted{} discrete spaces is not a
		Polish space (if the natural projection $\pi:\dstar\rightarrow \mathcal D$ is required to be measurable).}
	
	\begin{remark}
		\label{rem:topology}
		{The Gromov-Hausdorff-Prokhorov topology (see e.g., \cite{bookVi10}) is defined on the set of \textit{boundedly-compact} pointed metric spaces equipped with \textit{boundedly-finite} Borel measures. This induces a topology on $\dstar$ (here, one should equip every discrete metric space $D$ with the counting measure on $D$). It can be seen that the metric $\kappa$ is a metrization of this topology.}
		It can {also} be seen that  $\kappa$ extends the metric between rooted graphs used
		in~\cite{processes} and called the Benjamini-Schramm metric~\cite{BeSc01}.  
		Also, on the class of discrete subsets of $\mathbb R^d$, its topology
		extends the classical one in the context of point processes, which is that
		of vague convergence (see e.g., Appendix~A of~\cite{bookDaVe03I}).
	\end{remark}

	{Now, the precise definition of the set $\widehat{\mathcal D}_*$ is provided. Here, multiplicity of points is represented by the notion of pseudo metrics. Recall that} 
	a set $D$ equipped with a function $d:D\times D\to \mathbb R^{\geq 0}$ is called 
	a \defstyle{pseudo metric space}, if $d$ has the properties of a metric except that, $d(x,y)=0$ 
	does not necessarily imply $x=y$.\del{ As before, the pseudo metric is always denoted by $d$ except when explicitly mentioned.}
	The balls $N_r(v)$ are defined in the usual way. In this paper, 
	it is always assumed that the pseudo metric is \defstyle{boundedly finite}; i.e.,
	every subset of $D$
	included in a ball of finite radius is finite. For the sake of simplicity, 
	the term \defstyle{discrete pseudo metric space} (abbreviated as \texttt{DPMS})
	will be used to refer to boundedly finite pseudo metric spaces.
	\del{Note that $\{(u,v)\in D: d(u,v)=0 \}$ is an equivalence relation on $D$.
		By the assumption of boundedly finiteness, each equivalence class is finite.
		One can regard an equivalence class with $n$ elements as a point with multiplicity $n$.}
	\Rooted{} \texttt{DPMS}s and isomorphisms are defined in the usual way\del{ (being injective
		is important for pseudo metric spaces in this definition)}. Let $\dstarhat$ be the set of
	equivalence classes of \rooted{} \texttt{DPMS}s under isomorphism. Note that $\dstar\subseteq\dstarhat$.\del{ Moreover, it is easy to see that $\dstar$ is dense in $\dstarhat$.}
	Now, $r$-embeddability, $r$-similarity and the function $\kappa$ can be defined in exactly the
	same way as in Definitions~\ref{def:r-embedding} and~\ref{def:kappa}. 

{
	Now, the proofs of the above theorems are presented. The proofs are based on the following lemmas.}
	
	\begin{lemma}
		\label{lem:metric}
		The function $\kappa$ is a metric on $\dstarhat$.
	\end{lemma}
	
	\begin{proof}
		The definition readily implies that $\kappa$ is well defined on $\dstarhat$ and is symmetric.
		Also, it is clear that for all $[D,o]\in\dstar$, one has $\kappa([D,o],[D,o])=0$.
		
		Conversely, assume $(D,o)$ and $(D',o')$ are \rooted{} \texttt{DPMS}s
		satisfying the condition $\kappa((D,o),(D',o'))=0$. 
		The first claim is that for all $r>0$, there is a \rooted{}-isomorphism between $\nei ro$ and $\nei r {o^\prime}$.
		Let $r>1$ be given. Let $s\geq r$ be large enough, which will be determined later.
		Since $\kappa((D,o), (D',o'))=0<\frac 1 s$, the definition of $\kappa$ implies that
		there is an injective function $f:N_r(o)\rightarrow D'$ such that $f(o)=o'$ and 
		the distortion of $f$ is at most $\frac 1 s\leq \frac  1 r$.  By~\eqref{eq:distortion},
		the image of $f$ is contained in $\nei{r+1/s}{o'}$. Assume $s$ is large enough to ensure 
		that $\nei{r+1/s}{o'}=\nei{r}{o'}$ (which is possible since $\nei{r}{o'}$ is a closed ball). 
		So the range of $f$ is contained in $\nei{r}{o'}$. Being injective implies that
		$\card{\nei{r}{o}}\leq \card{\nei{r}{o'}}$. By switching the roles of the two spaces,
		one similarly proves the other direction of the inequality;
		Hence $\card{\nei{r}{o}} = \card{\nei{r}{o'}}$. This implies that
		$f:\nei{r}{o}\to\nei{r}{o'}$ is a bijection. {Since the distortion is at most $\frac 1 s$ and the sets are discrete, by choosing $s$ large enough, $f$ is necessarily an isometry between $N_r(o)$ and $N_r(o')$. So the above claim is proved.} \del{Now consider the union $A$ of
			the sets $\{d(u,v): u,v\in \nei{r}{o} \}$ and $\{d(u,v): u,v\in \nei{r}{o'} \}$.
			Let $\epsilon$ be the minimum distance of the pairs of distinct numbers in $A$.
			From the beginning, assume $s>\frac 1 \epsilon$. Now, \eqref{eq:distortion} implies
			that $d(x,y)=d(f(x),f(y))$ for all $x,y\in\nei{r}{o}$.
			In other words, $f$ is an isometry between $\nei{r}{o}$ and $\nei{r}{o'}$ and the above claim is proved.}
		
		For each integer $n\geq 1$, let $A_n$ be the {finite} set of \rooted{}-isomorphisms from $N_n(o)$ to $N_n(o')$
		which is already shown to be nonempty.\del{ Since each ball is finite, $A_n$ is also finite.}
		It is clear that for $n\geq 2$ and $f\in A_n$, the restriction of $f$ to $\nei{n-1}{o}$ belongs to $A_{n-1}$.
		Therefore, by K\"onig's infinity lemma, there is a sequence of isomorphisms $f_n\in A_n$
		such that the restriction of $f_n$ to $\nei{n-1}{o}$ is equal to $f_{n-1}$ for each $n$.
		Now, one can safely define $\rho(v):=\lim f_n(v)$ for each $v\in D$.
		It is easy to see that $\rho$ is an isomorphism between $(D,o)$ and $(D',o')$. So $[D,o]=[D',o']$.
		
		It remains to prove the triangle inequality for $\kappa$. Let $(D_i,o_i), 1\leq i\leq 3$ be \rooted{} \texttt{DPMS}s.
		Let $\kappa_{ij}:=\kappa((D_i,o_i),(D_j,o_j))$.
		One has to prove $\kappa_{13}\leq\kappa_{12}+\kappa_{23}$. If $\kappa_{12}+\kappa_{23}\geq 1$, the claim is clear.
		So assume $\kappa_{12}+\kappa_{23}<1$. Let $\epsilon>\kappa_{12}$ and $\delta>\kappa_{23}$
		be arbitrary such that $\epsilon+\delta<1$. Below, it is proved that
		$\kappa_{13}\leq \epsilon+\delta$. Since $\epsilon$ and $\delta$ are arbitrary, the claim follows.
		
		Since $\kappa_{12}<1$, $\kappa_{12}<\epsilon$ and $\epsilon+\delta> \epsilon$, 
		by Definition~\ref{def:kappa}, there is an injective function $f:\nei{1/(\epsilon+\delta)}{o_1}\to D_2$
		with distortion at most $\epsilon$ such that $f(o_1)=o_2$. Similarly, there is an injective 
		function $g:\nei{1/\delta}{o_2}\to D_3$ with distortion at most $\delta$ such that $f(o_2)=o_3$.
		The image of $f$ is contained in $\nei{1/(\epsilon+\delta)+\epsilon}{o_2}$.
		It is straightforward that $\epsilon+\delta<1$ implies  $1/(\epsilon+\delta)+\epsilon< 1/\delta$\del{ 
			(use the mean value theorem for the function $1/x$ in the interval $(0,1)$)}.
		Therefore, $g\circ f:N_{1/(\epsilon+\delta)}(o_1)\to D_3$ is well defined.
		By the definition of distortion in~\eqref{eq:distortion}, one readily gets that the distortion of
		$g\circ f$ is at most $\epsilon+\delta$. This means that $(D_1,o_1)$ is $(\epsilon+\delta)^{-1}$-embeddable
		in $(D_3,o_3)$. By swapping the roles of the two spaces, one gets that they are
		$(\epsilon+\delta)^{-1}$-similar. This means that $\kappa_{13}\leq \epsilon+\delta$ and the claim is proved.
	\end{proof}

	\begin{lemma}
		\label{lem:complete}
		The metric space $\dstarhat$ is complete {and separable}.
	\end{lemma}
	\begin{proof}
		{It is straightforward to see that every element of $\dstar$ can be approximated by a sequence of finite pointed \texttt{DPMS}s for which the pseudo metrics are rational valued. This implies the {separability of $\dstarhat$}.}
		{To prove completeness,} assume $\{[D_n,o_n]\}_{n=1}^{\infty}$ is a Cauchy sequence in $\dstarhat$ under
		the metric $\kappa$. One needs to show that it {has a convergent subsequence}. 
		
		The first claim is that, for all given $r>0$, $\{\card{\nei{r}{o_n}}\}_{n=1}^{\infty}$ is bounded. Assume this is not the case. So for each $m$, there is $n>m$ such that $\card{\nei{r}{o_n}}>\card{\nei{r+1}{o_m}}$. The reader can verify that this implies that $\kappa((D_n,o_n),(D_m,o_m))\geq \frac 1r$, which contradicts being a Cauchy sequence. So this claim is proved.
		
		The second claim is that for $r>0$ given, the sequence of balls $\{N_r(o_n)\}_{n=1}^{\infty}$ has a convergent subsequence under the metric $\kappa$ (it should be noted that the whole sequence is not necessarily convergent). By the above argument and passing to a subsequence, one may assume $\card{\nei{r}{o_n}}=l$ for each $n$, where $l$ is constant. For each $n$, consider an arbitrary order on the points of $\nei{r}{o_n}$ but let $o_n$ be the first in this order. Let $A^{(n)}$ be the distance matrix of $\nei{r}{o_n}$ (i.e., $A^{(n)}_{i,j}$ is the distance from the $i$-th point to the $j$-th point). Note that the entries of these matrices are in $[0,2r]$. Therefore, one can find a convergent subsequence of these matrices, say converging to matrix $A$. It is easy to see that $(i,j)\mapsto A_{i,j}$ is a pseudo metric on $\{1,\ldots,l\}$. Let $D'_r$ be the resulting \texttt{DPMS} with \rooot{} $o'_r:=1$. Now, it is easy to see that $\kappa(\nei{r}{o_n}, (D'_r,o'_r))\rightarrow 0$ as $n\rightarrow \infty$ and the second claim is proved.
		
		Third, by a diagonal argument using the second claim, one can assume that for each $m\in \mathbb N$, $\lim_n \nei{m}{o_n}$ exists (in $\dstarhat$), say $(D'_m,o'_m)$. Fix $m$ and let $\epsilon>0$ be arbitrary. So for large enough $n$, one has $\kappa(\nei{m}{o_n}, (D'_m,o'_m))<\epsilon$ and $\kappa(\nei{m-1}{o_n}, (D'_{m-1},o'_{m-1}))<\epsilon$. If $\epsilon<1/m$, there exist injective functions $f:D'_m \to N_m(o_n)$ and $g:N_{m-1}(o_n) \to D'_{m-1}$ with distortion less than $\epsilon$ and respecting the \rooot{}s. Therefore, $g \circ f$ is well-defined on $\nei{m-1-\epsilon}{o'_m}$, is injective and has distortion less than $2\epsilon$. By letting $\epsilon$ tend to zero while $m$ is fixed, one finds an isometric embedding of $\oball{m-1}{o'_m}$ into $D'_{m-1}$, where $\oball{r}{o}$ is the \textit{open ball} $\{u\in D: d(o,u)<r\}$.
		By considering $\epsilon$-embeddings in the other side, one also finds an isometric embedding of $\oball{m-1}{o'_{m-1}}$ into $D'_m$. Therefore, $\oball{m-1}{o'_m}$ is isomorphic to $\oball{m-1}{o'_{m-1}}$. It follows that the sequence $(D'_m,o'_m)$ of \rooted{} \texttt{DPMS}s can be paste together to form a pseudo metric space, namely $(D,o)$, which is discrete and boundedly finite. Also, $\oball{m}{o}$ is isometric to $\oball{m}{o'_m}$ for each $m$. It follows easily that $\kappa((D_n,o_n), (D,o))\rightarrow 0$. In other words, $[D_n,o_n]\rightarrow [D,o]$.
		\del{
			Finally, note that a specific subsequence is taken in the beginning of the third step and its convergence is proved. Being a Cauchy sequence implies the convergence of the whole sequence. So the claim is proved.}
	\end{proof}
	
	\begin{proof}[Proof of Theorem~\ref{thm:polish}]
		It is proved in the above lemmas that $\dstarhat$ is a complete separable metric space. So it remains to prove that $\dstar$ is a Borel subset of $\dstarhat$.
		For integers $m,n\geq 1$, let $A_{m,n}$ be the set of elements $[D,o]\in \dstarhat$ such that for all distinct pairs $x,y\in \oball{m}{o}$ (where $\oball{m}{o}$ is the {open ball} {defined above}), one has $d(x,y)\geq \frac 1n$
		It is not hard to show that $A_{m,n}$ is a closed subset of $\dstarhat$ and also
		$\dstar = \cap_m \cup_n A_{m,n}$. This proves the claim.
	\end{proof}

	\section{Tightness for Equivariant Processes}
	\label{ap:tightness}

	The following proposition is a converse to Lemma~\ref{lem:equivProcess}. Its proof is moved to Appendix~\ref{ap:lemmas}.
	
	\begin{proposition}
		\label{prop:equivProcessConverse}
		Let $[\bs D, \bs o]$ be a unimodular discrete space.
		If $[\bs D', \bs o'; \bs m']$ is a unimodular marked discrete space such that 
		$[\bs D', \bs o']$ (obtained by forgetting the marks)
		has the same distribution as $[\bs D, \bs o]$, then there is an equivariant process
		$\bs Z$ such that $[\bs D, \bs o; \bs Z_{\bs D}]$ has the same distribution as $[\bs D', \bs o'; \bs m']$.
	\end{proposition}

	\begin{remark}
		\label{rem:equivProcessEquivalence}
		Lemma~\ref{lem:equivProcess} and Proposition~\ref{prop:equivProcessConverse} imply that equivariant processes on $\bs D$ are equivalent to unimodular marked discrete spaces
		$[\bs D', \bs o'; \bs m']$ such that $[\bs D', \bs o']$ has the same distribution as $[\bs D, \bs o]$. 
		This enables one to define weak convergence of equivariant processes on {$\bs D$} 
		as probability measures on $\mathcal D'_*$. 
	\end{remark}
	
	\begin{lemma}
		\label{lem:equivariant-tight}
		Let $[\bs D, \bs o]$ be a unimodular discrete space. If $\Xi$ is a compact metric space,
		then the set of $\Xi$-valued equivariant processes on {$\bs D$} is tight and compact under weak convergence
		(see Remark~\ref{rem:equivProcessEquivalence}).
	\end{lemma}
	
	\begin{proof}
		Let $M$ be the set of unimodular marked discrete spaces $[\bs D', \bs o'; \bs m']$ with marks
		in $\Xi$ such that $[\bs D', \bs o']$ has the same distribution as $[\bs D, \bs o]$.
		By Remark~\ref{rem:equivProcessEquivalence}, it is enough to prove that $M$ is tight and compact.
		It is easy to see that $M$ is closed (under weak convergence). So it is enough to show that it is tight.
		Let $\epsilon>0$ and $\pi:\mathcal D'_*\rightarrow\mathcal D_*$ be the projection of forgetting the marks.
		By Prokhorov's theorem, \del{a single probability measure on $\mathcal D_*$ is tight. So }there is a compact
		set $K\subseteq \mathcal D_*$ such that $\myprob{[\bs D, \bs o]\in K}>1-\epsilon$.
		So for any equivariant process $[\bs D', \bs o';\bs m']$ on $[\bs D, \bs o]$ with values in $\Xi$,
		$\myprob{[\bs D', \bs o';\bs m']\in \pi^{-1}(K)}>1-\epsilon$.
		It is shown below that $\pi^{-1}(K)$ is compact. This implies that $M$ is tight and the claim is proved.
		
		To prove compactness of $\pi^{-1}(K)$, let $[D_n,o_n;m_n]\in\pi^{-1}(K)$ be an arbitrary
		sequence. By compactness of $K$,\del{ One has $[D_n,o_n]\in K$. So the latter has a convergent subsequence.
			Thus, from the beginning,} one may assume $[D_n,o_n]$ is convergent. Let $r>0$ be given.
		According to the proof of Lemma~\ref{lem:complete}, the sequence $\card{N_r(o_n)}$ is bounded.
		Now, the proof of the claim that $[D_n,o_n;m_n]$ has a convergent subsequence is similar to that of
		Lemma~\ref{lem:complete} and is left to the reader
		(one should first show that the sequence of balls $[N_r(o_n),o_n;m_n]$ has a convergent subsequence and then,
		deduce the claim by a diagonal argument).
	\end{proof}
	
	{This tightness criterion implies the existence of optimal coverings as follows.}
	
	\begin{proof}[Proof of Theorem~\ref{thm:optimalCovering}]
		Let $\bs S_1,\bs S_2,\ldots$ be a sequence of $r$-coverings of $\bs D$ 
		such that $\myprob{\bs o\in \bs S_n}\rightarrow \lambda_r$. 
		By Lemma~\ref{lem:equivariant-tight} and
		choosing a subsequence if necessary, one may assume from the beginning that the equivariant subsets $\bs S_n$ converge weakly
		to an equivariant subset $\bs S$ of $\bs D$. Since each $\bs S_n$ is an $r$-covering,
		$\myprob{\bs S_n\cap N_r(\bs o)=\emptyset} = 0$. {It is straightforward to deduce} $\myprob{\bs S\cap N_r(\bs o)=\emptyset}=0$\del{By the assumption of weak convergence, one can obtain
			$\myprob{\bs S\cap N_r(\bs o)=\emptyset}=0$ (let $\epsilon>0$ be arbitrary and $h:\mathbb R^{\geq 0}\to [0,1]$
			be a continuous  function that is identical to one on $[0,r]$ and zero on $[r+\epsilon,\infty)$.
			It can be seen that $h'(D,o;S):= 1\wedge \sum_{v\in S_D} h(d(o, v))$ is a continuous bounded function
			on $\mathcal D'_*$. By weak convergence, one gets $\omid{h'(\bs D, \bs o;\bs S)} = \lim_n \omid{h'(\bs D, \bs o;\bs S_n)} = 1$.
			Now, the claim follows by letting $\epsilon$ tend to zero)}.
		So by putting balls of radius $r$ on the points
		of $\bs S$, $\bs o$ is covered a.s. So Lemma~\ref{lem:happensAtRoot} implies that every point is covered a.s.; i.e., $\bs S$
		is an $r$-covering. Also, by weak convergence, $\myprob{\bs o\in \bs S}= \lim_n \myprob{\bs o\in \bs S_n} = \lambda_r$.
		This implies that $\bs S$ is an optimal $r$-covering.
	\end{proof}

	Finally, the proof of Proposition~\ref{prop:equivProcessConverse} is presented.
	\begin{proof}[Proof of Proposition~\ref{prop:equivProcessConverse}]
		\del{The reader can skip the proof at first reading.
			Here is a sketch. The random mark functions are obtained by the disintegration theorem for the natural map $\pi:\mathcal D'_*\rightarrow\mathcal D_*$ (some care needs to be taken since only equivalence classes of discrete spaces are considered). The harder part is to use the crucial assumption of unimodularity to deduce that the distribution of the marks do not depend on the \rooot{}. 
			This is similar to the \textit{invariant disintegration theorem}~\cite{disintegration}. To reduce it to the invariant disintegration theorem, an action of a countable group is needed. The latter is given by a result of Feldman and Moore~\cite{FeMo77} as discussed below. This theorem is in the context of Borel equivalence relations, which we refrain from introducing here. }%
		The claim is proved by \textit{invariant disintegration} applied to the action of a group constructed in~\cite{FeMo77}. Let $\pi:\dstar'\to\dstar$ be defined by $\pi[D,o;m]:=[D,o]$. First, assume that $\bs D$ has no automorphisms other than the identity a.s. Let $\mathfrak D$ be the class of discrete spaces with no nontrivial automorphisms. 
		Consider the equivalence relations $R$ and $R'$ on $\dstar$ and $\dstar'$ respectively, defined by:
		\begin{equation*}
			\label{eq:equivalence}
			\forall D\in\mathfrak D: \forall u,v\in D: {[D,u]\sim_R[D,v] \text{ and }} [D,u;m]\sim_{R'}[D,v;m].
		\end{equation*}
		Here, for all discrete spaces $D\not\in\mathfrak D$, the elements $[D,u]$ and $[D,u;m]$ are not equivalent to any other elements. It can be seen that $R$ and $R'$ are Borel equivalence relations (see~\cite{FeMo77} for the definition).
		By Theorem~1 of~\cite{FeMo77}, there is a countable group $H$ acting measurably on $\dstar$ such that for every $[D,o]\in\dstar$, its $R$-equivalence class is just $\{h\cdot [D,o]: h\in H \}$. 
		Note that when $D\in\mathfrak D$, the $R$-equivalence class of $[D,o]$ (resp. $[D,o;m]$) is countable and has a natural bijection with $D$. If so, the action of every $h\in H$ induces a natural map $h_{D}:D\rightarrow D$ and this map is bijective. Hence, $H$ acts on $\dstar'$ as well by letting $h\cdot[D,o;m]:=[D,h_D(o);m]$ when $D\in\mathfrak D$ and $h\cdot[D,o;m]:=[D,o;m]$ when $D\not\in\mathfrak D$. Note that this action generates the equivalence relation $R'$.
		
		Since $\bs D$ has no nontrivial automorphisms a.s., so does $\bs D'$. So, given $h\in H$, the map $h_{\bs D'}$ is bijective a.s. 
		Therefore, by using the mass transport principle~\eqref{eq:unimodularMarked}, one can show that $h\cdot [\bs D', \bs o';\bs m']$ has the same
		distribution as $[\bs D', \bs o'; \bs m']$ (this is analogous to Mecke's theorem. See also~\cite{ThLa09} 
		or Proposition~3.6 of~\cite{eft}). In addition, the joint distribution of	$([\bs D', \bs o'; \bs m'], [\bs D', \bs o'])$ is invariant under the diagonal action of the group $H$.
		Therefore, the invariant disintegration theorem~\cite{disintegration} gives a kernel $t$ from $\mathcal D_*$
		to $\mathcal D'_*$ that is invariant under the diagonal action of $H$ and such that $t([D,o],\cdot)$ is supported on $\pi^{-1}[D,o]$ and $t$ pushes the distribution of $[\bs D', \bs o']$ to the distribution
		of $[\bs D', \bs o'; \bs m']$. For every deterministic $D\in\mathfrak D$ and $o\in D$, the probability measure $t([D,o],\cdot)$ gives a random
		element $\bs Z_{(D,o)}\in \Xi^{D\times D}$.
		It can be shown from the invariance of $t$ that $\bs Z_{(D,h\cdot o)}$ has the same distribution as
		$\bs Z_{(D,o)}$ for every $h\in H$. Hence, $\bs Z_{(D,v)}$ has the same distribution as
		$\bs Z_{(D,o)}$ for every $v\in D$. So one can finally write $\bs Z_{D}$ instead of
		$\bs Z_{(D,o)}$. It can be seen that $\bs Z$ satisfies the claim.
		
		Now, consider the general case when $\bs D$ may have nontrivial automorphisms. 
		Let $\bs U$ be the equivariant process obtained by adding i.i.d. marks to the points with the uniform distribution on $[0,1]$.
		Then, $[\bs D, \bs o; \bs U]$ is unimodular and has no nontrivial automorphisms a.s.
		Also, by Remark~\ref{rem:equivProcessExtraMarks}, $[\bs D', \bs o'; (\bs U, \bs m')]$
		is an equivariant process on $[\bs D, \bs o; \bs U]$. 
		Now, one can repeat the above arguments line by line, which results in a random element
		in $\Xi^{D\times D}$ for every $D$ that is equipped with distinct marks in $[0,1]$.
		By regarding the latter marks as an additional source of randomness, this gives the desired $\bs Z_{D}$ for non-marked $D$, which satisfies the claim.	
	\end{proof}

	\section{Proofs and Auxiliary Lemmas}
	\label{ap:lemmas}
	
		Some of the proofs and lemmas are gathered in this appendix to prevent distraction from the main thread of the paper.

	\begin{proof}[Proof of Lemma~\ref{lem:equivProcess}]
		In what follows, \eqref{eq:equivProcessDistribution} is written in the forms $\mathcal Q(A) = \myprob{[\bs D, \bs o; \bs Z_{\bs D}]\in A}$ or
		$
		\mathcal Q(A)= \omid{\int \identity{A}[\bs D, \bs o; m] {\mathrm d}\mathcal P_{\bs D}(m)}
		$	
		by keeping in mind that a realization of $[\bs D, \bs o]$ is considered in the term inside the expectation.
		To prove unimodularity, let $f:\mathcal D'_{**}\rightarrow\mathbb R^{\geq 0}$ be a measurable function.
		For all deterministic discrete spaces $D$ and $x,y\in D$, let 
		\[
		g(D,x,y):=\omid{f(D,x,y;\bs Z_{D})} = \int f(D,x,y;m) {\mathrm d}\mathcal P_D(m).
		\]
		One has
		\begin{eqnarray*}
			\omid{\sum_{x\in \bs D} f[\bs D, \bs o, x; \bs Z_{\bs D}] }
			&=& \omid{\int \sum_{x\in \bs D} f[\bs D, \bs o, x; m] {\mathrm d}\mathcal P_{\bs D}(m) }\\
			&=& \omid{\sum_{x\in \bs D} g[\bs D, \bs o, x]}.
		\end{eqnarray*}
		One can similarly obtain that 
		$\omid{\sum_{x\in \bs D} f[\bs D,x, \bs o; \bs Z_{\bs D}] } = \omid{\sum_{x\in \bs D} g[\bs D, x, \bs o]}$. 
		Therefore, the claim follows by~\eqref{eq:unimodular} for $g$. 
	\end{proof}

\begin{lemma}
	\label{lem:metricChange-measurability}
	Let $\mathcal N\subseteq \mathcal D'_*$ be the set of \rooted{} marked discrete spaces $[(D,d),o;d']$ with mark space $\mathbb R$ such that $d'$ is a boundedly-finite metric on $D$. Then, $\mathcal N$ is a Borel subset of $\mathcal D'_*$ and the map $\rho:\mathcal N\rightarrow\mathcal N$ defined by $\rho[(D,d),o;d']:=[(D,d'),o;d]$ is Borel measurable.
\end{lemma}
\begin{proof}[Proof of Lemma~\ref{lem:metricChange-measurability}]
	The proof of the first claim is similar to that of Theorem~\ref{thm:polish}: For $n\in\mathbb N$ and $\epsilon>0$, let $\mathcal N_{n,\epsilon}$ be the set of $[(D,d),o;d']\in\mathcal D'_*$ such that the restriction of $d'$ to $\oball{n}{o}$ is a metric and $\forall u\neq v\in \oball{n}{o}: d'(u,v)\geq \epsilon$. It is not hard to see that $\mathcal N_{n,\epsilon}$ is closed in $\mathcal D'_*$. Also, one has $\mathcal N = \cap_n \cup_m \mathcal N_{n,1/m}$, which implies that $\mathcal N$ is Borel subset of $\mathcal D'_*$.
	
	For the second claim, it is enough to prove that the inverse image of any open ball in $\mathcal N$ under $\rho$ is measurable. Let $\xi_0:=[(D_0,d_0),o_0;d'_0]\in N$ and consider the open ball $U:=B_{\epsilon}{(\xi_0)}:=\{\xi\in \mathcal N: \kappa(\xi,\xi_0)<\epsilon \}$ in $\mathcal N$. 
	Let $r>1/\epsilon$ be an arbitrary rational number and $M\in\mathbb N$. Let $A_{r,M}$ be the set of $[(D,d),o;d']\in\mathcal N$ such that $\rho(\oball{M}{o})$ and $\xi_0$ are $r$-similar (Definition~\ref{def:r-embedding}), where $\oball{M}{o}$ is equipped with the \rooot{} and metrics induced from $[(D,d),o;d']$. It is not hard to show that $A_{r,M}$ is a closed subset of $\mathcal N$. Also, one has $\rho^{-1}(U) = \cup_{r>1/\epsilon} \cap_M A_{r,M}$. This implies that $\rho^{-1}(U)$ is a Borel subset of $\mathcal N$ and the claim is proved.
\end{proof}
	
	\begin{lemma}
		\label{lem:canopy-point-stationary}
		In the setting of Subsection~\ref{subsec:canopyGeneralized}, $\Phi-\bs o$ is a point-stationary point process.
	\end{lemma}
	\begin{proof}
		Let $g=g(\varphi,x,y)$ be a (measurable) function that assigns a real number to every discrete
		subset $\varphi\subseteq\mathbb R^2$ and each $x,y\in\varphi$ which is invariant under (joint) translations of
		$\varphi,x$ and $y$. The claim is that
		\begin{equation}
			\label{eq:lem:canopy:1}
			\omid{\sum_{x\in \Phi} g(\Phi,\bs o, x)} = \omid{\sum_{x\in \Phi} g(\Phi,x, \bs o)}.
		\end{equation}
		By additivity, it is enough to assume that $g(\Phi,x,y)$ is zero except when $x\in \Phi_i$ 
		and $y\in \Phi_j$ for some fixed $i$ and $j$. So \eqref{eq:lem:canopy:1} is equivalent to
		\begin{equation}
			\label{eq:lem:canopy:2}
			p_i \omid{\sum_{x\in \Phi_j} g(\Phi,\bs o_i, x)} = p_j \omid{\sum_{x\in \Phi_i} g(\Phi,x, \bs o_j)}.
		\end{equation}
		Define $h:\mathbb Z^2\to\mathbb R$ by $h(k,l):= \omid{\sum \sum g(\Phi,x,y)}$,
		where the sum is over all pairs of points $x,y\in\Phi$ such that 
		\[
		x\in [k, k+1)\times\{i\}, \quad y\in [l, l+1)\times\{j\}.
		\]
		It can be seen that \eqref{eq:lem:canopy:2} can be written as
		\begin{equation}
			\label{eq:lem:canopy:3}
			\sum_{l\in\mathbb Z}h(0,l) = \sum_{k\in\mathbb Z}h(k,0).
		\end{equation}
		Note that the coefficients $p_i$ and $p_j$ disappear in this formula
		since one has $\myprob{\bs o_i\in [0,1)\times \{i\}}=p_i$.
		Now, the invariance of $g$ under translations and the stationarity of $\Phi$ under horizontal
		translations imply that $h(0,k)=h(-k,0)$. This proves~\eqref{eq:lem:canopy:3}.
		So \eqref{eq:lem:canopy:1} is also proved and hence, $\Phi-\bs o$ is a point-stationary point process.
		\del{Therefore, by Example~\ref{ex:point-stationary}, $[\Phi,\bs o]$ is a unimodular discrete space.}	
	\end{proof}

	\begin{lemma}
		\label{lem:logbound}
		Let ${(X_n)_{n=1}^{\infty}}\geq 0$ be a monotone sequence of random variables. Then almost surely, $\growthu{X_n} \leq  \growthu{\omid{X_n}}$.
		\del{\begin{eqnarray}
				\label{eq:lem:logbound:1}
				\growthu{X_n} &\leq & \growthu{\omid{X_n}}.
				\label{eq:lem:logbound:2}
				\growthl{X_n} &\leq & \growthl{\omid{X_n}}.
			\end{eqnarray}
		}
		\del{Moreover,
			if $\sum_n var(X_n)/\omid{X_n}^2 < \infty$, then 
			\[
			\growthl{X_n} = \growthl{\omid{X_n}}.
			\]}
	\end{lemma}
	{One can also deduce that $\growthl{X_n} \leq \growthl{\omid{X_n}}$, but this is skipped since it is not needed here.}
	\begin{proof}
		The claim will be proved assuming $0\leq X_1\leq X_2\leq\cdots$.
		The non-increasing case can be proved with minor changes.
		Let $\alpha$ and $\beta$ be arbitrary such that
		$\growthu{\omid{X_n}} <\beta<\alpha$.
		So there is a constant $c$ such that  $\omid{X_n}\leq cn^{\beta}$ for all $n\geq 1$.
		Let $M:=\max\{n:X_n>n^{\alpha}\}$, with the convention $\max\emptyset:=0$.
		Below, it will be shown that $M<\infty$ a.s. Assuming this, it follows
		that $\growthu{X_n}\leq \alpha$ a.s.
		By considering this for all $\alpha$ and $\beta$, the claim is implied.
		
		Now, it is proved that $M<\infty$ a.s. With an abuse of notation, the constant $c$ below
		is updated in each step without changing the symbol.
		\begin{eqnarray*}
			\myprob{M\geq n}& =& \myprob{\exists {k\geq n}: X_k> k^{\alpha}}
			\leq  \sum_{j=0}^{\infty} \myprob{\exists k:  {n2^j\leq k\leq n2^{j+1}}, X_k> k^{\alpha}}\\	
			&\leq & \sum_{j=0}^{\infty} \myprob{X_{n2^{j+1}}> (n2^j)^{\alpha}}	
			\leq  \sum_{j=0}^{\infty} \frac{\omid{X_{n2^{j+1}}}}{(n2^j)^{\alpha}} 
			\leq  \sum_{j=0}^{\infty} \frac{c(n2^{j+1})^{\beta}}{(n2^j)^{\alpha}}\\
			&\leq & \sum_{j=0}^{\infty} c(n2^j)^{{\beta}-\alpha}
			\leq  cn^{{\beta}-\alpha}.
		\end{eqnarray*}
		The RHS is arbitrarily small for large $n$. This implies that $M<\infty$ a.s. and the claim is proved.
	\end{proof}

	\begin{lemma}
		\label{lem:BaumKatz1}
		Let $X,X_1,X_2,\ldots$ be a non-negative i.i.d. sequence and $t>0$ be such that
		$\myprob{X>r}\geq cr^{-t}$ for large enough $r$. Let $S_n:=X_1+\cdots+X_n$. 
		Then there exists $C<\infty$ such that almost surely,
		\[
		\exists n: \forall k\geq n: S^{-1}(k) \leq Ck^{t} \log \log k. 
		\]
	\end{lemma}
	
	\begin{proof}
		First, one has
		\begin{eqnarray}
			\nonumber
			\myprob{ S^{-1}(n)\geq  m} &=& \myprob{S_{m}\leq n} \leq \myprob{\forall i\leq m: X_i\leq n} = \myprob{X\leq n}^m\\
			\label{eq:S_m<n} &\leq& (1-cn^{-t})^m \leq e^{-cmn^{-t}}.
		\end{eqnarray}
		Let $C:=2^{t+1}/c$ and $\psi(x):=Cx^t \log \log x$, 
		Therefore, for large $n$, one has 
		\begin{eqnarray*}
			\myprob{\exists k\geq n: S^{-1}(k) > \psi(k)} &=& \myprob{\max_{k\geq n} \frac{S^{-1}(k)}{\psi(k)}>1}\\
			&\hspace{-9cm}\leq & \hspace{-4.5cm} \sum_{j=0}^{\infty} \myprob{\max_{n2^j\leq k < n2^{j+1}} \frac{S^{-1}(k)}{\psi(k)}>1}
			\leq  \sum_{j=0}^{\infty} \myprob{S^{-1}(n2^{j+1}) > \psi(n2^j)}\\
			&\hspace{-9cm}\leq &\hspace{-4.5cm} \sum_{j=0}^{\infty} e^{-c\psi(n2^j)(n2^{j+1})^{-t} }
			\leq  \sum_{j=0}^{\infty} e^{-2\log \log (n2^j)}
			= \sum_{j=0}^{\infty} \frac 1{(j \log 2 + \log n)^{2}}.
		\end{eqnarray*}
		It is clear that the sum in the last term is convergent. Therefore,
		dominated convergence implies that the right hand side tends to zero as $n\rightarrow 0$. This proves the claim.
	\end{proof}
	
	\begin{lemma}
		\label{lem:regtree-weight}
		Let $\alpha<\infty$ and $(T,o)$ be a deterministic rooted tree such that
		$\mathrm{deg}(o)\geq 2$ and $\mathrm{deg}(v)\geq 3$ for all $v\neq o$.
		Let $d'$ be a metric on $T$ which is generated by \del{a function on the edges }{edge lengths} such that $d'(\cdot)\geq 1$.
		Let $w(u):=C\sum_{v \sim u} \bs d'(u,v)^{\alpha}$. Then $C=C(\alpha)$ can be chosen such that 
		for all $r\geq 0$, one has $w(N'_r(o))\geq r^{\alpha}$, {where $N'_r$ denotes the ball of radius $r$ under the metric $\bs d'$.}
		
	\end{lemma}
	
	\begin{proof}
		{Let $C$ be a constant such that $\forall x\in [0,1]: Cx^{\alpha} + (1-x)^{\alpha}\geq \frac 12$.
			It is easy to see that such $C$ exists.}
		For $r\geq 0$, let $f(r)$ be the infimum value of $w(N'_r(o))$ for all trees
		with the stated conditions. So one should prove $f(r)\geq r^{\alpha}$.
		The claim is true for $r=0$. Also, if $0<r<1$, one has $N'_r(o)=\{o\}$ and the claim is trivial. The proof uses
		induction on $\lfloor r\rfloor$. Assume that $r\geq 1$ and for all $s<\lfloor r\rfloor$, one has $f(s)\geq s^{\alpha}$.
		For $y\sim o$, let $T_y$ be the connected component containing $y$ when the edge $(o,y)$ is removed.
		It can be seen that $[T_y,y]$ satisfies the conditions of the lemma. Therefore,  
		\begin{eqnarray*}
			w(N'_r(o)) &=&  w(o) + \sum_{y: y\sim o}  w(N'_{r- d'(o,y)}(T_y, y))
			\geq 	 w(o) + \sum_{y: y\sim o} f(r- d'(o,y))\\
			&\geq & \sum_{y: y\sim o} \left[C  d'(o,y)^{\alpha} + (r- d'(o,y))^{\alpha}\right]\\
			&\geq & \mathrm{deg}(o)\cdot \min_{0\leq x \leq r}\{Cx^{\alpha} +  (r-x)^{\alpha}\}
			\geq  \mathrm{deg}(o) r^{\alpha}/2
			\geq  r^{\alpha},
		\end{eqnarray*}
		where  the third line is by the definition of $w(o)$ and the induction hypothesis, the fifth line is due to the definition of $C$,
		and the last line is by the assumption $\mathrm{deg}(o)\geq 2$.
		Hence $f(r)\geq r^{\alpha}$, which proves the induction claim.
	\end{proof}
	
	\begin{lemma}
		\label{lem:minimalCut}
		An equivariant cut-set is equivariantly minimal if and only if it is almost surely minimal.
	\end{lemma}
	\begin{proof}
		Let $\Pi$ be an equivariant cut-set. If $\Pi$ is almost surely minimal,
		then it is also equivariantly minimal by definition. Conversely, assume $\Pi$ is equivariantly minimal but not almost surely minimal.
		Call an edge $e'$ \textit{above} an edge $e$ if $e'$ separates $e$ from the end. Call an edge $e\in \Pi$ \textit{bad} if there is
		an edge of $\Pi$ above $e$. Let $\Pi'$ be the set of bad edges 
		of $\Pi$. Let $\Pi''$ be the set of lowest edges in $\Pi'$; i.e., the edges $e\in\Pi'$
		such that there is no other edge of $\Pi'$ below $e$. It can be seen that the assumption implies that $\Pi''$
		is nonempty with positive probability. Now, it can be seen that $\Pi\setminus\Pi''$ is an equivariant cut-set, which contradicts the minimality of $\Pi$.
	\end{proof}

\end{appendices}

\section*{{Acknowledgements}}
%
\formir{Supported in part by a grant of the Simons Foundation (\#197982) to The University of Texas at Austin
and by the ERC NEMO grant, under the European Union's Horizon 2020 research and innovation programme,
grant agreement number 788851 to INRIA.
The second author thanks the Research and Technology Vice-presidency of Sharif University of Technology for its support. This research was done while the third author was affiliatd with Tarbiat Modares University and was in part supported by a grant from IPM (No.98490118). He is currently affiliated with INRIA Paris.
}

\bibliography{bib} 
\bibliographystyle{plain}
\end{document}